\documentclass[a4paper, oneside, 11pt]{amsart}
\usepackage{import}

\title[Twisted representations of conformal nets]{Twisted representations of conformal nets and crossed balanced tensor categories}
\author{Adrià Marín-Salvador}
\date{}

\newcommand{\Z}{\mathbb{Z}}
\newcommand{\R}{\mathbb{R}}
\newcommand{\C}{\mathbb{C}}

\newcommand{\Diff}{\text{Diff}}

\newcommand{\Mob}{\mathbf{M\ddot{o}b}}
\newcommand{\A}{\mathcal{A}}
\newcommand{\Rot}{\text{Rot}}

\newcommand{\id}{\text{id}}

\newcommand{\Hom}{\text{Hom}}
\newcommand{\Aut}{\text{Aut}}
\newcommand{\Rep}{\text{Rep}}
\newcommand{\colim}{\text{colim}}

\newcommand{\Dcal}{\mathcal{D}}

\newcommand{\Ccal}{\mathcal{C}}
\newcommand{\End}{\text{End}}

\newcommand{\Jcal}{\mathcal{J}}
\newcommand{\ConfR}{\widetilde{\text{Conf}}_{2}(S^1)}

\newcommand{\B}{\mathbb{B}}
\newcommand{\Vect}{\text{Vect}}
\newcommand{\Loc}{\text{Loc}}

\newcommand{\1}{\mathbf{1}}

\usepackage{amssymb}
\usepackage{amsfonts}
\usepackage{amsmath}
\usepackage{quiver}
\usepackage{amsthm}
\usepackage{comment}
\usepackage[english]{babel}
\usepackage{quiver}
\usepackage{nicefrac}
\usepackage[a4paper]{geometry}
\usepackage{amssymb}
\usepackage{mathtools}
\usepackage{stackrel}
\usepackage[hidelinks]{hyperref}
\usepackage{mathrsfs} 
\usetikzlibrary{decorations.pathmorphing}
\usetikzlibrary{calc}
\usetikzlibrary{decorations.markings}
\usetikzlibrary{fadings,decorations.pathreplacing}
\usetikzlibrary{matrix,arrows}
\usetikzlibrary{arrows,calc,decorations.pathreplacing,decorations.markings,shapes.geometric,shadows}

\usepackage{xcolor}

\tikzset{-dot-/.style={decoration={
  markings,
  mark=at position 0.5 with {\fill circle (2pt);}},postaction={decorate}}}

\tikzset{
	Fdot/.style={circle, draw, fill, inner sep=0pt}, 
	Odot/.style={circle, draw, inner sep=0.1pt, minimum size=0.1cm}
	}

  \vfuzz \hfuzz
\usepackage{xcolor}
\usetikzlibrary{arrows}

\theoremstyle{definition}
\newtheorem{definition}{Definition}[section]
\newtheorem{proposition}[definition]{Proposition}
\newtheorem{lemma}[definition]{Lemma}
\newtheorem{theorem}[definition]{Theorem}
\newtheorem{corollary}[definition]{Corollary}

\newtheorem{remark}[definition]{Remark}

\newtheorem*{theorem*}{Theorem}
\newtheorem*{definition*}{Definition}
\newtheorem*{example*}{Example}
\newtheorem*{corollary*}{Corollary}
\newtheorem*{conjecture*}{Conjecture}

\definecolor{myviolet}{HTML}{7D3C98}
\definecolor{myorange}{HTML}{F39C12}
\definecolor{myblue}{HTML}{2E86C1}
\definecolor{mygreen}{HTML}{1E8449}
\definecolor{myred}{HTML}{C0392B}

\tikzstyle{ipe stylesheet} = [
  ipe import,
  even odd rule,
  line join=round,
  line cap=butt,
  ipe pen normal/.style={line width=0.4},
  ipe pen heavier/.style={line width=0.8},
  ipe pen fat/.style={line width=1.2},
  ipe pen ultrafat/.style={line width=2},
  ipe pen normal,
  ipe mark normal/.style={ipe mark scale=3},
  ipe mark large/.style={ipe mark scale=5},
  ipe mark small/.style={ipe mark scale=2},
  ipe mark tiny/.style={ipe mark scale=1.1},
  ipe mark normal,
  /pgf/arrow keys/.cd,
  ipe arrow normal/.style={scale=7},
  ipe arrow large/.style={scale=10},
  ipe arrow small/.style={scale=5},
  ipe arrow tiny/.style={scale=3},
  ipe arrow normal,
  /tikz/.cd,
  ipe arrows, 
  <->/.tip = ipe normal,
  ipe dash normal/.style={dash pattern=},
  ipe dash dotted/.style={dash pattern=on 1bp off 3bp},
  ipe dash dashed/.style={dash pattern=on 4bp off 4bp},
  ipe dash dash dotted/.style={dash pattern=on 4bp off 2bp on 1bp off 2bp},
  ipe dash dash dot dotted/.style={dash pattern=on 4bp off 2bp on 1bp off 2bp on 1bp off 2bp},
  ipe dash normal,
  ipe node/.append style={font=\normalsize},
  ipe stretch normal/.style={ipe node stretch=1},
  ipe stretch normal,
  ipe opacity 10/.style={opacity=0.1},
  ipe opacity 30/.style={opacity=0.3},
  ipe opacity 50/.style={opacity=0.5},
  ipe opacity 75/.style={opacity=0.75},
  ipe opacity opaque/.style={opacity=1},
  ipe opacity opaque,
]
\definecolor{red}{rgb}{1,0,0}
\definecolor{blue}{rgb}{0,0,1}
\definecolor{green}{rgb}{0,1,0}
\definecolor{yellow}{rgb}{1,1,0}
\definecolor{orange}{rgb}{1,0.647,0}
\definecolor{gold}{rgb}{1,0.843,0}
\definecolor{purple}{rgb}{0.627,0.125,0.941}
\definecolor{gray}{rgb}{0.745,0.745,0.745}
\definecolor{brown}{rgb}{0.647,0.165,0.165}
\definecolor{navy}{rgb}{0,0,0.502}
\definecolor{pink}{rgb}{1,0.753,0.796}
\definecolor{seagreen}{rgb}{0.18,0.545,0.341}
\definecolor{turquoise}{rgb}{0.251,0.878,0.816}
\definecolor{violet}{rgb}{0.933,0.51,0.933}
\definecolor{darkblue}{rgb}{0,0,0.545}
\definecolor{darkcyan}{rgb}{0,0.545,0.545}
\definecolor{darkgray}{rgb}{0.663,0.663,0.663}
\definecolor{darkgreen}{rgb}{0,0.392,0}
\definecolor{darkmagenta}{rgb}{0.545,0,0.545}
\definecolor{darkorange}{rgb}{1,0.549,0}
\definecolor{darkred}{rgb}{0.545,0,0}
\definecolor{lightblue}{rgb}{0.678,0.847,0.902}
\definecolor{lightcyan}{rgb}{0.878,1,1}
\definecolor{lightgray}{rgb}{0.827,0.827,0.827}
\definecolor{lightgreen}{rgb}{0.565,0.933,0.565}
\definecolor{lightyellow}{rgb}{1,1,0.878}
\definecolor{black}{rgb}{0,0,0}
\definecolor{white}{rgb}{1,1,1}

\usetikzlibrary{arrows.meta,patterns}
\usetikzlibrary{ipe}

\begin{document}

\begin{abstract}
    Let $\A$ be a (not necessarily rational) conformal net with an action of a discrete group $G$. We show that the category $\Rep^G(\A)$ of $G$-twisted representations of $\A$ is canonically a $G$-crossed balanced $\mathrm{W}^*$-tensor category. This extends the results of \cite{muger05}, in the language of localized endomorphisms, that $\Rep^G(\A)$ is a $G$-crossed braided tensor category.
\end{abstract}

\maketitle

\tableofcontents

\section{Introduction}
\addtocontents{toc}{\protect\setcounter{tocdepth}{-1}}

Conformal nets provide a mathematical formalization of unitary 2-dimensional chiral conformal field theory. Taking spacetime to be the 2-dimensional Minkowski space, or a partial compactification into a cylinder, one is required to provide, in the spirit of algebraic quantum field theory, a net of algebras on the cylinder which is compatible with the group of conformal diffeomorphisms of the cylinder. Whenever the net is independent of the positive light-ray coordinate, one talks about chiral nets. In this setting, it is natural to consider nets of von Neumann algebras of observables on $S^1$, seen as a Cauchy slice on the cylinder, in a way that is covariant with the group of diffeomorphisms of the circle. These are conformal nets \cite{bmt88, bsm90, bgl93, gf93, w95, kl04, bdh15}. More explicitly, a conformal net consists of a \textit{vacuum} Hilbert space $H_0$ equipped with a projective representation of the group $\Diff^+(S^1)$ of orientation-preserving diffeomorphisms of the circle together with, for every interval $I\subset S^1$, a von Neumann algebra $\A(I)\subset B(H_0)$ in such a way that the assignment $I\mapsto \A(I)$ is covariant with respect to the inclusion of intervals and compatible with the action of $\Diff^+(S^1).$ Given a conformal net $\A$, a representation of $\A$ consists of a Hilbert space $H$ with compatible $*$-actions of the von Neumann algebras $\A(I)$ for all intervals $I\subset S^1$. The category of representations of a conformal net $\A$, denoted $\Rep(\A)$, is a braided tensor category, see \cite[Sec. 2]{frs}, \cite[Sec. 7]{lon89} and \cite[Sec. IV.4]{gf93} for the definition of the Connes fusion of representations and the braiding, and \cite{bdh15, bicommutantfromnets, Gui21} for a more modern treatment. Actually, given $H\in \Rep(\A)$, the endomorphisms $\End_{\Rep(\A)}(H)$ of $H$ form a von Neumann algebra, making $\Rep(\A)$ into a braided $\mathrm{W}^*$-tensor category (see \cite{henriques2024completewcategories} for an introduction to $\mathrm{W}^*$-categories).

A lot of the literature on conformal nets focuses on the particular class of \textit{rational} conformal nets \cite{klm01}. The most relevant consequence of rationality for our discussion is that, when $\mathcal{A}$ is rational, the $\mathrm{W}^*$-category $\Rep(\A)$ has finitely many simple objects. When one drops rationality, $\Rep(\A)$ may have infinitely (or continuously) many simple objects, and a general object may be a direct integral over the space of simple objects, as opposed to a direct sum. In this paper, we do not assume rationality at any point.

\subsection*{Twisted representations of conformal nets} Given a conformal net $\A$, an automorphism of $\A$ consists of a collection of compatible automorphisms of the von Neumann algebras $\A(I)$ for all $I\subset S^1$, which are implemented unitarily in the vacuum representation $H_0$, see Section \ref{Sec: ConfNets}. We write $\Aut(\A)$ for the group of automorphisms of $\A$. Given an automorphism $\varphi\in \Aut(\A)$, one can consider the notion of a $\varphi$-twisted representation of $\A$, which consists of a Hilbert space $H$ with actions of the von Neumann algebras $\A(I)$, compatible with the inclusions $\A(J)\subset \A(I)$ up to the action of $\varphi$, see Section \ref{Sec: ConfNets} and \cite{muger05}. We write $\Rep^\varphi(\A)$ for the $\mathrm{W}^*$-category of $\varphi$-twisted representations of $\A$. The category $\Rep^\varphi(\A)$ is not in general a tensor category, but rather there exist tensor products
\[
\Rep^{\varphi}(\A)\times\Rep^{\sigma}(\A)\to \Rep^{\varphi\circ\sigma}(\A)
\]
for $\varphi, \sigma\in \Aut(\A)$ generalizing the Connes fusion of $\A$-representations. Given a discrete group $G$ acting faithfully on $\A$ by an injective group homomorphism $\Phi: G\to \Aut(\A)$, we can therefore consider the tensor category
\[
\Rep^G(\A):=\bigoplus_{g\in G} \Rep^{\Phi(g)}(\A).
\]
Twisted representations of conformal nets were introduced in \cite{muger05} in the language of localized endomorphisms. Müger defines the $\mathrm{W}^*$-tensor category $G-\Loc_{I_0}(\A)$ of $G$-localized endomorphisms of $\A$ in some interval $I_0\subset S^1$, generalizing DHR endomorphisms, and shows that it admits a canonical $G$-crossed braiding, which is a modification of the usual structure of a braiding in the presence of a $G$-grading and a compatible $G$-action. The $\mathrm{W}^*$-tensor categories $\Rep^G(\A)$ and $G-\Loc_{I_0}(\A)$ are equivalent, and hence $\Rep^G(\A)$ is also canonically a $G$-crossed braided $\mathrm{W}^*$-tensor category (an element $g\in G$ acts on $\Rep^G(\A)$ by precomposing the action of $\A$ on a twisted representation by the automorphism $\Phi(g)^{-1}\in \Aut(\A)$). In this paper, we give an explicit description of the $G$-crossed braided structure on $\Rep^G(\A)$. This allows us to define a \emph{$G$-crossed balance} (or $G$-crossed categorical twist) on $\Rep^G(\A)$. Given a $G$-crossed braided category $\Ccal = \bigoplus\limits_{g\in G}\Ccal_g$, a $G$-crossed balance is a family of isomorphisms
\[
\theta_X: X\xrightarrow{\cong }g X
\]
for all $g\in G$ and $X\in \mathcal{C}_g$, satisfying a set of compatibilities with the $G$-crossed braiding (see Definition \ref{def: CrossedCategoricalTwist}). Here, $gX$ is the object obtained by acting with $g\in G$ on $X\in\Ccal$. We show that any twisted representation of $\A$ is conformal covariant, and hence admits an action of the universal cover of the Möbius group. This allows us to define the $G$-crossed balance on $\Rep^G(\A)$ as the action $e^{-2\pi i L_0}$, where $L_0$ denotes the generator of rotations in $\Diff^+(S^1)$. A $G$-crossed braided $\mathrm{W}^*$-tensor category with a compatible $G$-crossed balance is called a $G$-crossed balanced $\mathrm{W}^*$-tensor category. The main results of this paper are Theorems \ref{Thm: RepAutAIsCrossedBalanced} and~\ref{thm: BigCorollary}.

\begin{theorem*}
    Let $\A$ be a (not necessarily rational) conformal net with a faithful action of a discrete group $G$. The $G$-crossed braided $\mathrm{W}^*$-tensor category $\Rep^{G}(\A)$ of $G$-twisted representations of $\A$ admits a canonical structure of a $G$-crossed balanced $\mathrm{W}^*$-tensor category.
\end{theorem*}

Note that, restricting to the case where $G$ is the trivial group, we obtain a balanced structure on the braided $\mathrm{W}^*$-tensor category $\Rep(\A)$. This result has appeared independently in the note \cite{marinsalvador2026balancedstructurecategoryrepresentations}, which contains a much more accessible proof for the case where no group is~present.

\begin{corollary*}
    Let $\A$ be a (not necessarily rational) conformal net. The braided $\mathrm{W}^*$-tensor category $\Rep(\A)$ of representations of $\A$ admits a canonical structure of a balanced $\mathrm{W}^*$-tensor category.
\end{corollary*}

Whenever $\A$ is rational, $\Rep(\A)$ is a unitary fusion category and hence it admits a canonical pivotal structure. Such a pivotal structure induces, via the usual drawing of a kink and the evaluation and coevaluation maps, also a canonical balance on $\Rep(\A)$. In that situation, the conformal Spin and Statistics Theorem of \cite{MR1410566} says exactly that this balance and the one we construct agree.

We want to emphasize that, while the $G$-crossed braided structure on $\Rep^G(\A)$ can be constructed by transporting the $G$-crossed braided structure on $G-\Loc_{I_0}(\A)$ under the equivalence of $\mathrm{W}^*$-categories $\Rep^G(\A)\cong G-\Loc_{I_0}(\A)$, this is not true for the crossed balance. The crossed balance does not seem to have a natural definition in $G-\Loc_{I_0}(\A)$, which requires us to construct the $G$-crossed braided structure on $\Rep^G(\A)$ explicitly in a way that makes evident the existence of the $G$-crossed balance.

\subsection*{Relation to the category of representations of the fixed-points conformal net} In the companion paper \cite{FixedPoints}, we use the results in this article to characterize categories of representations of fixed-points conformal nets. We include the main result of that reference~here.

Fix $G$ to be a finite group acting faithfully on $\A$. The action of $G$ on $\A$ allows us to construct the fixed-points conformal net $\A^G$ by setting, for an interval~$I\subset S^1$,
\[
\A^G(I):=\A(I)^{G},
\]
the von Neumann subalgebra of $\A(I)$ given by the fixed points of the restriction of the action of $G$ to $\A(I)$. These algebras act on the fixed-points Hilbert space $H_0^G$ of $H_0$ under the action of $G$ that implements the actions on the algebras $\A(I)$. In parallel, the $G$-crossed braided category $\Rep^G(\A)$ can be equivariantized with respect to its $G$-action in order to obtain a braided $\mathrm{W}^*$-tensor category $\big(\Rep^G(\A)\big)^G$ \cite{braidedFC}. 

It follows from \cite[Thm. 3.12]{muger05} and the equivalence $\Rep^G(\A)\cong G-\Loc_{I_0}(\A)$ that, if $\A$ is a \emph{rational} conformal net, there is an equivalence of braided tensor categories
    \begin{equation}\label{eq: intro}
        \Rep(\A^G)\cong \big(\Rep^G(\A)\big)^G=\Big(\bigoplus_{g\in G} \Rep^{\Phi(g)}(\A)\Big)^G.
    \end{equation}
In \cite{FixedPoints}, we extend this result and prove the equivalence \eqref{eq: intro} when $\A$ is not necessarily rational, and hence $\Rep^G(\A)$ is not necessarily a fusion category (while keeping $G$ finite). The underlying functor of this equivalence is the functor $\mathfrak{R}: (\Rep^G(\A))^G\to \Rep(\A^G)$ that restricts a $G$-equivariant twisted representation of $\A$ on a Hilbert space $H$ to an $\A^G$-representation on the subspace of $G$-invariant vectors of $H$. Note that the equivariantization of the $G$-crossed balance on $\Rep^G(\A)$ produces an honest balance on the braided $\mathrm{W}^*$-tensor category $(\Rep^G(\A))^G$. As we have discussed above, our construction also produces a balance on the category $\Rep(\A^G)$. We show that the equivalence of braided $\mathrm{W}^*$-tensor categories $(\Rep^G(\A))^G\cong \Rep(\A^G)$ is actually an equivalence of balanced $\mathrm{W}^*$-tensor categories.

\begin{theorem*}(\cite{FixedPoints})
    Let $\A$ be a (not necessarily rational) conformal net with a faithful action of a finite group $G$. The equivalence of braided $\mathrm{W}^*$-tensor categories
    \[
    \mathfrak{R}:(\Rep^G(\A))^G\xrightarrow{\cong} \Rep(\A^G)
    \]
   provides an equivalence of balanced $\mathrm{W}^*$-tensor categories.
\end{theorem*}

   These results are used in \cite{RepHeis} to produce the first explicit computation of a category of representations of a conformal net with irreducible representations that tensor to a direct integral of irreducible representations, as opposed to a direct sum.

\section*{Acknowledgements}
I am grateful to André Henriques for constant conversations during this project, and to Bin Gui for his help setting up the basics of the Connes fusion of twisted representations in Section \ref{Sec: CatOfTwistedReps} and his valuable comments on earlier versions of this article. I want to thank Josep Fontana McNally for proofreading this paper. This work has been funded by the EPSRC grant EP/W524311/1.

For the purpose of Open Access, the author has applied a CC BY public copyright license to any Author Accepted Manuscript version arising from this submission.

\addtocontents{toc}{\protect\setcounter{tocdepth}{5}}

\section{Preliminaries}

\subsection{Conformal nets}
\label{Sec: ConfNets}

An interval of the standard circle $S^1:=\{z\in \C\ |\ |z| = 1\}$ is an open, connected, non-empty, non-dense subset of $S^1$, and we denote by $\Jcal$ the collection of intervals of $S^1$. We denote by $\Mob$ the set of Möbius transformations of $S^1$, that is, transformations of the form
\[
z\mapsto \frac{az+b}{\bar{b}z+\bar{a}}
\]
with $a,b\in \C$, $|a|^2-|b|^2 = 1$. The group $\Mob$ is a Lie group isomorphic to $PSL(2, \R)$ and contains the group $\Rot$ of rotations of the circle. We denote an anticlockwise rotation of angle $t$ by $R_t$. We assume all Hilbert spaces to be separable, and all von Neumann algebras to have separable predual.

\begin{definition}
    A \textit{Möbius covariant net on $S^1$} is a tuple $(H_0, \A, U, \Omega)$ where $H_0$ is a Hilbert space equipped with a nonzero vector $\Omega\in H_0$, $U$ is a strongly continuous unitary representation of $\Mob$ on $H_0$ and $\A$ is an assignment of a von Neumann algebra $\A(I)$ acting on $H_0$ for every interval $I\in\Jcal$. This data is required to satisfy, for every $I,J\in\Jcal$ and $\varphi\in\Mob$,
    \begin{enumerate}
        \item isotony: if $J\subset I$, then $\A(J)\subset \A(I)$;
        \item locality: if $I\cap J = \emptyset$, then $\A(J)$ and $\A(I)$ commute in $B(H_0)$;
        \item Möbius covariance: $U(\varphi)\A(I)U(\varphi)^* = \A(\varphi I)$;
        \item positivity of the energy: the representation $U$ is of positive energy, meaning that the conformal Hamiltonian $L_0$, defined by $U(R_t) = e^{itL_0}$ is positive;
        \item uniqueness of the vacuum: the vector $\Omega\in H_0$ is the unique vector, up to a constant, which is invariant under $U$;
        \item cyclicity of the vacuum: $\Omega$ is cyclic for the von Neumann  algebra $\A(S^1):=\bigvee\limits_{I\in \Jcal}\A(I)\subset B(H_0)$.
    \end{enumerate}
\end{definition}

Let $\Diff^+(S^1)$ denote the group of orientation-preserving diffeomorphisms of $S^1$. A strongly continuous projective unitary representation $V$ of $\Diff^+(S^1)$ on a Hilbert space $H$ is a strongly continuous homomorphism $V: \Diff^+(S^1)\to \mathcal{P}U(H)$. Note that $\Mob\subset\Diff^+(S^1)$. We say that $V$ is an extension of a unitary representation $U$ of $\Mob$ on $H$ if, for every $\varphi\in \Mob$, it holds that $V(\varphi) = [U(\varphi)]$, where $[-]: U(H)\to \mathcal{P}U(H)$ denotes the projection map.

\begin{definition}
    A Möbius covariant net $(H_0, \A ,U, \Omega)$ is said to be a \emph{conformal net} if there is an extension of $U$ to a unitary projective representation of $\Diff^+(S^1)$, which we also denote by $U$, such that for all $I\in\Jcal$ and $\varphi\in \Diff^+(S^1)$, it holds that
    \begin{enumerate}
        \item $U(\varphi)\A(I) U(\varphi)^* = \A(\varphi I)$;
        \item if $\varphi|_{I} = \id_I$, then $\text{Ad}(U(\varphi))|_{\A(I)} = \id_{\A(I)}.$
    \end{enumerate}
\end{definition}

The following are well-known properties of any conformal net $(H_0, \A, U, \Omega)$. 
\begin{enumerate}
    \item Haag duality: the commutant of $\A(I)$ in $B(H_0)$ is $\A(I^c)$, where $I^c$ denotes the interior of $S^1\setminus I$ \cite[Thm. 2.19]{gf93};
    \item The Reeh-Schlieder Theorem: the vacuum vector $\Omega$ is cyclic and separating for every von Neumann algebra $\A(I)$ \cite[Cor. 2.8]{gf93};
    \item Each von Neumann algebra $\A(I)$ is a type $\mathrm{III}_1$ factor \cite[Prop. 1.2]{MR1410566};
    \item Additivity: for any interval $I\in \Jcal$ and any collection of intervals $\{I_\alpha\}_{\alpha \in A}$ with $I_\alpha\in \Jcal$ and $\bigcup\limits_{\alpha\in A}I_\alpha = I$, it holds that $\A(I) = \bigvee\limits_{\alpha\in A}\A(I_\alpha)$ \cite[p. 545]{MR1376431}.
\end{enumerate}

Fix a conformal net $(H_0,\A,U, \Omega)$, which we will simply denote by $\A$ from now on.

\begin{definition}\label{def: Rep}
    A \emph{representation} of a conformal net $\A$ consists of a Hilbert space $H$ and a collection of $*$-homomorphisms $\pi_I:\A(I)\to B(H)$ for every interval $I\in\Jcal$ such that
\[
\pi_I|_{\A(J)} = \pi_J
\]
whenever $J\in\Jcal$ is an interval such that $J\subset I$. We write $\Rep(\A)$ for the category whose objects are representations of $\A$ and whose morphisms are bounded linear maps intertwining the actions of all the algebras $\A(I)$.
\end{definition}

The conformal net $\A$ has a canonical representation on its vacuum Hilbert space $H_0$ by definition, called the \emph{vacuum representation of} $\A$. Given $I\in\Jcal$ and $x\in \A(I)$, we write $\pi_{0,I}(x)$ or $\pi_0(x)$ for the action of $x$ on the vacuum representation $H_0$.

The category $\Rep(\A)$ is a braided tensor category, the structure of which can be produced in different equivalent ways \cite{gf93, was98, bdh15,Gui21}.

We are interested in more general representations of $\A$, those which are twisted by automorphisms of the conformal net. An \emph{automorphism} $\varphi$ of $\A$ consists of a von Neumann algebra automorphism $\varphi_I: \A(I)\to\A(I)$ for every $I\in\Jcal$ such that for every inclusion $J\subset  I$ of intervals, the following diagram commutes
\[\begin{tikzcd}
	{\A(I)} & {\A(I)} \\
	{\A(J)} & {\A(J)}.
	\arrow["{\varphi_I}", from=1-1, to=1-2]
	\arrow[hook, from=2-1, to=1-1]
	\arrow["{\varphi_J}"', from=2-1, to=2-2]
	\arrow[hook, from=2-2, to=1-2]
\end{tikzcd}\]
Automorphisms are furthermore required to be unitarily implemented in the vacuum representation in the sense that they come equipped with a unitary $V_\varphi: H_0\to H_0$ such that, for every $I\in \Jcal$, it holds that
\[
\pi_{0,I}\circ\varphi_I(-) = V_\varphi\circ \pi_{0,I}(-)\circ V_\varphi^{-1},
\]
and $V_\varphi(\Omega) = \Omega$. We denote by $\Aut(\A)$ the group of automorphisms of $\A$. Given $\varphi\in \Aut(\A)$, an interval $I\in \Jcal$ and $x\in\A(I)$, we will write $\varphi x:=\varphi_I x$.

Automorphisms of $\A$ allow us to define twisted representations of the conformal net, as follows. Let us pick $\mathrm{p}:=1\in S^1$, and fix it for the rest of the paper.

\begin{definition}\label{def: TwistedRep}
Let $\varphi\in\Aut(\A)$ be an automorphism of $\A$ and recall that we have fixed $\mathrm{p} = 1\in S^1$. A $\varphi$\emph{-twisted representation} of $\A$ consists of a Hilbert space $H$ equipped with a collection of $*$-actions $\pi^H_I$ of $\A(I)$ on $H$, such that for every pair of intervals $I,J\in\Jcal$ with $J\subset I$, the diagram
\[\begin{tikzcd}
	{\A(I)} & {\A(I)} & {B(H)} \\
	& {\A(J)}
	\arrow["{\varphi_I^{-1}}", from=1-1, to=1-2]
	\arrow["{\pi^H_I}", from=1-2, to=1-3]
	\arrow[hook, from=2-2, to=1-1]
	\arrow["{\pi^H_J}"', from=2-2, to=1-3]
\end{tikzcd}\]
commutes if $\mathrm{p}\in \mathrm{cl}({I})$, $\mathrm{p}\notin{\mathrm{cl}({J})}$, and $J$ is counter-clockwise to $\mathrm{p}$ ($\mathrm{cl}(L)$ denotes the closure of an interval $L$); and 
\[\begin{tikzcd}
	{\A(I)} & {B(H)} \\
	{\A(J)}
	\arrow["{\pi^H_I}", from=1-1, to=1-2]
	\arrow[hook, from=2-1, to=1-1]
	\arrow["{\pi_J^H}"', from=2-1, to=1-2]
\end{tikzcd}\]
commutes otherwise. We write $\Rep^\varphi(\A)$ for the category of $\varphi$-twisted representations of $\A$, where morphisms are bounded linear operators commuting with the actions of all the von Neumann algebras $\A(I)$. We denote by $\Rep^{\Aut(\A)}(\A) :=\bigoplus\limits_{\varphi\in \Aut(\A)}\Rep^\varphi(\A)$ the category whose objects are countable direct sums of twisted $\A$-representations.
\end{definition}
\begin{remark}
    Note that, although $\Aut(\A)$ may be naturally a topological group, we are treating it as a discrete group, and an object of $\Rep^{\Aut(\A)}(\A)$ decomposes as a direct sum of objects living in countably many of the subcategories $\Rep^\varphi(\A)$. If $\Aut(\A)$ is a finite group, or we are only interested in a finite subgroup of automorphisms, then $\Rep^{\Aut(\A)}(\A)$ provides a complete picture of the problem at hand.
\end{remark}

\begin{remark}
    Given distinct automorphisms $\varphi,\mu\in \Aut(\A)$, and twisted representations $H^\varphi\in\Rep^\varphi(\A)$ and $K^\mu\in\Rep^\mu(\A)$, the only bounded linear map $H^\varphi\to K^\mu$ commuting with the actions of all the algebras $\A(I)$ is the zero map.
\end{remark}

The condition in Definition \ref{def: TwistedRep} can also be formulated as follows. Let $q: \R\to S^1$ be the map $t\mapsto e^{2\pi i t}$. Given an interval $I\in \Jcal$, we shall write $\hat{I}\subset \R_{>0}$ for the lift of $I$ to $\R_{> 0}$ as close to zero as possible so that $\mathrm{cl}(\hat{I})\subset\R_{>0}$. We call an inclusion $J\subset I$ in $\Jcal$ \emph{standard} if $\hat{J}\subset \hat{I}$, and \emph{special} if it is not standard. For an inclusion $J\subset I$, we shall write $\delta^{\varphi^{-1}}_{J\subset I}: \A(J)\to \A(I)$ for the inclusion $\A(J)\hookrightarrow \A(I)$ if $J\subset I$ is standard, and for $\A(J)\hookrightarrow\A(I)\xrightarrow{\varphi^{-1}_I}\A(I)$ if $J\subset I$ is special. Then, the condition in Definition \ref{def: TwistedRep} is equivalent to the diagram
\[\begin{tikzcd}
	{\A(I)} & {B(H)} \\
	{\A(J)}
	\arrow["{\pi^H_I}", from=1-1, to=1-2]
	\arrow["{\delta_{J\subset I}^{\varphi^{-1}}}", hook, from=2-1, to=1-1]
	\arrow["{\pi^H_J}"', from=2-1, to=1-2]
\end{tikzcd}\]
commuting for all inclusions $J\subset I$ of intervals in $\Jcal$. 

In what follows, it will be useful to have the following characterization of special inclusions. Let $\Jcal_\R := \{\tilde{I}\subset \R\ |\ \tilde{I} \text{ is an open interval such that $q\tilde{I} \in \Jcal$}\}$. For an interval $\tilde{I}\in \Jcal_\R$, we denote $I : =q\tilde{I}\in\Jcal$. Also, let 
\[
\epsilon(\tilde{I}):=\begin{cases}
    |(0, \partial_-\tilde{I})\cap \Z|,
    & \text{if $\partial_{-} \tilde{I}$} > 0\\
    -|[\partial_-\tilde{I}, 0]\cap \Z|, & \text{if $\partial_-\tilde{I}\leq0$}.
\end{cases}
\] 
We have the following lemma.

\begin{lemma}\label{lemm: epsilon}
    Let $\tilde{J}\subset \tilde{I}$ be an inclusion of intervals in $\Jcal_\R$. The inclusion $J\subset I$ is standard if and only if $\epsilon(\tilde{J}) = \epsilon(\tilde{I})$. If the inclusion $J\subset  I $ is special, then $\epsilon(\tilde{I}) = \epsilon(\tilde{J}) - 1$.
\end{lemma}
\begin{proof}
   Assume that $0<\partial_-\tilde{I}$. Then, we have that $\epsilon(\tilde{J}) = \epsilon(\tilde{I})\iff [\partial_-\tilde{I} ,\partial_-\tilde{J})\cap \Z = \emptyset$. This is equivalent to $\mathrm{p}\notin [\partial_-I, \partial_-J)$, where $\partial_-$ denotes the left boundary of an interval in $S^1$. This is in turn equivalent to $J\subset I$ being standard. If $\partial_-\tilde{I}\leq 0<\partial_-\tilde{J}$, then it holds that $\epsilon(\tilde{I})<0\leq \epsilon(\tilde{J})$ and also $J\subset I$ is special. The case where $\partial_-\tilde{J}\leq 0$ can be treated similarly to the case $0<\partial_-\tilde{I}$ above. The second part of the statement is clear.
\end{proof}

\subsection{$G$-crossed balanced categories}\label{Sec: G-X-balanced}

We recall the notion of a crossed balanced tensor category. Let $G$ be a discrete group and denote by $\underline{G}$ the tensor category with objects $\text{Ob(\underline{G})} = G$ and morphisms $\Hom_{\underline{G}}(g,g) =\{\id_g\}$ and $\Hom_{\underline{G}}(g,h) = \emptyset$ if $g\neq h$, for $g,h\in G$. The tensor product is given by $g\otimes h = gh$. Let $\mathcal{C}$ be a $\mathbb{C}$-linear category. We assume all categories to be direct sum and idempotent complete. Recall that $\mathcal{C}$ is a $*$-category if it is equipped with a dagger structure $*:\text{Hom}(X,Y)\to \Hom(Y,X)$ which is antilinear and satisfies $f^{**} = f$ and $(f\circ g)^* = g^*\circ f^*$. We say that $\mathcal{C}$ is a $\mathrm{W}^*$-category if it is a $*$-category and $\End(X)$ is a von Neumann algebra for all $X\in \mathcal{C}$. A functor between $\mathrm{W}^*$-categories is a $\mathbb{C}$-linear functor $F$ respecting the $*$-structure and such that the map $\End(X)\to \End(FX)$ is normal for all $X$ in the domain category. $\mathrm{W}^*$-tensor categories (and braided $\mathrm{W}^*$-tensor categories) are defined in the obvious way, and requiring all structure morphisms to be unitary \cite{henriques2024completewcategories}.

Given a $\mathrm{W}^*$-category $\mathcal{C}$, we denote by $\underline{\text{Aut}}_\otimes(\mathcal{C})$ the tensor category of $\mathrm{W}^*$-tensor automorphisms of $\mathcal{C}$ and unitary $\mathrm{W}^*$-tensor natural transformations. 

\begin{definition}
An \textit{action of $G$ on $\mathcal{C}$} is a tensor functor
\[
T:\underline{G}\to \underline{\text{Aut}}_\otimes(\mathcal{C}).
\]
\end{definition}
Hence, the data of an action of $G$ on $\mathcal{C}$ provides
\begin{enumerate}
    \item $\mathrm{W}^*$-automorphisms $T_g:\mathcal{C}\to \mathcal{C}$ for every $g\in G$;
    \item a unitary natural equivalence $\text{Id}_{\Ccal}\cong T_e$, where $e\in G$ is the unit;
    \item a unitary isomorphism $i_g: \1\to T_g(\1)$ for every $g\in G$;
    \item a unitary $\mathrm{W}^*$-tensor natural isomorphism $s_g:T_g(-)\otimes T_g(-)\to  T_g(-\otimes -)$ for every $g\in G$;
    \item a unitary $\mathrm{W}^*$-tensor natural isomorphism $\eta_{g,h}:T_g\circ T_h \to  T_{gh}$ for every $g,h\in G$
\end{enumerate}
satisfying certain coherences. We say the action is strict if $s_g$ and $\eta_{g,h}$ are identities for all $g,h\in G$. Let us recall the concept of a $G$-crossed braided tensor category \cite{braidedFC}.

\begin{definition}\label{def: G-X-braidedWTensorCat}
    A \textit{$G$-crossed braided $\mathrm{W}^*$-tensor category} consists of
    \begin{enumerate}
        \item a $\mathrm{W}^*$-tensor category $\mathcal{C}$;
        \item an action $T: \underline{G}\to \underline{\text{Aut}}_\otimes(\mathcal{C})$ of $G$ on $\mathcal{C}$;
        \item a grading $\mathcal{C} = \bigoplus\limits_{g\in G}\mathcal{C}_g$ compatible with the tensor product in the sense that $-\otimes-$ provides functors 
        \[
        \Ccal_g\times\Ccal_h\to \Ccal_{gh}
        \]
        for all $g,h\in G$;
        \item for every $g\in G$, $X\in \mathcal{C}_g$ and $Y\in \mathcal{C}$, a unitary isomorphism
        \[
        \beta_{X,Y} : X\otimes Y\to T_g(Y)\otimes X
        \]
        natural in $X$ and $Y$.
    \end{enumerate}
    This data is required to satisfy
    \begin{enumerate}
        \item $T_g(\mathcal{C}_h) \subset \Ccal_{ghg^{-1}}$ for all $g,h\in G$;
        \item the following diagram commutes 
\[\begin{tikzcd}
	{T_g(X)\otimes T_g(Y)} && {T_{ghg^{-1}}T_g(Y)\otimes T_g(X)} \\
	{T_g(X\otimes Y)} && {T_g(T_h(Y)\otimes X)}
	\arrow["{\beta_{T_g(X),T_g(Y)}}", from=1-1, to=1-3]
	\arrow["\cong", from=1-1, to=2-1]
	\arrow["\cong"', from=1-3, to=2-3]
	\arrow["{T_g(\beta_{X,Y})}"', from=2-1, to=2-3]
\end{tikzcd}\]
for all $g,h\in G$, $X\in\mathcal{C}_h$ and $Y\in\mathcal{C}$;
\item the following diagram commutes
\begin{equation}\label{eq: Xbraiding1}\begin{tikzcd}
	{X\otimes Y\otimes Z} && {T_g(Y\otimes Z)\otimes X} \\
	{T_g(Y)\otimes X\otimes Z} && {T_g(Y)\otimes T_g(Z)\otimes X}
	\arrow["{\beta_{X, Y\otimes Z}}", from=1-1, to=1-3]
	\arrow["{\beta_{X,Y}\otimes\id_Z}"', from=1-1, to=2-1]
	\arrow["{\id_{T_g(Y)}\otimes \beta_{X,Z}}"', from=2-1, to=2-3]
	\arrow["\cong"', from=2-3, to=1-3]
\end{tikzcd}\end{equation}
for all $g\in G$, $X\in \mathcal{C}_g$ and $Y,Z\in \mathcal{C}$;
\item the following diagram commutes
\begin{equation}\label{eq: Xbraiding2}
\begin{tikzcd}
	{X\otimes Y\otimes Z} && {T_{gh}(Z)\otimes X\otimes Y} \\
	{X\otimes T_h(Z)\otimes Y} && {T_gT_h(Z)\otimes X\otimes Y}
	\arrow["{\beta_{X\otimes Y,Z}}", from=1-1, to=1-3]
	\arrow["{\id_X\otimes \beta_{Y,Z}}"', from=1-1, to=2-1]
	\arrow["{\beta_{X, T_h(Z)}\otimes \id_Y}"', from=2-1, to=2-3]
	\arrow["\cong"', from=2-3, to=1-3]
\end{tikzcd}\end{equation}
for all $g,h\in G$, $X\in \mathcal{C}_g$, $Y\in\mathcal{C}_h$ and $Z\in\mathcal{C}$.
    \end{enumerate}
\end{definition}
The unlabelled vertical isomorphisms in the diagrams above are constructed using the natural transformations $s$ and $\eta$, and the associators are omitted for readability. We call a $\mathrm{G}$-crossed braided tensor category \emph{strict} if the $G$-action is strict. Given $(\mathcal{C}, \otimes^\Ccal, T^\Ccal, \beta^\Ccal)$ and $(\mathcal{D}, \otimes^\Dcal, T^\Dcal, \beta^\Dcal )$ two $G$-crossed braided $\mathrm{W}^*$-tensor categories, a functor between them consists of a $\mathrm{W}^*$-tensor functor $(F,\Psi): (\Ccal,\otimes^\Ccal)\to (\Dcal, \otimes^\Dcal)$ preserving the $G$-grading, together with unitary natural isomorphisms 
\[
\Phi_g(X): T_g^\Dcal(F(X)) \xrightarrow{\cong }   F(T^\Ccal_g(X))
\]
for all $g\in G$ and $X\in \Ccal$. These are required to make the following diagrams commute
\[\begin{tikzcd}[column sep=large]
	{T_g^\Dcal(T_h^\Dcal(F(X)))} && {T_{gh}^{\Dcal}(F(X))} \\
	{T^\Dcal_g(F(T_h^\Ccal(X)))} && {F(T_{gh}^\Ccal(X))} \\
	& {F(T_g^\Ccal(T_h^\Ccal(X)))}
	\arrow["\cong", from=1-1, to=1-3]
	\arrow["{T_g^\Dcal(\Phi_h(X))}"', from=1-1, to=2-1]
	\arrow["{\Phi_{gh}(X)}", from=1-3, to=2-3]
	\arrow["{\Phi_{g}(T_h^\Ccal(X))}"', from=2-1, to=3-2]
	\arrow["\cong"', from=3-2, to=2-3]
\end{tikzcd}\]
\[\begin{tikzcd}[column sep=large]
	{T_g^\Dcal(F(X))\otimes^\Dcal T_g^\Dcal(F(Y))} && {F(T_g^\Ccal(X))\otimes^\Dcal F(T_g^\Ccal(Y))} \\
	{T_g^\Dcal(F(X)\otimes^\Dcal F(Y))} && {F(T_g^\Ccal(X)\otimes^\Ccal T_g^\Ccal(Y))} \\
	{T_g^\Dcal(F(X\otimes^\Ccal Y))} && {F(T_g^\Ccal(X\otimes^\Ccal Y))}
	\arrow["{\Phi_g(X)\otimes^\Dcal \Phi_g(Y)}", from=1-1, to=1-3]
	\arrow["\cong"', from=1-1, to=2-1]
	\arrow["{\Psi_{T_g^\Ccal(X),T_g^\Ccal(Y)}}"', from=1-3, to=2-3]
	\arrow["{T_g^\Dcal(\Psi_{X,Y})}", from=2-1, to=3-1]
	\arrow["\cong", from=2-3, to=3-3]
	\arrow["{\Phi_g(X\otimes^\Ccal Y)}"', from=3-1, to=3-3]
\end{tikzcd}\]
for all $g,h\in G$ and $X,Y\in\Ccal,$ and where the unlabelled arrows correspond to the structure morphisms of the actions of $G$ on $\Ccal$ and $\Dcal$. Furthermore, the triple $(F,\Psi,\Phi)$ is required to satisfy the following compatibility with the $G$-crossed braiding,
\[\begin{tikzcd}
	{F(X)\otimes^\Dcal F(Y)} && {F(X\otimes^\Ccal Y)} \\
	{T_g^\Dcal(F(Y))\otimes^\Dcal F(X)} && {F(T_g^\Ccal(Y)\otimes^\Ccal X)} \\
	& {F(T_g^\Ccal(Y))\otimes^\Dcal F(X)}
	\arrow["{\Psi_{X,Y}}", from=1-1, to=1-3]
	\arrow["{\beta_{F(X),F(Y)}^\Dcal}"', from=1-1, to=2-1]
	\arrow["{F(\beta^\Ccal_{X,Y})}", from=1-3, to=2-3]
	\arrow["{\Phi_{g}(Y)\otimes\id_{F(X)}}"', from=2-1, to=3-2]
	\arrow["{\Psi_{T_g^\Ccal(Y), X}}"', from=3-2, to=2-3]
\end{tikzcd}\]
for all $g\in G$ and $X\in \mathcal{C}_g$ and $Y\in\Ccal$. We shall call such a triple $(F,\Psi,\Phi)$ a $G$\emph{-crossed braided} $\mathrm{W}^*$\emph{-tensor functor}. A $G$-crossed braided $\mathrm{W}^*$-tensor functor is an equivalence of $G$-crossed braided $\mathrm{W}^*$-tensor categories if its underlying $\mathrm{W}^*$-functor is an equivalence of $\mathrm{W}^*$-categories.

Let $\mathcal{C}$ be a $G$-crossed braided $\mathrm{W}^*$-tensor category.

\begin{definition}\label{def: CrossedCategoricalTwist}
    A $G$\emph{-crossed balance} on $\mathcal{C}$ is a family of natural unitary isomorphisms $\theta_X: X\to T_g(X)$ for every $g\in G$ and $X\in \mathcal{C}_g$ such that 
    \begin{enumerate}
                \item the following diagram commutes for all $g,h\in G$ and $X\in \mathcal{C}_g$
\[\begin{tikzcd}
	{T_h(X)} && {T_{hgh^{-1}}T_h(X)} \\
	{T_hT_g(X)} && {T_{gh}(X)};
	\arrow["{\theta_{T_h(X)}}", from=1-1, to=1-3]
	\arrow["{T_h(\theta_X)}"', from=1-1, to=2-1]
	\arrow["\cong", from=1-3, to=2-3]
	\arrow["\cong"', from=2-1, to=2-3]
\end{tikzcd}\]
        \item the following diagram commutes
\begin{equation}\label{eq: Xtwist}
\begin{tikzcd}
	{X\otimes Y} && {T_{gh}(X\otimes Y)} \\
	&& {T_{gh}(X)\otimes T_{gh}(Y)} \\
	&& {T_{ghgh^{-1}g^{-1}}T_{ghg^{-1}}(X)\otimes T_{ghg^{-1}}T_g(Y)} \\
	{T_g(Y)\otimes X} && {T_{ghg^{-1}}(X)\otimes T_g(Y)}
	\arrow["{\theta_{X\otimes Y}}", from=1-1, to=1-3]
	\arrow["{\beta_{X,Y}}"', from=1-1, to=4-1]
	\arrow["\cong"', from=2-3, to=1-3]
	\arrow["\cong"', from=3-3, to=2-3]
	\arrow["{\beta_{T_g(Y), X}}"', from=4-1, to=4-3]
	\arrow["{\theta_{T_{ghg^{-1}}(X)}\otimes\theta_{T_g(Y)}}"', from=4-3, to=3-3]
\end{tikzcd}\end{equation}
for all $g,h\in G$, $X\in \mathcal{C}_g$, $Y\in\mathcal{C}_h$.
    \end{enumerate}
    A $G$\emph{-crossed balanced $\mathrm{W}^*$-tensor category} is a $G$-crossed braided $\mathrm{W}^*$-tensor category with a $G$-crossed balance.
\end{definition}

Given $\Ccal$ and $\Dcal$ two $G$-crossed braided $\mathrm{W}^*$-tensor categories with $G$-crossed balances $\theta^\Ccal$ and $\theta^\Dcal$ respectively, a $G$\emph{-crossed balanced $\mathrm{W}^*$-tensor functor} $(F,\Psi,\Phi):\Ccal\to \Dcal$ is a functor of $G$-crossed braided $\mathrm{W}^*$-tensor categories such that
\[\begin{tikzcd}[column sep=large]
	{F(X)} && {T_g^\Dcal(F(X))} \\
	&& {F(T^\Ccal_g(X))}
	\arrow["{\theta^\Dcal_{F(X)}}", from=1-1, to=1-3]
	\arrow["{F(\theta^\Ccal_X)}"', from=1-1, to=2-3]
	\arrow["{\Phi_g(X)}", from=1-3, to=2-3]
\end{tikzcd}\]
commutes for all $g\in G$ and $X\in\Ccal_g$.

\section{The $G$-crossed balanced category of twisted representations of a conformal net}
\label{Sec: CatOfTwistedReps}

Let $G$ be a discrete group acting on a conformal net $\A$ by a group homomorphism $G\to \Aut(\A)$. All actions of groups on conformal nets are assumed to be faithful. Throughout, we identify an element $g\in G$ with the automorphism of $\A$ it induces. In this section, we show that the $\mathrm{W}^*$-category
\[
\Rep^G(\A):=\bigoplus\limits_{g\in G}\Rep^g(\A)
\]
admits a structure of a $G$-crossed balanced $\mathrm{W}^*$-tensor category. The fact that $\Rep^G(\A)$ admits the structure of a $G$-crossed braided category was shown, in the language of localized endomorphisms of $\A_\infty:=\bigcup\limits_{I\in \Jcal_\mathrm{p}}\A(I) \subset B(H_0)$, in \cite{muger05}, where $\Jcal_\mathrm{p}$ denotes the set of intervals in $\Jcal$ not containing the fixed point $\mathrm{p} = 1\in S^1$ in their closure. In \cite{Gui21} the author gives a prescription to define a braided tensor structure on the category $\Rep(\A)$ directly, without making use of localized endomorphisms. We extend his arguments to construct a $G$-crossed braided structure on $\Rep^G(\A)$, which makes apparent the existence of the $G$-crossed balance. We do not know how to obtain the $G$-crossed balance in the language of localized endomorphisms of \cite{muger05}, and hence we need to construct the crossed braided structure on $\Rep^G(\A)$ directly, without making reference to localized endomorphisms.

\subsection{Connes fusion of twisted representations}\label{Sec: ConnesFusionOfTwisted}

Recall that we fix the point $\mathrm{p} = 1\in S^1$ and we write $q: \R\to S^1$ for the exponential map $q(t) = e^{2\pi i t}$. We also denote by $\Jcal_\R$ the set of intervals $\tilde{I}$ on $\R$ such that $I:=q(\tilde{I})$ is an interval in $\Jcal$. For an interval $I\in \Jcal$, we write $\hat{I}$ for the lift of $I$ to $\R$ as close to 0 as possible so that $\mathrm{cl}(\hat{I})\subset\R_{>0}$. Recall that we have defined, for every $\tilde{I}\in \Jcal_\R$, 
\[
\epsilon(\tilde{I}):=\begin{cases}
    |(0, \partial_-\tilde{I})\cap \Z|,
    & \text{if $\partial_{-} \tilde{I}$} > 0\\
    -|[\partial_-\tilde{I}, 0]\cap \Z|, & \text{if $\partial_-\tilde{I}\leq0$}.
\end{cases}
\] 

Given a conformal net $\A$, we can produce a partially defined net $\A_\R$ on the real line by assigning to every interval $\tilde{I}\in \Jcal_\R$ the von Neumann algebra $\A_\R(\tilde{I}) := \A(I)$. Note that the net $\A_\R$ satisfies the weaker version of locality that $\A_\R(\tilde{I})$ and $\A_\R(\tilde{J})$ commute if $I\cap J = \emptyset$. Given a $\varphi$-twisted representation $(H^\varphi, \pi^H)$ of $\A$, we obtain an action $\pi^H_{\tilde{I}}$ of every von Neumann algebra $\A_\R(\tilde{I})$ on $H^\varphi$ by setting
\[
    \pi^H_{\tilde{I}}(x) := \pi^H_I(\varphi^{\epsilon(\tilde{I})}x)
\]
for every $\tilde{I}\in \Jcal_\R$ and $x\in A_\R(\tilde{I}) = \A(I)$. Note that it holds that, for every inclusion $\tilde{J}\subset \tilde{I}$ in $\Jcal_\R$ and $x\in \A_\R(\tilde{J}) = \A(J)$,
\[
\pi^H_{\tilde{J}}(x) = \pi^H_J(\varphi^{\epsilon(\tilde{J})}x) = \pi^H_I(\delta^{\varphi^{-1}}_{J\subset I}\varphi^{\epsilon(\tilde{J})}x) = \pi^H_I(\varphi^{\epsilon(\tilde{I})}x) = \pi_{\tilde{I}}^H(x),
\]
where we have used Lemma \ref{lemm: epsilon} in the third equality. In addition, for all $x\in \A_\R(\tilde{I}) = \A(I) = \A_\R(\tilde{I} + 1)$, we have
\begin{equation}\label{eq: TwistedRepOnR}
\pi_{\tilde{I}}^H(\varphi x ) = \pi_{\tilde{I} + 1}^H( x)
\end{equation}
coming from the fact that $\epsilon(\tilde{I} + 1) = \epsilon(\tilde{I}) + 1$. Actually, given a Hilbert space $K$ with an action $\pi_{\tilde{I}}$ of the von Neumann algebra $\A_\R(\tilde{I})$ for every $\tilde{I}\in\Jcal_\R$, compatible with the inclusion of intervals in $\Jcal_\R$ and satisfying \eqref{eq: TwistedRepOnR}, we obtain a $\varphi$-twisted representation $\pi$ of $\A$ on $K$ by setting
\[
\pi_I(x) := \pi_{\hat{I}}(x)
\]
for every $I\in\Jcal$ and $x\in \A(I)$. We will freely move between the notion of a $\varphi$-twisted representation of $\A$ on a Hilbert space $H$ in the sense of Definition \ref{def: TwistedRep} and in the sense of a compatible family of $*$-actions of the von Neumann algebras $\A_\R(\tilde{I})$ on $H$ satisfying \eqref{eq: TwistedRepOnR}.

In this section, we generalize \cite[Sec. 2]{Gui21} to define the tensor product of twisted representations. Given $\tilde{I}\in \Jcal_\R$, we write $\tilde{I}^{c+} := (\partial_+\tilde{I}, \partial_-\tilde{I} + 1)$ and $\tilde{I}^{c-}:= (\partial_+\tilde{I} - 1, \partial_-\tilde{I})$ for the \emph{positive} and the \emph{negative complements} of $\tilde{I}$ respectively, which are intervals in $\Jcal_\R$. Note that $q(\tilde{I}^{c\pm}) = I^{c}$. Let $\varphi\in \Aut(\A)$ and $H^\varphi = (H^\varphi, \pi^H)\in\Rep^\varphi(\A)$. We denote by $\Hom_{\A_\R(\tilde{I})}(H_0, H^\varphi)$ the space of bounded linear operators $T:H_0\to H^\varphi$ such that, for all $x\in \A(I) = \A_\R(\tilde{I})$, it holds that
\[
    T\circ \pi_{0, I}(x) = \pi^H_{\tilde{I}}(x)\circ T.
\]
Since $\A_\R(\tilde{I})$ is a type $\mathrm{III}$ factor and $H_0$ and $H^\varphi$ are separable Hilbert spaces, the $*$-actions of $\A_\R(\tilde{I})$ on them are automatically normal \cite[Thm. V.5.1]{MR1873025}. Furthermore, since $H_0$ and $H^\varphi$ are non-trivial $\A_\R(\tilde{I})$-modules, there is a unitary equivalence $H_0\cong H^\varphi$ of $\A_\R(\tilde{I})$-modules. Hence, there are unitary operators in $\Hom_{\A_\R(\tilde{I})}(H_0, H^\varphi)$.

\begin{definition}
    Let $\tilde{I}\in \Jcal_\R$. A vector $\xi\in H^\varphi$ is said to be $(\tilde{I}, +)$\emph{-bounded} if there exists $T\in \Hom_{\A_\R(\tilde{I}^{c+})}(H_0, H^\varphi)$ such that $\xi = T\Omega$. We write $H_+^\varphi(\tilde{I})$ for the subspace of $H^\varphi$ of $(\tilde{I},+)$-bounded vectors. We define similarly the space of $(\tilde{I},-)$-bounded vectors $H^\varphi_-(\tilde{I})$, using the negative complement $\tilde{I}^{c-}$ of $\tilde{I}$. 
\end{definition}

Given $\xi\in H_+^\varphi(\tilde{I})$, there exists a unique operator $T\in \Hom_{\A_\R(\tilde{I}^{c+})}(H_0, H^\varphi)$ such that $\xi = T\Omega$, as $\A(I^{c})\Omega$ is dense in $H_0$ by the Reeh-Schlieder Theorem. We denote such operator by $T = Z^+(\xi, \tilde{I})$. Similarly, if $\xi\in H^\varphi_-(\tilde{I})$, we denote by $T = Z^-(\xi, \tilde{I})$ the unique operator $T\in \Hom_{\A_\R(\tilde{I}^{c-})}(H_0, H^\varphi)$ such that $\xi = T\Omega$. We have that $H^\varphi_{\pm}(\tilde{I}) = \Hom_{\A(\tilde{I}^{c\pm})}(H_0, H^\varphi)\Omega$. For an untwisted representation $K\in \Rep(\A)$, there is no distinction between $(\tilde{I},+)$- and $(\tilde{I}, -)$-bounded vectors of $K$, and these furthermore only depend on $I$. Also, given $\xi\in K_{\pm}(\tilde{I})$, we have that $Z^+(\xi, \tilde{I}) = Z^-(\xi,\tilde{I})$, and these also only depend on $I$. We may therefore write $K_{\pm}(\tilde{I})$ as $K(I)$. In particular, for the vacuum representation we have that $H_0(I) = \A(I)\Omega$, by Haag duality. This implies that $H_0(I)\subset H_0$ is dense for all $I\in\Jcal$. Since there exist unitary operators in $\Hom_{\A_\R(\tilde{I}^{c\pm})}(H_0, H^\varphi)$, we obtain that $H_\pm^\varphi(\tilde{I})\subset H^\varphi$ are dense inclusions for all $\tilde{I}\in\Jcal_\R$. In addition, if $\tilde{I_1}\subset \tilde{I}$ is an inclusion of intervals in $\Jcal_\R$, we have inclusions $H_\pm^\varphi(\tilde{I_1})\subset H_\pm^\varphi(\tilde{I})$, which are also dense. 

Recall that the automorphism $\varphi\in \Aut(\A)$ comes equipped by definition with a unitary $V_\varphi\in U(H_0)$ implementing the action of $\varphi$ on the algebras $\A(I)$. We claim that there is an equality $H_{\pm}^\varphi(\tilde{I}) = H_\pm^\varphi(\tilde{I} + 1)$ for all $\tilde{I}$ in $\Jcal_\R$. Indeed, let $\xi = Z^\pm(\xi, \tilde{I})\Omega\in H_\pm^\varphi(\tilde{I})$. Then, the bounded linear operator
\(
Z^\pm(\xi, \tilde{I})\circ V_{\varphi}
\)
satisfies that, for all $x\in \A(I^c) = \A_{\R}(\tilde{I}^{c\pm})$
\begin{align*}
Z^\pm(\xi, \tilde{I})\circ V_{\varphi}\circ \pi_{0, I^c}(x) &= Z^\pm(\xi, \tilde{I})\circ \pi_{0, I^c}(\varphi x)\circ V_{\varphi} \\ &= \pi_{\tilde{I}^{c\pm}}^H(\varphi x)\circ Z^\pm(\xi, \tilde{I})\circ V_{\varphi} \\ &=\pi^{H}_{(\tilde{I} +1)^{c\pm}}(x)\circ Z^\pm(\xi, \tilde{I})\circ V_{\varphi}
\end{align*}
and $Z^\pm(\xi, \tilde{I})\circ V_{\varphi}(\Omega) = Z^\pm(\xi, \tilde{I})(\Omega) = \xi.$ Therefore $\xi\in H_\pm^\varphi(\tilde{I} + 1)$ and
\[
Z^\pm(\xi, \tilde{I} +1) = Z^\pm(\xi, \tilde{I})\circ V_{\varphi}.
\]
 Analogously, $H_\pm^\varphi(\tilde{I} + 1)\subseteq H_\pm^\varphi(\tilde{I})$, and we obtain the equality $H_\pm^\varphi(\tilde{I} + 1) =  H_\pm^\varphi(\tilde{I})$. Similarly, there is an equality $H^\varphi_-(\tilde{I}) = H^\varphi_+(\tilde{I})$, with 
 \begin{equation}\label{eq: RelationZ+Z-}
 Z^-(\xi, \tilde{I}) = Z^+(\xi, \tilde{I})\circ V_{\varphi^{-1}}.
 \end{equation}

Let $\mu\in \Aut(\A)$ and $K^\mu = (K^\mu, \pi^K)\in \Rep^\mu(\A)$. Let $\tilde{I}, \tilde{J}\in\Jcal_\R $ be intervals such that $\tilde{J}\subset\tilde{I}^{c+}$. We define a positive sesquilinear form $\langle - |-\rangle$ on $H_+^\varphi(\tilde{I})\otimes K_-^\mu(\tilde{J})$ by setting, for $\xi, \xi'\in H_+^\varphi(\tilde{I})$ and $\eta, \eta'\in K_-^\mu(\tilde{J})$, 
\[
\langle \xi\otimes\eta|\xi'\otimes\eta'\rangle:=\langle Z^-(\eta', \tilde{J})^*Z^-(\eta, \tilde{J})Z^+(\xi', \tilde{I})^*Z^+(\xi, \tilde{I})\Omega|\Omega\rangle.
\]
Note that $Z^+(\xi', \tilde{I})^*Z^+(\xi, \tilde{I}): H_0\to H_0$ satisfies that, for all $x\in \A(I^c) = \A_\R(\tilde{I}^{c+})$,
\begin{align*}
Z^+(\xi', \tilde{I})^*Z^+(\xi, \tilde{I})\pi_{0, I^c}(x) &= Z^+(\xi', \tilde{I})^*\pi^H_{\tilde{I}^{c+}}(x)Z^+(\xi, \tilde{I})\\
&=\pi_{0, I^c}(x) Z^+(\xi', \tilde{I})^*Z^+(\xi, \tilde{I}),
\end{align*}
and hence $Z^+(\xi', \tilde{I})^*Z^+(\xi, \tilde{I})\in \Hom_{\A(I^c)}(H_0, H_0) \cong \A(I) = \A_\R(\tilde{I})$. Similarly, we have that $Z^-(\eta', \tilde{J})^*Z^-(\eta, \tilde{J})\in \Hom_{\A(J^c)}(H_0, H_0) \cong \A(J) = \A_\R(\tilde{J})$. Hence, the maps $Z^+(\xi', \tilde{I})^*Z^+(\xi, \tilde{I})$ and $Z^-(\eta', \tilde{J})^*Z^-(\eta, \tilde{J})$ commute by locality, and the inner product above also equals
\[
\langle \xi\otimes\eta|\xi'\otimes\eta'\rangle=\langle Z^+(\xi', \tilde{I})^*Z^+(\xi, \tilde{I})Z^-(\eta', \tilde{J})^*Z^-(\eta, \tilde{J})\Omega|\Omega\rangle.
\]
The positivity of $\langle-|-\rangle$ follows from the same arguments as in the untwisted setting, see \cite[Prop. IX.3.15]{MR1943006}.

\begin{definition}\label{def: ConnesFusion}
    The Hilbert space $H_+^\varphi(\tilde{I})\boxtimes K_-^\mu(\tilde{J})$ is the completion of $H^\varphi_+(\tilde{I})\otimes K^\mu_-(\tilde{J})$~with respect to the inner product $\langle -|-\rangle$. We call it the \emph{Connes fusion of $H^\varphi, K^\mu$ over the intervals $\tilde{I}$ and $\tilde{J}.$} 
\end{definition}

We continue denoting by $\xi\otimes \eta$ the image of a vector $\xi\otimes \eta\in H_+^\varphi(\tilde{I})\otimes K_-^\mu(\tilde{J})$ in $H_+^\varphi(\tilde{I})\boxtimes K_-^\mu(\tilde{J})$. Recall that we can regard $Z^+(\xi', \tilde{I})^*Z^+(\xi, \tilde{I})$ as an element of $\A(I)$, which we denote by the same symbol, and it satisfies $Z^+(\xi', \tilde{I})^*Z^+(\xi, \tilde{I}) = \pi_{0,I}(Z^+(\xi', \tilde{I})^*Z^+(\xi, \tilde{I}))$. We have
\begin{align}\label{eq: OneSidedInnerProd}
    \langle \xi\otimes\eta|\xi'\otimes \eta'\rangle &=\langle Z^-(\eta', \tilde{J})^*Z^-(\eta, \tilde{J})Z^+(\xi', \tilde{I})^*Z^+(\xi, \tilde{I})\Omega|\Omega\rangle\\ \nonumber
    &=\langle Z^-(\eta, \tilde{J})\circ \pi_{0, I}(Z^+(\xi', \tilde{I})^*Z^+(\xi, \tilde{I}))\Omega|\eta'\rangle \\ \nonumber
    & = \langle \pi^K_{\tilde{J}^{c-}}(Z^+(\xi', \tilde{I})^*Z^+(\xi, \tilde{I}))\eta|\eta'\rangle 
    \\ & = \langle \pi^K_{\tilde{I}}(Z^+(\xi', \tilde{I})^*Z^+(\xi, \tilde{I}))\eta|\eta'\rangle  \nonumber
\end{align}
and
\[
\langle \xi\otimes\eta|\xi'\otimes \eta'\rangle  = \langle \pi^H_{\tilde{J}}(Z^-(\eta', \tilde{J})^*Z^-(\eta, \tilde{J}))\xi|\xi'\rangle.
\]
Given $\tilde{I_1},\tilde{J_1}\in\Jcal_\R$ such that $\tilde{J_1}\subset\tilde{I_1}^{c+}$ and $\tilde{I_1}\subset\tilde{I}$, $\tilde{J_1}\subset \tilde{J}$, there is a canonical equivalence
\begin{equation}\label{eq: CanonicalEquivalenceInclusion}
H^\varphi_+(\tilde{I_1})\boxtimes K_-^\mu(\tilde{J_1})\xrightarrow{\cong} H_+^\varphi(\tilde{I})\boxtimes K_-^\mu(\tilde{J})
\end{equation}
induced by the inclusion $H_+^\varphi(\tilde{I_1})\otimes K_-^\mu(\tilde{J_1})\hookrightarrow H_+^\varphi(\tilde{I})\otimes K_-^\mu(\tilde{J})$. 

\begin{definition}\label{def: FusionOfMorphisms}
    Fix automorphisms $\varphi,\mu\in \Aut(\A)$ and let $H^\varphi, \hat{H}^{\varphi}\in \Rep^\varphi(\A)$, $K^\mu, \hat{K}^{\mu}\in \Rep^\mu(\A)$ be representations of $\A$ twisted by $\varphi$ and $\mu$ respectively. Let $\tilde{I},\tilde{J}\in\Jcal_\R$ with $\tilde{J}\subset\tilde{I}^{c+}$. Let $F: H^\varphi\to \hat{H}^{\varphi}$ and $G: K^\mu\to \hat{K}^{\mu}$ be morphisms of twisted $\A$-representations. We denote by $F\boxtimes G : H^\varphi_+(\tilde{I})\boxtimes K^\mu_-(\tilde{J})\to \hat{H}^{\varphi}_+(\tilde{I})\boxtimes \hat{K}_-^{\mu}(\tilde{J})$ the bounded map induced by 
    \[
    (F\boxtimes G)(\xi\otimes\eta) = F(\xi)\otimes G(\eta)
    \]
    for all $\xi\in H^\varphi_+(\tilde{I})$ and $\eta\in K^\mu_-(\tilde{J})$.
\end{definition}

Now, given $z, \zeta\in \R$ such that $\zeta\in (z,z+1)$, we define
\[
H_+^\varphi(z)\boxtimes K_-^\mu(\zeta):=\varinjlim\limits_{(z,\zeta)\in\tilde{I}\times \tilde{J}}H_+^\varphi(\tilde{I})\boxtimes K_-^\mu(\tilde{J}) = \Bigg(\bigsqcup\limits_{(z,\zeta)\in\tilde{I}\times \tilde{J}}H_+^\varphi(\tilde{I})\boxtimes K_-^\mu(\tilde{J})\Bigg)\Big/\cong,
\]
where the limit runs over pairs $\tilde{I}, \tilde{J}\in \Jcal_\R$ satisfying $\tilde{J}\subset\tilde{I}^{c+}$ containing $z$ and $\zeta$ respectively, and $\cong$ denotes the equivalence obtained by identifying $H^\varphi_+(\tilde{I_1})\boxtimes K^\mu_-(\tilde{J_1})$ with $H^\varphi_+(\tilde{I})\boxtimes K^\mu_-(\tilde{J})$ by the unitary \eqref{eq: CanonicalEquivalenceInclusion} whenever $\tilde{I_1}\subset \tilde{I}$ and $\tilde{J_1}\subset \tilde{J}$. Given $\tilde{I}, \tilde{J}\in \Jcal_\R$ such that $\tilde{J}\subset\tilde{I}^{c+}$ containing $z$ and $\zeta$ respectively, the canonical map 
\[
H_+^\varphi(\tilde{I})\boxtimes K_-^\mu(\tilde{J})\to H_+^\varphi(z)\boxtimes K_-^\mu(\zeta)
\]
is a unitary equivalence. 

We will next relate the Connes fusions over different pairs of intervals. Given intervals $\tilde{I}_i, \tilde{J}_i\in \Jcal_\R$ with $\tilde{J_i}\subset\tilde{I}_i^{c+}$ for $i = 0,1$, and a suitable path $\gamma$ in $\R\times\R$ from a point in $\tilde{I_0}\times\tilde{J_0}$ to a point in $\tilde{I_1}\times \tilde{J_1}$, we wish to define a unitary $\gamma^\bullet: H^\varphi_+(\tilde{I_0})\boxtimes K^\mu_-(\tilde{J_0})\xrightarrow{\cong} H^\varphi_+(\tilde{I_1})\boxtimes K^\mu_-(\tilde{J_1})$. Consider the topological space $(\R\times\R)\setminus (\R\times_{S^1}\R) = \{(z,\zeta)\in \R\times\R\ |\  qz\neq q\zeta\}$, which is a disconnected space with contractible connected components. We write $\ConfR$ for the connected component of $(\R\times\R)\setminus (\R\times_{S^1}\R)$ given by the points $(z,\zeta)$ with $\zeta\in (z,z+1).$ Then, $\ConfR$ is the universal cover of $\text{Conf}_2(S^1)$, the space of configurations of two points in $S^1$. Let $\gamma = (\alpha, \beta):[0,1]\to \ConfR$ be a path in $\ConfR$ from $\gamma(0) =: (z_0, \zeta_0)$ to $\gamma(1) =: (z_1, \zeta_1)$. Assume that $\gamma$ is small enough so that $\gamma([0,1])\subset \tilde{I}\times \tilde{J}$, for some $\tilde{I}, \tilde{J}\in \Jcal_\R$ satisfying $\tilde{J}\subset \tilde{I}^{c+}$. We define the unitary equivalence $\gamma^\bullet: H_+^\varphi(z_0)\boxtimes K_-^\mu(\zeta_0)\xrightarrow{\cong}H_+^\varphi(z_1)\boxtimes K_-^\mu(\zeta_1)$ as the composition
\[
H_+^\varphi(z_0)\boxtimes K_-^\mu(\zeta_0)\xrightarrow{\cong}H_+^\varphi(\tilde{I})\boxtimes K_-^\mu(\tilde{J})\xrightarrow{\cong}H_+^\varphi(z_1)\boxtimes K_-^\mu(\zeta_1).
\]
For a generic path $\gamma$, we choose $0 = t_0<t_1<\ldots <t_k = 1$ such that each $\gamma_n:=\gamma|_{[t_n, t_{n+1}]}$ is small enough in the sense above, and define $\gamma^\bullet: H_+^\varphi(z_0)\boxtimes K_-^\mu(\zeta_0)\xrightarrow{\cong}H_+^\varphi(z_1)\boxtimes K_-^\mu(\zeta_1)$ as 
\[
\gamma^\bullet = \gamma_{k-1}^\bullet\ldots\gamma_1^\bullet\gamma_0^\bullet.
\]
Since finer partitions give the same result, the map $\gamma^\bullet$ is independent of the choice of partition. We call $\gamma^\bullet$ the \emph{path-continuation} induced by $\gamma$.

Let $\tilde{I_0}, \tilde{I_1}, \tilde{J_0}, \tilde{J_1}\in \Jcal_\R$ be intervals such that $\tilde{J_0}\subset\tilde{I_0}^{c+}$ and $\tilde{J_1}\subset\tilde{I_1}^{c+}$. Then, $\tilde{I_0}\times\tilde{J_0}$ and $\tilde{I_1}\times \tilde{J_1}$ lie in $\ConfR$. Let $\gamma$ be a path in $\ConfR$ from $\gamma(0) =: (z_0, \zeta_0)\in \tilde{I_0}\times\tilde{J_0}$ to $\gamma(1) =: (z_1, \zeta_1)\in \tilde{I_1}\times \tilde{J_1}$. We define the path continuation $\gamma^\bullet : H_+^\varphi(\tilde{I_0})\boxtimes K_-^\mu(\tilde{J_0})\xrightarrow{\cong} H_+^\varphi(\tilde{I_1})\boxtimes K_-^\mu(\tilde{J_1})$ as the composition
\[
H_+^\varphi(\tilde{I_0})\boxtimes K_-^\mu(\tilde{J_0})\xrightarrow{\cong}H_+^\varphi(z_0)\boxtimes K_-^\mu(\zeta_0)\xrightarrow{\gamma^\bullet} H_+^\varphi(z_1)\boxtimes K_-^\mu(\zeta_1)\xrightarrow{\cong}H_+^\varphi(\tilde{I_1})\boxtimes K_-^\mu(\tilde{J_1}).
\]
The next proposition follows from the same arguments as \cite[Prop. 2.11]{Gui21}.
\begin{proposition}
    Let $\gamma, \tilde{\gamma}$ be two paths in $\ConfR$ with $\gamma(0), \tilde{\gamma}(0)\in \tilde{I_0}\times\tilde{J_0}$ and $\gamma(1),\tilde{\gamma}(1)\in \tilde{I_1}\times\tilde{J_1}$. Suppose there exists a homotopy $\Gamma:[0,1]\times[0,1]\to \ConfR$ from $\gamma = \Gamma(-,0)$ to $\tilde{\gamma}  = \Gamma(-,1)$. Assume moreover that $\Gamma(0, [0,1])\subset \tilde{I_0}\times\tilde{J_0}$ and $\Gamma(1, [0,1])\subset \tilde{I_1}\times\tilde{J_1}$. Then, $\gamma^\bullet = \tilde{\gamma}^\bullet$.
\end{proposition}

In Definition \ref{def: ConnesFusion}, we have defined the Connes fusion of $H^\varphi$ and $K^\mu$ over a pair of intervals $\tilde{I},\tilde{J}\in \Jcal_\R$ with $\tilde{J}\subset \tilde{I}^{c+}$ as a certain completion of the vector space $H^\varphi_+(\tilde{I})\otimes K^\mu_-(\tilde{J})$. The same Hilbert space can be obtained as a completion of the vector spaces $H^\varphi_+(\tilde{I})\otimes K^\mu$ and $H^\varphi\otimes K^\mu_-(\tilde{J})$. We define $H_+^\varphi(\tilde{I})\boxtimes K^\mu$ to be the completion of $H_+^\varphi(\tilde{I})\otimes K^\mu$ under the inner product
\[
\langle \xi\otimes\eta\,|\,\xi'\otimes \eta'\rangle := \langle\pi^K_{\tilde{I}}\big(Z^+(\xi', \tilde{I})^*Z^+(\xi, \tilde{I})\big)(\eta)\,|\, \eta'\rangle
\]
for $\xi, \xi'\in H_+^\varphi(\tilde{I})$ and $\eta, \eta'\in K^\mu$, and call it the \emph{Connes fusion of $H^\varphi$ and $K^\mu$ over $\tilde{I}$ on the left}. Then, if $\tilde{J}\in \Jcal_\R$ is an interval such that $\tilde{J}\subset \tilde{I}^{c+}$, the inclusion $H_+^\varphi(\tilde{I})\otimes K_-^\mu(\tilde{J})\hookrightarrow H^\varphi_+(\tilde{I})\otimes K^\mu$ induces a canonical equivalence $H_+^\varphi(\tilde{I})\boxtimes K_-^\mu(\tilde{J})\xrightarrow{\cong} H_+^\varphi(\tilde{I})\boxtimes K^\mu$, by Equation \eqref{eq: OneSidedInnerProd}. For any $z\in \R$, we define $H^\varphi_+(z)\boxtimes K^\mu := \varinjlim\limits_{z\in\tilde{I}}\, H^\varphi_+(\tilde{I})\boxtimes K^\mu$. Given a path $\alpha$ in $\R$ from $z_0$ to $z_1$, we also obtain a path continuation $\alpha^\bullet: H^\varphi_+(z_0)\boxtimes K^\mu\xrightarrow{\cong} H^\varphi_+(z_1)\boxtimes K^\mu$ as follows. If $\alpha$ is small enough so that $\alpha([0,1])$ can be covered by an interval $\tilde{I}\in \Jcal_\R$, we define
\[
\alpha^\bullet: H^\varphi_+(z_0)\boxtimes K^\mu\xrightarrow{\cong} H^\varphi_+(\tilde{I})\boxtimes K^\mu\xrightarrow{\cong} H^\varphi_+(z_1)\boxtimes K^\mu.
\]
For a generic path $\alpha$, we pick a partition $0 = t_0<t_1<\ldots<t_k = 1$ of $[0,1]$ so that each $\alpha_n:=\alpha|_{[t_n, t_{n+1}]}$ is small enough in the sense above, and define $\alpha^\bullet = \alpha_{k-1}^\bullet\ldots\alpha_0^\bullet$. Given two intervals $\tilde{I_0},\tilde{I_1}\in \Jcal_\R$ and a path $\alpha$ in $\R$ from a point in $\tilde{I_0}$ to a point in $\tilde{I_1}$, we define
\[
\alpha^\bullet:H^\varphi_+(\tilde{I_0})\boxtimes K^\mu\xrightarrow{\cong } H_+^\varphi(\alpha(0))\boxtimes K^\mu\xrightarrow{\alpha^\bullet} H^\varphi_+(\alpha(1))\boxtimes K^\mu\xrightarrow{\cong} H^\varphi_+(\tilde{I_1})\boxtimes K^\mu.
\]
Homotopic paths induce the same unitary equivalence. The following result follows from the same arguments as \cite[Prop. 2.12]{Gui21}.

\begin{proposition}\label{prop: SinglePathContinuationWRTDoublePathContinuation}
    Let $\tilde{I_0},\tilde{I_1},\tilde{J_0},\tilde{J_1}\in \Jcal_\R$ be such that $\tilde{J_i}\subset \tilde{I_i}^{c+}$ for $i = 0,1$. Let $\gamma = (\alpha, \beta)$ be a path in $\ConfR$ from $\tilde{I_0}\times\tilde{J_0}$ to $\tilde{I_1}\times \tilde{J_1}$. Then, the path-continuation $\alpha^\bullet: H_+^\varphi(\tilde{I_0})\boxtimes K^\mu\xrightarrow{\cong} H_+^\varphi(\tilde{I_1})\boxtimes K^\mu$ is given by
    \[
    H_+^\varphi(\tilde{I_0})\boxtimes K^\mu\xrightarrow{\cong} H_+^\varphi(\tilde{I_0})\boxtimes K_-^\mu(\tilde{J_0})\xrightarrow{\gamma^\bullet} H_+^\varphi(\tilde{I_1})\boxtimes K_-^\mu(\tilde{J_1})\xrightarrow{\cong} H_+^\varphi(\tilde{I_1})\boxtimes K^\mu.
    \]
\end{proposition}
Similar properties hold for  $H^\varphi\boxtimes K^\mu_-(\tilde{J})$, the \emph{Connes fusion of $H^\varphi$ and $K^\mu$ over $\tilde{J}$ on the right}, defined as the completion of $H^\varphi\otimes K^\mu_-(\tilde{J})$ with respect~to
\[
\langle \xi\otimes\eta\,|\,\xi'\otimes\eta'\rangle:=\langle \pi^H_{\tilde{J}}\big(Z^-(\eta', \tilde{J})^*Z^-(\eta, \tilde{J})\big)(\xi)\,|\,\xi'\rangle
\]
for $\xi, \xi'\in H^\varphi$ and $\eta,\eta'\in K_-^\mu(\tilde{J})$.

\subsection{Endowing the Connes fusion with a twisted action}

From now on, we fix intervals $\tilde{I},\tilde{J}\in \Jcal_\R$ such that $\tilde{J}\subset \tilde{I}^{c+}$, as well as automorphisms $\varphi,\mu\in\Aut(\A)$ and twisted representations $H^\varphi = (H^\varphi, \pi^H)\in \Rep^\varphi(\A)$ and $K^\mu = (K^\mu, \pi^K)\in \Rep^\mu(\A)$. We shall endow the Hilbert space $H^\varphi_+(\tilde{I})\boxtimes K^\mu_-(\tilde{J})$ with a $(\varphi\circ\mu)$-twisted action of $\A$. We will do this by endowing it with compatible actions $\pi^{H\boxtimes K}_{\tilde{L}}$ of the algebras $\A_\R(\tilde{L})$ such that $\pi^{H\boxtimes K}_{\tilde{L}} \circ (\varphi\circ\mu) = \pi^{H\boxtimes K}_{\tilde{L}+1}$ for all $\tilde{L}\in \Jcal_\R$. Note that we have natural actions of $\A_\R(\tilde{I} )$ and $\A_\R(\tilde{J})$ on $H^\varphi(\tilde{I})\boxtimes K^\mu(\tilde{J})$. Indeed, given $x\in \A(I) = \A_\R(\tilde{I})$ and $y\in \A(J) = \A_\R(\tilde{J})$, we can define
\[
\begin{array}{lll}
(\pi^{H}_{\tilde{I}}(x)\boxtimes\id)(\xi\otimes \eta) &= \Big(\pi_{\tilde{I}}^H(x)(\xi)\Big)\otimes \eta\\
(\id\boxtimes \pi^{K}_{\tilde{J} }(y))(\xi\otimes\eta) &= \xi\otimes \Big(\pi_{\tilde{J}}^K(y)(\eta)\Big)
\end{array}
\]
for $\xi\otimes\eta \in H_+^\varphi(\tilde{I})\otimes K_-^\mu(\tilde{J})\subset H_+^\varphi(\tilde{I})\boxtimes K_-^\mu(\tilde{J}) $. This also provides canonical actions of all $\A_\R(\tilde{L})$ for $\tilde{L}\subset \tilde{I}$ or $\tilde{L}\subset\tilde{J}$.

\begin{proposition}\label{prop: SwapIntervals}
    Let $\gamma$ be a path in $\ConfR$ from $\tilde{I}\times \tilde{J}$ to $\tilde{J}\times (\tilde{I} +1)$, and let $\vartheta$ be a path in $\ConfR$ from $\tilde{I}\times \tilde{J}$ to $(\tilde{J}-1)\times \tilde{I}$. Then,
    \begin{enumerate}
        \item for all $x\in \A(I)$, the following diagrams commute
\[\begin{tikzcd}
	{H^\varphi_+(\tilde{I})\boxtimes K_-^\mu(\tilde{J})} & {H^\varphi_+(\tilde{J})\boxtimes K^\mu_-(\tilde{I}+1)} \\
	{H^\varphi_+(\tilde{I})\boxtimes K_-^\mu(\tilde{J})} & {H^\varphi_+(\tilde{J})\boxtimes K^\mu_-(\tilde{I}+1)} \\
	{H^\varphi_+(\tilde{I})\boxtimes K_-^\mu(\tilde{J})} & {H^\varphi_+(\tilde{J}- 1)\boxtimes K^\mu_-(\tilde{I})} \\
	{H^\varphi_+(\tilde{I})\boxtimes K_-^\mu(\tilde{J})} & {H^\varphi_+(\tilde{J}- 1)\boxtimes K^\mu_-(\tilde{I})};
	\arrow["{\gamma^\bullet}", from=1-1, to=1-2]
	\arrow["{\pi^H_{\tilde{I}}(\varphi\mu x)\boxtimes\id}"', from=1-1, to=2-1]
	\arrow["{\id\boxtimes\pi^K_{\tilde{I}+1}(x)}", from=1-2, to=2-2]
	\arrow["{\gamma^\bullet}"', from=2-1, to=2-2]
	\arrow["{\vartheta^\bullet}", from=3-1, to=3-2]
	\arrow["{\pi^H_{\tilde{I}}(x)\boxtimes\id}"', from=3-1, to=4-1]
	\arrow["{\id\boxtimes\pi^K_{\tilde{I}}(x)}", from=3-2, to=4-2]
	\arrow["{\vartheta^\bullet}"', from=4-1, to=4-2]
\end{tikzcd}\]
\item for all $y\in \A(J)$, the following diagrams commute
\[\begin{tikzcd}
	{H^\varphi_+(\tilde{I})\boxtimes K_-^\mu(\tilde{J})} & {H^\varphi_+(\tilde{J})\boxtimes K^\mu_-(\tilde{I}+1)} \\
	{H^\varphi_+(\tilde{I})\boxtimes K_-^\mu(\tilde{J})} & {H^\varphi_+(\tilde{J})\boxtimes K^\mu_-(\tilde{I}+1)}.
	\arrow["{\gamma^\bullet}", from=1-1, to=1-2]
	\arrow["{\id\boxtimes \pi^K_{\tilde{J}}(y)}"', from=1-1, to=2-1]
	\arrow["{\pi^H_{\tilde{J}}(y)\boxtimes\id}", from=1-2, to=2-2]
	\arrow["{\gamma^\bullet}"', from=2-1, to=2-2]
\end{tikzcd}\]
\[\begin{tikzcd}
	{H^\varphi_+(\tilde{I})\boxtimes K_-^\mu(\tilde{J})} & {H^\varphi_+(\tilde{J}- 1)\boxtimes K^\mu_-(\tilde{I})} \\
	{H^\varphi_+(\tilde{I})\boxtimes K_-^\mu(\tilde{J})} & {H^\varphi_+(\tilde{J}- 1)\boxtimes K^\mu_-(\tilde{I})}
	\arrow["{\vartheta^\bullet}", from=1-1, to=1-2]
	\arrow["{\id\boxtimes \pi^K_{\tilde{J}}(y)}"', from=1-1, to=2-1]
	\arrow["{\pi^H_{\tilde{J}-1}(\varphi\mu y)\boxtimes \id}", from=1-2, to=2-2]
	\arrow["{\vartheta^\bullet}"', from=2-1, to=2-2]
\end{tikzcd}\]
    \end{enumerate}
\end{proposition}
\begin{proof}
    This is essentially the generalization of \cite[Prop. 2.13]{Gui21}, but we include the proof for completeness. We prove the commutativity of the first diagram. The other three follow from analogous arguments.
    
    Firstly, since by additivity we have that $\A(I) = \bigvee\limits_{L\Subset I}\A(L)$ (here, $L \Subset I$ means that the closure of $L$ lies in $I$) it suffices to verify the statement for all $L\Subset I$ and $x\in \A(L)$. Pick one such $L$ and $x\in \A(L)$. We denote by $\tilde{L}$ the lift of $L$ to $\R$ lying in $\tilde{I}+1$
. 
Since $\ConfR$ is contractible, $\gamma$ is homotopic to a path $(\alpha, \beta)$ such that both $\alpha([0,1])$ and $\beta([0,1])$ can be covered by intervals in $\Jcal_\R$. Since the path continuation is invariant under homotopy, let us assume that $\gamma$ is of this form, and denote $\gamma= (\alpha, \beta)$. Choose a pair of intervals $\tilde{I_1}, \tilde{J_1}$ in $\Jcal_\R$ such that
\begin{enumerate}
    \item $\tilde{I}\cap \tilde{I_1},\tilde{I_1}\cap\tilde{J},\tilde{J}\cap\tilde{J_1},\tilde{J_1}\cap (\tilde{I}+1)\in\Jcal_\R$
    \item $\tilde{L}-1\subset \tilde{I_1}$
    \item $\alpha([0,1])$ is covered by $\tilde{I}\cup \tilde{I_1}\cup \tilde{J}$ and $\beta([0,1])$ is covered by $\tilde{J}\cup\tilde{J_1}\cup(\tilde{I}+1)$.
    \item $\tilde{J_1}\cup \tilde{J}\cup\tilde{I_1}\in \Jcal_\R$.
\end{enumerate}

\begin{figure}[htb]
    \centering
    \resizebox{.5\linewidth}{!}{
    \begin{tikzpicture}[line cap=round, line join=round]
  \def\R{6.0}    
  \def\ry{1.2}   
  \def\pitch{1.7}
  \def\turns{2}  
  \pgfmathsetmacro{\T}{360*\turns}

  \draw[line width=5pt]
    plot[samples=700, domain=100:\T, variable=\t]
      ({- \R*cos(\t)}, {\ry*sin(\t) - \pitch*(\t/360)});

    \def\tA{295}            
\def\tB{615}            
\def\Icol{pink!50!blue} 

\draw[line width=3pt, \Icol,  latex-, shorten >=3.5pt]
  plot[samples=200, domain=\tA:\tB, variable=\t]
    ({-\R*cos(\t)}, {\ry*sin(\t) - \pitch*(\t/360)-.3});

\pgfmathsetmacro{\tm}{310}
\node[\Icol!80!black]
  at ({-\R*cos(\tm)}, {\ry*sin(\tm) - \pitch*(\tm/360) - 1.05})
  {\huge $\alpha$};

\def\tA{230}            
\def\tB{270}            
\def\Icol{red}

\draw[line width=7pt, \Icol]
  plot[samples=200, domain=\tA:\tB, variable=\t]
    ({-\R*cos(\t)}, {\ry*sin(\t) - \pitch*(\t/360)});

\pgfmathsetmacro{\tm}{0.5*(.18*\tA+1.75*\tB)}
\node[\Icol!80!black]
  at ({-\R*cos(\tm)}, {\ry*sin(\tm) - \pitch*(\tm/360) + 0.60})
  {\huge $\tilde{I}+1$};

\def\tA{230+360}            
\def\tB{270+360}   

\draw[line width=7pt, \Icol]
  plot[samples=200, domain=\tA:\tB, variable=\t]
    ({-\R*cos(\t)}, {\ry*sin(\t) - \pitch*(\t/360)});

\pgfmathsetmacro{\tm}{0.5*(.8*\tA+1.2*\tB)}
\node[\Icol!80!black]
  at ({-\R*cos(\tm)}, {\ry*sin(\tm) - \pitch*(\tm/360) + 0.65})
  {\huge $\tilde{I}$};

\def\tA{280}            
\def\tB{310}            
\def\Icol{blue} 

\draw[line width=7pt, \Icol]
  plot[samples=200, domain=\tA:\tB, variable=\t]
    ({-\R*cos(\t)}, {\ry*sin(\t) - \pitch*(\t/360)});

\pgfmathsetmacro{\tm}{0.5*(\tA+\tB)}
\node[\Icol!80!black]
  at ({-\R*cos(\tm)}, {\ry*sin(\tm) - \pitch*(\tm/360) + 0.65})
  {\huge $\tilde{J}$};

\def\tA{250}            
\def\tB{300}            
\def\Icol{green!60!black} 

\draw[line width=7pt, \Icol]
  plot[samples=200, domain=\tA:\tB, variable=\t]
    ({-\R*cos(\t)}, {\ry*sin(\t) - \pitch*(\t/360)-.1});

\pgfmathsetmacro{\tm}{0.5*(\tA+\tB)}
\node[\Icol!80!black]
  at ({-\R*cos(\tm)}, {\ry*sin(\tm) - \pitch*(\tm/360) - 0.75})
  {\huge $\tilde{J_1}$};

\def\tA{240+360}            
\def\tB{305}            
\def\Icol{orange} 

\draw[line width=7pt, \Icol]
  plot[samples=200, domain=\tA:\tB, variable=\t]
    ({-\R*cos(\t)}, {\ry*sin(\t) - \pitch*(\t/360)+.1});

\pgfmathsetmacro{\tm}{0.5*(\tA+\tB)}
\node[\Icol!80!black]
  at ({-\R*cos(\tm)}, {\ry*sin(\tm) - \pitch*(\tm/360) + 0.75})
  {\huge $\tilde{I_1}$};

  \def\tA{235}            
\def\tB{240}            
\def\Icol{purple!60!black} 

\draw[line width=7pt, \Icol]
  plot[samples=200, domain=\tA:\tB, variable=\t]
    ({-\R*cos(\t)}, {\ry*sin(\t) - \pitch*(\t/360)-.1});

\pgfmathsetmacro{\tm}{0.5*(\tA+\tB)}
\node[\Icol!80!black]
  at ({-\R*cos(\tm)}, {\ry*sin(\tm) - \pitch*(\tm/360) - 0.75})
  {\huge $\tilde{L}$};

  \def\tA{210}            
\def\tB{225}            
\def\Icol{black} 

\draw[line width=5pt, \Icol]
  plot[samples=200, domain=\tA:\tB, variable=\t]
    ({-\R*cos(\t)}, {\ry*sin(\t) - \pitch*(\t/360)});

  \def\tA{269}            
\def\tB{283}            
\def\Icol{pink!50!blue} 

\draw[line width=3pt, \Icol,  latex-, shorten >=3.5pt]
  plot[samples=200, domain=\tA:\tB, variable=\t]
    ({-\R*cos(\t)}, {\ry*sin(\t) - \pitch*(\t/360)+.4});

\pgfmathsetmacro{\tm}{0.5*(\tA+\tB)}
\node[\Icol!80!black]
  at ({-\R*cos(\tm)}, {\ry*sin(\tm) - \pitch*(\tm/360) + 0.75})
  {\huge $\beta$};

\end{tikzpicture}}
\caption{}
    \label{fig:intervals}
\end{figure}

See Figure \ref{fig:intervals}. Consider the unitary maps 
\begin{align*}
    R: &\ H^\varphi_+(\tilde{I})\boxtimes K_-^\mu(\tilde{J})\xrightarrow{\cong} H^\varphi_+(\tilde{I}\cap \tilde{I_1})\boxtimes K^\mu_-(\tilde{J}\cap \tilde{J_1})\xrightarrow{\cong} H^\varphi_+(\tilde{I_1})\boxtimes K^\mu_-(\tilde{J_1}) \\
    S:&\ H^\varphi_+(\tilde{I_1})\boxtimes K^\mu_-(\tilde{J_1}) \xrightarrow{\cong } H^\varphi_+(\tilde{I_1}\cap \tilde{J})\boxtimes K^\mu_-(\tilde{J_1}\cap (\tilde{I}+1))\xrightarrow{\cong} H^\varphi_+(\tilde{J})\boxtimes K_-^\mu(\tilde{I}+1).
\end{align*}
By the straightforward generalization of \cite[Lem. 2.10]{Gui21} to our case, we have that $\gamma^\bullet = SR$, and by \cite[Lem. 2.14]{Gui21}, $R$ intertwines the actions of $x$ on the domain and the target. Hence, it is enough to show that $S^*(\id\boxtimes \pi^K_{\tilde{I}+1}(x))S = \pi^H_{\tilde{I}}(\varphi\mu x)\boxtimes\id$. Let $\xi,\xi'\in H^\varphi_+(\tilde{I_1}\cap\tilde{J})$ and $\eta,\eta'\in K^\mu_-(\tilde{J_1}\cap(\tilde{I}+1))$. Note that such $\xi\otimes\eta$ span a dense subspace of $H^\varphi_+(\tilde{I_1})\boxtimes K^\mu_-(\tilde{J_1})$, and that  $S(\xi\otimes \eta) = \xi\otimes \eta$, and $S(\xi'\otimes \eta') = \xi'\otimes \eta'$. Let us compute
\begin{align*}
    \langle S^*(\id\boxtimes \pi^K_{\tilde{I}+1}(x))S(\xi\otimes\eta)\,|\,\xi'\otimes\eta'\rangle & = \langle \xi\otimes \big(\pi_{\tilde{I}+1}^K(x)(\eta)\big)\,|\,\xi'\otimes\eta'\rangle \\ & = \langle \pi^K_{\tilde{J}}\big(Z^+(\xi', \tilde{J})^*Z^+(\xi, \tilde{J})\big)\circ \pi^K_{\tilde{L}}(x)(\eta)\,|\,\eta'\rangle\\ & = \langle \pi^K_{\tilde{J_1}\cup\tilde{J}\cup \tilde{I_1}} \big(Z^+(\xi', \tilde{J})^*Z^+(\xi, \tilde{J})\big )\circ \pi_{\tilde{L}-1}^K(\mu x)(\eta)\,|\,\eta'\rangle
    \\ & = \langle \pi^K_{\tilde{J_1}\cup\tilde{J}\cup \tilde{I_1}} \big(Z^+(\xi', \tilde{J})^*Z^+(\xi, \tilde{J})\,\mu x)(\eta)\,|\,\eta'\rangle.
\end{align*}
It follows from the properties of $Z^+(\xi, \tilde{J}) = Z^+(\xi,\tilde{I_1}\cap\tilde{J}) = Z^+(\xi, \tilde{I_1})$ and the fact that $\tilde{L}$ is disjoint from $\tilde{I_1}\cap\tilde{J}$ that $Z^+(\xi,\tilde{J})\,\mu x = Z^+(\pi^H_{\tilde{L}}(\mu x)(\xi), \tilde{I}) = Z^+(\pi^H_{\tilde{L}-1}(\varphi\mu x)(\xi), \tilde{I})$. Therefore, the inner product above equals
\begin{align*}
\langle \pi^K_{\tilde{I_1}} \big(Z^+(\xi', \tilde{I_1})^*Z^-(\pi^H_{\tilde{L}-1}(\varphi\mu x)(\xi), \tilde{I_1})\big)(\eta)\,|\,\eta'\rangle & = \langle \pi^H_{\tilde{L}-1}(\varphi\mu x)(\xi)\otimes \eta\,|\,\xi'\otimes\eta'\rangle,
\end{align*}
and the commutativity of the first diagram follows.
\end{proof}

We can now equip the Connes fusion $H^\varphi_+(\tilde{I})\boxtimes K^\mu_-(\tilde{J})$ with a $(\varphi\circ\mu)$-twisted action of $\A$.

\begin{theorem}\label{thm: RepresentationOnFusion}
    Let $\varphi,\mu\in \Aut(\A)$ be automorphisms and $H^\varphi$ and $K^\mu$ be $\varphi$- and $\mu$-twisted $\A$-representations respectively. Fix intervals $\tilde{I},\tilde{J}\in\Jcal_\R$ with $\tilde{J}\subset\tilde{I}^{c+}$. Then,
    \begin{enumerate}
        \item there is a unique compatible collection $\pi^l$ of actions of the $*$-algebras $\A_\R(\tilde{L})$ on $H^\varphi_+(\tilde{I})\boxtimes K^\mu_-(\tilde{J})$ such that the following condition holds: for every $\tilde{L},\tilde{M}\in \Jcal_\R$ with $\tilde{M}\subset \tilde{L}^{c+}$, any path $\gamma$ in $\ConfR$ from $\tilde{I}\times\tilde{J}$ to $\tilde{L}\times\tilde{M}$, it holds that
        \begin{equation}\label{eq: leftRep}
        \pi_{\tilde{L}}^l(x) = (\gamma^\bullet)^{-1} (\pi_{\tilde{L}}^{H}(x)\boxtimes \id)\gamma^\bullet
        \end{equation}
        for all $x\in \A_\R(\tilde{L})$;
        \item  there is a unique compatible collection $\pi^r$ of actions of the $*$-algebras $\A_\R(\tilde{L})$ on $H^\varphi_+(\tilde{I})\boxtimes K^\mu_-(\tilde{J})$ such that the following condition holds: for every $\tilde{L},\tilde{M}\in \Jcal_\R$ with $\tilde{L}\subset \tilde{M}^{c+}$, any path $\vartheta$ in $\ConfR$ from $\tilde{I}\times\tilde{J}$ to $\tilde{M}\times\tilde{L}$, it holds that
        \begin{equation}\label{eq: rightRep}
        \pi^r_{\tilde{L}}(x) = (\vartheta^\bullet)^{-1} (\id \boxtimes \pi_{\tilde{L}}^{K}(x))\vartheta^\bullet
        \end{equation}
        for all $x\in \A_\R(\tilde{L})$;
        \item the actions $\pi^l$ and $\pi^r$ agree, and hence there is a canonical natural compatible collection $\pi^{H\boxtimes K}$ of actions of the algebras $\A_\R(\tilde{L})$ on $H^\varphi_+(\tilde{I})\boxtimes K^\mu_-(\tilde{J})$;
    \item the unitary maps induced by inclusions of intervals, restrictions of intervals, and path continuations between Connes fusions intertwine the actions $\pi^{H\boxtimes K}$;
    \item the actions $\pi^{H\boxtimes K}$ satisfy that, for all $\tilde{L}\in \Jcal_\R$ and $x\in \A(L)$, 
    \[
    \pi^{H\boxtimes K}_{\tilde{L}}(\varphi\mu x) = \pi^{H\boxtimes K}_{\tilde{L}+1}(x),
    \]
    hence providing a $\varphi\mu$-twisted representation $\pi^{H\boxtimes K}$ of $\A$ on $H^\varphi_+(\tilde{I})\boxtimes K_-^\mu(\tilde{J})$.
    \end{enumerate}
\end{theorem}
\begin{proof}
Choose $\tilde{L},\tilde{M}\in \Jcal_\R$ with $\tilde{M}\subset \tilde{L}^{c+}$. Then, $\tilde{L}\subset (\tilde{M} - 1)^{c+}$. Let $\gamma$ and $\vartheta$ be paths in $\ConfR$ from $\tilde{I}\times \tilde{J}$ to $\tilde{L}\times\tilde{M}$ and $(\tilde{M}-1)\times\tilde{L}$ respectively. Then, $\vartheta\circ\gamma^{-1}$ is a path in $\ConfR$ from $\tilde{L}\times\tilde{M}$ to $(\tilde{M}-1)\times\tilde{L}$, and we can apply Proposition \ref{prop: SwapIntervals} to obtain
\begin{equation}\label{eq: GammaVartheta}
(\gamma^\bullet)^{-1}\circ\big(\pi^H_{\tilde{L}}(x)\boxtimes \id\big)\circ \gamma^\bullet = (\vartheta^\bullet)^{-1}\circ \big(\id\boxtimes\pi^K_{\tilde{L}}(x)\big)\circ \vartheta^\bullet.
\end{equation}
We define the action $\pi^l$ using Equation \eqref{eq: leftRep}, which is independent of the chosen $\gamma$ by Equation \eqref{eq: GammaVartheta}. We argue similarly for $\pi^r$. Then, $\pi^l = \pi^r$ follows from \eqref{eq: GammaVartheta}. The fact that $\pi^{H\boxtimes K} := \pi^r = \pi^l$ commutes with the maps induced by inclusion and restriction of intervals and path continuations is obvious. Let us prove the last statement. Let $\gamma_1$ be a path from $\tilde{L}\times\tilde{M}$ to $(\tilde{M}-1)\times \tilde{L}$ and $\gamma_2$ a path from $(\tilde{M}-1)\times \tilde{L}$ to $(\tilde{L}-1)\times(\tilde{M}-1)$. We write $\gamma_2*\gamma_1$ for the concatenation of $\gamma_1$ and $\gamma_2$. Then, $(\gamma_2*\gamma_1)^\bullet = \gamma_2^\bullet\gamma_1^\bullet$. By Proposition \ref{prop: SwapIntervals}, the two squares in 
\[\begin{tikzcd}
	{H^\varphi_+(\tilde{L})\boxtimes K^\mu_-(\tilde{M})} & {H^\varphi_+(\tilde{M}-1)\boxtimes K^\mu_-(\tilde{L})} & {H^\varphi_+(\tilde{L}-1)\boxtimes K^\mu_-(\tilde{M}-1)} \\
	{H^\varphi_+(\tilde{L})\boxtimes K^\mu_-(\tilde{M})} & {H^\varphi_+(\tilde{M}-1)\boxtimes K^\mu_-(\tilde{L})} & {H^\varphi_+(\tilde{L}-1)\boxtimes K^\mu_-(\tilde{M}-1)}
	\arrow["{\gamma_1^\bullet}", from=1-1, to=1-2]
	\arrow["{\tilde{\gamma}^\bullet}"', curve={height=-24pt}, from=1-1, to=1-3]
	\arrow["{\pi^H_{\tilde{L}}(x)\boxtimes\id}"', from=1-1, to=2-1]
	\arrow["{\gamma_2^\bullet}", from=1-2, to=1-3]
	\arrow["{\id\boxtimes\pi^K_{\tilde{L}}(x)}", from=1-2, to=2-2]
	\arrow["{\pi_{\tilde{L}-1}^H(\varphi\mu x)\boxtimes\id}", from=1-3, to=2-3]
	\arrow["{\gamma_1^\bullet}"', from=2-1, to=2-2]
	\arrow["{\tilde{\gamma}^\bullet}"', curve={height=24pt}, from=2-1, to=2-3]
	\arrow["{\gamma_2^\bullet}"', from=2-2, to=2-3]
\end{tikzcd}\]
commute for all $x\in \A(L)$, making the whole diagram commute. Therefore
\[\hspace{-.4cm}
\pi^{H\boxtimes K}_{\tilde{L}-1}(\varphi\mu x) = ((\gamma_2*\gamma_1*\gamma)^\bullet)^{-1}\circ(\pi_{\tilde{L}-1}^H(\varphi\mu x)\boxtimes \id)\circ(\gamma_2*\gamma_1*\gamma)^\bullet = (\gamma^\bullet)^{-1}\circ(\pi_{\tilde{L}}^H(x)\boxtimes\id)\circ\gamma^\bullet  = \pi^{H\boxtimes K}_{\tilde{L}}(x),
\]
where we use that the concatenation $\gamma_2*\gamma_1*\gamma$ is a path in $\ConfR$ from $\tilde{I}\times\tilde{J}$ to $(\tilde{L} - 1)\times(\tilde{M}-1)$.
\end{proof}

Let $\hat{H}^{\varphi}\in \Rep^\varphi(\A)$ and $ \hat{K}^{\mu}\in \Rep^\mu(\A)$ be representations of $\A$ twisted by $\varphi$ and $\mu$ respectively. Recall that in Definition \ref{def: FusionOfMorphisms}, given morphisms of twisted $\A$-representations $F:H^\varphi\to \hat{H}^\varphi$ and $G:K^\mu\to \hat{K}^\mu$ we have defined a bounded map
\[
F\boxtimes G: H^\varphi_+(\tilde{I})\boxtimes K^\mu_-(\tilde{J})\to \hat{H}^\varphi_+(\tilde{I})\boxtimes \hat{K}^\mu_-(\tilde{J}).
\]
By the same arguments as \cite[Prop. 2.17]{Gui21}, $F\boxtimes G$ commutes with canonical equivalences given by inclusion and restriction of intervals, and path continuations, and is a morphism of twisted $\A$-representations.

We can define similarly twisted actions of $\A$ on fusions over single intervals. Let $\tilde{I}\in \Jcal_\R$. If $\tilde{L}\in \Jcal_\R$ is a subinterval $\tilde{L}\subset \tilde{I}$, we define the action of $x\in\A_\R(\tilde{L})$ on $H_+^\varphi(\tilde{I})\boxtimes K^\mu$ by $\pi_{\tilde{L}}^{H\boxtimes K}(\xi\otimes \eta) = \pi_{\tilde{L}}^H(x)(\xi)\otimes \eta$ for all $\xi\in H_+^\varphi(\tilde{I})$ and $\eta\in K^\mu$. For a generic interval $\tilde{L}\in\Jcal_\R$, we choose a path $\alpha$ in $\R$ from $\tilde{I}$ to $\tilde{L}$, and define $\pi_{\tilde{L}}^{ H\boxtimes K}(x) = (\alpha^\bullet)^{-1}\circ \big(\pi_{\tilde{L}}^H(x)\boxtimes\id\big)\circ\alpha^\bullet$. This is independent of the path chosen. Alternatively, one can define the action of $\A_\R$ on $H^\varphi_+(\tilde{I})\boxtimes K^\mu$ via the canonical equivalence $H^\varphi_+(\tilde{I})\boxtimes K^\mu_-(\tilde{J})\cong H^\varphi_+(\tilde{I})\boxtimes K^\mu$ for some $\tilde{J}\in \Jcal_\R$ with $\tilde{J}\subset\tilde{I}^{c+}$ and the action of $\A_\R$ on $H^\varphi_+(\tilde{I})\boxtimes K^\mu_-(\tilde{J})$ from Theorem \ref{thm: RepresentationOnFusion}. These two actions agree.

We finally consider the tensor product with the vacuum representation. We define the linear map
\[
\begin{array}{cccc}
    \natural_H: & H_+^\varphi(\tilde{I})\otimes H_0 & \to& H^\varphi \\
     & \xi\otimes \pi_0(x)(\Omega) &\mapsto & \pi_{\tilde{I}^{c+}}^H(x)(\xi)
\end{array}
\]
for all $\xi\in H^\varphi_+(\tilde{I})$ and $x\in \A(I^c)$. Note that we drop $\tilde{I}$ from the notation of $\natural_H$. Then, $\natural_H$ extends to a unitary map $\natural_H: H_+^\varphi(\tilde{I})\boxtimes H_0 \xrightarrow{\cong} H^\varphi$, which commutes with inclusions and restrictions of intervals, path continuations, and the actions of $\A$. Hence, we obtain the following result, analogously to \cite[Thm. 2.18]{Gui21}.
\begin{theorem}\label{thm: PreUnitors}
There exists a unique unitary isomorphism of twisted $\A$-representations $\natural_H: H_+^\varphi(\tilde{I})\boxtimes H_0 \xrightarrow{\cong} H^\varphi$ such that
\[
\natural_H(\xi\otimes \pi_0(x)(\Omega)) = \pi_{\tilde{I}^{c+}}^H(x)(\xi)
\]
for all $\xi\in H^\varphi_+(\tilde{I})$ and $x\in \A(I^c)$. Moreover, $\natural_H$ commutes with path-continuations.
\end{theorem}

\subsection{Associativity}\label{Sec: Associativity}

In this section, we consider the Connes fusion of more than two twisted representations of $\A$. We restrict to the case of 3 representations. Fix $\varphi, \mu, \nu\in\Aut(\A)$ and let $H^\varphi, K^\mu, R^\nu$ be representations of $\A$ twisted by $\varphi, \mu$ and $\nu$ respectively. Let $\tilde{I},\tilde{J}\in \Jcal_\R$ be such that $\tilde{J}\subset \tilde{I}^{c+}$. Define on $H^\varphi_+(\tilde{I})\otimes K^\mu\otimes R^\nu_-(\tilde{J})$ the positive sesquilinear form
\[
\langle \xi\otimes\eta\otimes \chi\,|\,\xi'\otimes\eta'\otimes\chi'\rangle := \langle \pi_{\tilde{I}}^K\big(Z^+(\xi', \tilde{I})^*Z^+(\xi, \tilde{I})\big)\pi_{\tilde{J}}^K\big(Z^-(\chi', \tilde{J})^*Z^-(\chi, \tilde{J})\big)(\eta)\,|\,\eta'\rangle.
\]
We let $H_+^\varphi(\tilde{I})\boxtimes K^\mu\boxtimes R_-^\nu(\tilde{J})$ denote the Hilbert space completion of $H_+^\varphi(\tilde{I})\otimes K^\mu\otimes R_-^\nu(\tilde{J})$ with respect to the positive form above. The Connes fusion $H_+^\varphi(\tilde{I})\boxtimes K^\mu\boxtimes R_-^\nu(\tilde{J})$ admits a $\varphi\mu\nu$-twisted action $\pi^{H\boxtimes K\boxtimes R}$ of $\A$ defined as follows. Let $\tilde{L}\in \Jcal_\R$ and $x\in \A_\R(\tilde{L}).$ Pick $\tilde{M}\in \Jcal_\R$ with $\tilde{M}\subset\tilde{L}^{c+}$ and $\gamma: [0,1]\to \ConfR$ a path from $\tilde{I}\times\tilde{J}$ to $\tilde{L}\times\tilde{M}$. Analogously to the definition of path-continuations in the previous section, $\gamma$ provides a unitary equivalence
\[
\gamma^\bullet: H_+^\varphi(\tilde{I})\boxtimes K^\mu\boxtimes R_-^\nu(\tilde{J})\xrightarrow{\cong} H_+^\varphi(\tilde{L})\boxtimes K^\mu\boxtimes R_-^\nu(\tilde{M})
\]
also called a path-continuation. The action $\pi_{\tilde{L}}^{H\boxtimes K\boxtimes R}(x)$ is defined as
\[
\pi_{\tilde{L}}^{H\boxtimes K\boxtimes R}(x):= (\gamma^{\bullet})^{-1}\circ (\pi^H_{\tilde{L}}(x)\boxtimes \id\boxtimes \id)\circ \gamma^\bullet,
\]
where $\pi^H_{\tilde{L}}(x)\boxtimes \id\boxtimes \id: H_+^\varphi(\tilde{L})\boxtimes K^\mu\boxtimes R_-^\nu(\tilde{M})\to H_+^\varphi(\tilde{L})\boxtimes K^\mu\boxtimes R_-^\nu(\tilde{M})$ is the unitary induced by the map $H_+^\varphi(\tilde{L})\otimes K^\mu\otimes R_-^\nu(\tilde{M})\to H_+^\varphi(\tilde{L})\otimes K^\mu\otimes R_-^\nu(\tilde{M})$ given by $\xi\otimes\eta\otimes\chi\mapsto \pi_{\tilde{L}}^H(x)(\xi)\otimes\eta\otimes\chi$. One can prove this is a well-defined action using similar arguments to the proof of Theorem \ref{thm: RepresentationOnFusion}. Analogous arguments to those in the theorem also show that the action can be equivalently defined by taking $\tilde{M}\subset\tilde{L}^{c-}$, a path $\vartheta$ from $\tilde{I}\times \tilde{J}$ to $\tilde{M}\times \tilde{L}$, and letting
\[
\pi_{\tilde{L}}^{H\boxtimes K\boxtimes R}(x):= (\vartheta^{\bullet})^{-1}\circ ( \id\boxtimes \id\boxtimes \pi^R_{\tilde{L}}(x))\circ \vartheta^\bullet.
\]
Here, the map $\id\boxtimes \id\boxtimes \pi^R_{\tilde{L}}(x)$ is defined analogously to the map $\pi^H_{\tilde{L}}(x)\boxtimes \id\boxtimes \id$ above.

Consider the iterated Connes fusion $\big(H^\varphi_+(\tilde{I})\boxtimes K^\mu\big)\boxtimes R^\nu_-(\tilde{J})$ of $H_1:= H^\varphi_+(\tilde{I})\boxtimes K^\mu$ with $R^\nu$ over the interval $\tilde{J}$ on the right. We write $\pi_1$ for the action of $\A$ on~$H_1$.

\begin{proposition}\label{prop: MapAssociativity}
    The map $\big(H^\varphi_+(\tilde{I})\otimes K^\mu\big)\otimes R^\nu_-(\tilde{J})\to H^\varphi_+(\tilde{I})\otimes K^\mu \otimes R^\nu_-(\tilde{J})$ given by $(\xi\otimes\eta)\otimes \chi\mapsto \xi\otimes\eta\otimes\chi$ extends to a unitary map
    \[
    \big(H^\varphi_+(\tilde{I})\boxtimes K^\mu\big)\boxtimes R^\nu_-(\tilde{J}) \xrightarrow{\cong } H^\varphi_+(\tilde{I})\boxtimes K^\mu\boxtimes R^\nu_-(\tilde{J}).
    \]
\end{proposition}
\begin{proof}
     Let $\xi,\xi'\in H^\varphi_+(\tilde{I})$, $\eta,\eta'\in K^\mu$ and $\chi, \chi'\in R^\nu_-(\tilde{J})$. Let $\alpha:[0,1]\to \R$ be a path from $\tilde{I}$ to $\tilde{J}$ such that $\alpha([0,1])\subset \tilde{L}$ for some $\tilde{L}\in \Jcal_\R$. Then, the inner product of $(\xi\otimes\eta)\otimes\chi$ and $(\xi'\otimes\eta')\otimes\chi'$ in $\big(H^\varphi_+(\tilde{I})\boxtimes K^\mu\big)\boxtimes R^\nu_-(\tilde{J})$ equals (we denote as a subscript the Hilbert space where we are taking the inner product, for clarity)
    \begin{align*}
        \langle (\xi\otimes\eta)\otimes\chi\,|\, (\xi'&\otimes\eta')\otimes\chi'\rangle \\&= \langle \pi_{1, \tilde{J}}\big(Z^-(\chi', \tilde{J})^*Z^-(\chi, \tilde{J})\big)(\xi\otimes\eta)\,|\, \xi'\otimes\eta'\rangle_{H^\varphi_+(\tilde{I})\boxtimes K^\mu}\\
        & = \langle (\alpha^\bullet)^{-1}\circ \big[\pi^H_{\tilde{J}}\big(Z^-(\chi', \tilde{J})^*Z^-(\chi, \tilde{J})\big)\boxtimes\id\big]\circ \alpha^\bullet(\xi\otimes\eta)\,|\,\xi'\otimes\eta'\rangle_{H^\varphi_+(\tilde{I})\boxtimes K^\mu}\\
        &=\langle \pi^H_{\tilde{J}}\big(Z^-(\chi', \tilde{J})^*Z^-(\chi, \tilde{J})\big)(\xi)\otimes\eta\,|\, \xi'\otimes\eta'\rangle_{H^\varphi_+(\tilde{L})\boxtimes K^\mu}\\
        &=\langle \pi^K_{\tilde{I}}\big(Z^+(\xi',\tilde{I})^*Z^+\big(\pi^H_{\tilde{J}}\big(Z^-(\chi', \tilde{J})^*Z^-(\chi, \tilde{J})\big)(\xi), \tilde{I}\big)\big)(\eta)\,|\,\eta'\rangle_{K^\mu}.
\end{align*}
Note that 
\begin{align*}
Z^+\big(\pi^H_{\tilde{J}}\big(Z^-(\chi', \tilde{J})^*Z^-(\chi, \tilde{J})\big)(\xi), \tilde{I}\big)\big) &= \pi^H_{\tilde{J}}\big(Z^-(\chi', \tilde{J})^*Z^-(\chi, \tilde{J})\big)\circ  Z^+(\xi, \tilde{I})\\& = Z^+(\xi, \tilde{I}) Z^-(\chi', \tilde{J})^*Z^-(\chi, \tilde{J}),
\end{align*}
and hence the inner product above equals
\[
\langle \pi^K_{\tilde{I}}\big(Z^+(\xi',\tilde{I})^*Z^+(\xi, \tilde{I}) Z^-(\chi', \tilde{J})^*Z^-(\chi, \tilde{J})\big)(\eta)\,|\,\eta'\rangle_{K^\mu}.
\]
Since $Z^+(\xi',\tilde{I})^*Z^+(\xi, \tilde{I})\in \A(I)$ and $Z^-(\chi', \tilde{J})^*Z^-(\chi, \tilde{J})\in \A(J)$ commute, the map $\big(H^\varphi_+(\tilde{I})\otimes K^\mu\big)\otimes R^\nu_-(\tilde{J})\to H^\varphi_+(\tilde{I})\otimes K^\mu \otimes R^\nu_-(\tilde{J})$ given by $(\xi\otimes\eta)\otimes \chi\mapsto \xi\otimes\eta\otimes\chi$ is an isometry between dense subsets of $\big(H^\varphi_+(\tilde{I})\boxtimes K^\mu\big)\boxtimes R^\nu_-(\tilde{J})$ and $H^\varphi_+(\tilde{I})\boxtimes K^\mu\boxtimes R^\nu_-(\tilde{J})$, and it extends to a unitary equivalence between the completions.
\end{proof}

It is left to show that the unitary equivalence in Proposition \ref{prop: MapAssociativity} is an equivalence of $\varphi\mu\nu$-twisted $\A$-representations. Note that the equivalence clearly commutes with the action of $\A_\R(\tilde{I})$. Hence, it suffices to show that it also commutes with path continuations. This is the content of the following result, which follows from the same arguments as \cite[Prop. 2.22]{Gui21}.

\begin{proposition}\label{prop: BreakGammaInComponents}
    Let $\varphi,\mu,\nu\in\Aut(\A)$ and $H^\varphi,K^\mu, R^\nu$ twisted representations of $\A.$ Let $\tilde{I_0},\tilde{I_1},\tilde{J_0},\tilde{J_1}\in \Jcal_\R$ with $\tilde{J_i}\subset \tilde{I_i}^{c+}$ for $i = 0,1$. Let $\gamma = (\alpha,\beta)$ be a path in $\ConfR$ from $\tilde{I_0}\times\tilde{J_0}$ to $\tilde{I_1}\times \tilde{J_1}$. Then, the following diagrams commute
\[\begin{tikzcd}
	{\big(H^\varphi_+(\tilde{I_0})\boxtimes K^\mu\big)\boxtimes R^\nu_-(\tilde{J_0})} && {\big(H^\varphi_+(\tilde{I_1})\boxtimes K^\mu\big)\boxtimes R^\nu_-(\tilde{J_1})} \\
	{H^\varphi_+(\tilde{I_0})\boxtimes K^\mu\boxtimes R^\nu_-(\tilde{J_0})} && {H^\varphi_+(\tilde{I_1})\boxtimes K^\mu\boxtimes R^\nu_-(\tilde{J_1})} \\
	{H^\varphi_+(\tilde{I_0})\boxtimes \big(K^\mu\boxtimes R^\nu_-(\tilde{J_0})\big)} && {H^\varphi_+(\tilde{I_1})\boxtimes \big(K^\mu\boxtimes R^\nu_-(\tilde{J_1})\big)}
	\arrow["{\beta^\bullet(\alpha^\bullet\boxtimes\id)}", from=1-1, to=1-3]
	\arrow["\cong"', from=1-1, to=2-1]
	\arrow["\cong", from=1-3, to=2-3]
	\arrow["{\gamma^\bullet}", from=2-1, to=2-3]
	\arrow["\cong"', from=2-1, to=3-1]
	\arrow["\cong", from=2-3, to=3-3]
	\arrow["{\alpha^\bullet(\id\boxtimes\beta^\bullet)}", from=3-1, to=3-3]
\end{tikzcd}\]
\end{proposition}

We can similarly produce a canonical unitary equivalence of twisted $\A$-representations $H^\varphi_+(\tilde{I})\boxtimes \big(K^\mu\boxtimes R^\nu_-(\tilde{J})\big) \xrightarrow{\cong } H^\varphi_+(\tilde{I})\boxtimes K^\mu\boxtimes R^\nu_-(\tilde{J})$. The composition
\[
\big(H^\varphi_+(\tilde{I})\boxtimes K^\mu\big)\boxtimes R^\nu_-(\tilde{J})\xrightarrow{\cong}H^\varphi_+(\tilde{I})\boxtimes K^\mu\boxtimes R^\nu_-(\tilde{J})\xrightarrow{\cong}H^\varphi_+(\tilde{I})\boxtimes \big(K^\mu\boxtimes R^\nu_-(\tilde{J})\big)
\]
is therefore an equivalence of $\varphi\mu\nu$-twisted $\A$-representations.

\subsection{The $\Aut(\A)$ action}

Let $\varphi,\nu\in \Aut(\A)$ and $H^\varphi\in\Rep^\varphi(\A)$. We define a new representation $T_\nu(H^{\varphi}) = (H^\varphi, \pi^{\nu\ast H})$ on the same underlying Hilbert space of $H^\varphi$ by defining a new action $\pi^{\nu\ast H}$ of $\A$ by
\[
\pi_I^{\nu\ast H} (x) : = \pi_I^H(\nu^{-1} x),
\]
for $I\in \Jcal$ and $x\in \A(I)$. If $J\subset I$ is a standard inclusion, then $\pi_J^{\nu\ast H} = \pi_I^{\nu\ast H}$. If $J\subset I$ is special, we have
\[
\pi_J^{\nu\ast H}(x) = \pi_J^H(\nu^{-1} x) = \pi^H_I(\varphi^{-1}\nu^{-1} x) = \pi_I^{\nu\ast H}(\nu\varphi^{-1}\nu^{-1} x).
\]
In addition, given $K^{\varphi}\in \Rep^{\varphi}(\A)$, if $F: H^\varphi\to K^{\varphi}$ intertwines the actions $\pi^H$ and $\pi^{K}$, it also intertwines the actions $\pi^{\nu\ast H}$ and $\pi^{\nu\ast K}$. We have argued the following result.

\begin{proposition}
    Let $\nu\in\Aut(\A)$. Then, precomposition by $\nu^{-1}$ defines a functor
    \[
    T_\nu: \Rep^{\Aut(\A)}(\A)\to \Rep^{\Aut(\A)}(\A)
    \]
    such that $T_\nu(\Rep^\varphi(\A))\subset \Rep^{\nu\varphi\nu^{-1}}(\A)$ for all $\varphi\in \Aut(\A).$
\end{proposition}

Given automorphisms $\varphi,\mu\in \Aut(\A)$ and a twisted representation $H^\varphi\in \Rep^\varphi(\A)$, we write $$\Gamma_\mu: H^\varphi\to T_\mu(H^\varphi)$$ for the identity map of the underlying Hilbert spaces. Note that this is not a map of twisted $\A$-representations. We will show next that the action $T$ is compatible with the Connes fusion of representations. Let $\tilde{I}\in \Jcal_\R$ and $\xi = Z^+(\xi,\tilde{I})(\Omega)\in H^\varphi_+(\tilde{I})$. Then, $\Gamma_\nu\circ Z^+(\xi,\tilde{I})\circ V_{\nu^{-1}}: H_0\to T_\nu(H^\varphi)$ satisfies that $\Gamma_\nu\circ Z^+(\xi,\tilde{I})\circ V_{\nu^{-1}}(\Omega) = \Gamma_\nu\xi$ and, for all $x\in \A_R(\tilde{I}^{c+})$,
\begin{align*}
\Gamma_\nu\circ Z^+(\xi,\tilde{I})\circ V_{\nu^{-1}}\circ \pi_{0,I^c}(x) &=\Gamma_\nu\circ  Z^+(\xi, \tilde{I})\circ \pi_{0, I^c}(\nu^{-1} x)\circ V_{\nu^{-1}}\\
& = \Gamma_\nu\circ \pi^H_{\tilde{I}^{c+}}(\nu^{-1} x)\circ Z^+(\xi, \tilde{I})\circ V_{\nu^{-1}}\\ &= \pi^{\nu\ast H}_{\tilde{I}^{c+}}( x)\circ \Gamma_\nu\circ Z^+(\xi, \tilde{I})\circ V_{\nu^{-1}}.
\end{align*}
Therefore, we obtain an inclusion map
\[
\begin{array}{ccc}
     H^\varphi_+(\tilde{I})&\to &T_\nu(H^\varphi)_+(\tilde{I})  \\
    Z^+(\xi,\tilde{I})(\Omega) & \mapsto &\Gamma_\nu\circ Z^+(\xi,\tilde{I})\circ V_{\nu^{-1}}(\Omega)
\end{array}
\]
which can be easily seen to provide the equality 
\[
H^\varphi_+(\tilde{I}) = T_\nu(H^\varphi)_+(\tilde{I}).
\]
Accordingly, we write $Z^+(\Gamma_\nu\xi, \tilde{I}) = \Gamma_\nu\circ Z^+(\xi, \tilde{I})\circ V_{\nu^{-1}}$. Analogously, $H^\varphi_-(\tilde{I}) = T_\nu(H^\varphi)_-(\tilde{I})$ and $Z^-(\Gamma_\nu\xi, \tilde{I}) = \Gamma_\nu \circ Z^-(\xi, \tilde{I})\circ V_{\nu^{-1}}$ for all $\xi \in H^\varphi_-(\tilde{I})$.

\begin{proposition}\label{prop: ActionData}
    Let $\varphi,\mu,\nu\in \Aut(\A)$ and $H^\varphi\in \Rep^\varphi(\A)$, $K^\mu\in \Rep^\mu(\A)$. The map $ T_\nu(H^\varphi)_+(\tilde{I})\otimes T_\nu(K^\mu)_-(\tilde{J})\to H^\varphi_+(\tilde{I})\otimes K^\mu_-(\tilde{J})$ given by $\Gamma_\nu\xi\otimes\Gamma_\nu\eta\mapsto \xi\otimes\eta$ induces a unitary
    \[
     T_\nu(H^\varphi)_+(\tilde{I})\boxtimes T_\nu(K^\mu)_-(\tilde{J})  \xrightarrow{\cong}  T_\nu\big(H^\varphi_+(\tilde{I})\boxtimes K^\mu_-(\tilde{J})\big).
    \]
    In addition, the map above commutes with path-continuations and is an equivalence of $\nu\varphi\mu\nu^{-1}$-twisted $\A$-representations.
\end{proposition}
\begin{proof}
    Let $\Gamma_\nu\xi,\Gamma_\nu\xi'\in T_\nu(H^\varphi)_+(\tilde{I})$ and $\Gamma_\nu\eta,\Gamma_\nu\eta'\in T_\nu(K^\mu)_-(\tilde{J})$. Let us compute
    \begin{align*}
        \langle \Gamma_\nu\xi\otimes \Gamma_\nu\eta\,|\Gamma_\nu\xi'\otimes\Gamma_\nu\eta'\rangle &=\langle V_{\nu^{-1}}^{-1} Z^-(\eta',\tilde{J})^*Z^-(\eta, \tilde{J})V_{\nu^{-1}} V_{\nu^{-1}}^{-1} Z^+(\xi',\tilde{I})^*Z^+(\xi, \tilde{I})V_{\nu^{-1}}(\Omega)\,|\,\Omega\rangle \\ &=\langle  Z^-(\eta',\tilde{J})^*Z^-(\eta, \tilde{J}) Z^+(\xi',\tilde{I})^*Z^+(\xi, \tilde{I})(\Omega)\,|\,\Omega\rangle,
     \end{align*}
     where we use that $V_{\nu^{-1}}(\Omega) = \Omega$. This implies that the map $ T_\nu(H^\varphi)_+(\tilde{I})\otimes T_\nu(K^\mu)_-(\tilde{J})\to H^\varphi_+(\tilde{I})\otimes K^\mu_-(\tilde{J})$ given by $\Gamma_\nu\xi\otimes\Gamma_\nu\eta\mapsto \xi\otimes\eta$ indeed extends to a unitary equivalence of the completions.
     
     Let us now show that the equivalence intertwines the twisted actions of $\A$. Let $\tilde{L},\tilde{M}\in \Jcal_\R$ with $\tilde{M}\subset \tilde{L}^{c+}$, and $\gamma$ a path in $\ConfR$ from $\tilde{I}\times\tilde{J}$ to $\tilde{L}\times\tilde{M}$. The action of $x\in \A_\R(\tilde{L})$ on the right-hand side is given by
\begin{align*}
    \pi_{\tilde{L}}^{\nu\ast(H\boxtimes K)}(x) = \pi_{\tilde{L}}^{H\boxtimes K}(\nu^{-1} x) = (\gamma^\bullet)^{-1}\circ \big(\pi_{\tilde{L}}^H(\nu^{-1} x)\boxtimes\id\big)\circ \gamma^\bullet
\end{align*}
and on the left-hand side by
\[
\pi_{\tilde{L}}^{(\nu* H)\boxtimes(\nu* K)}(x) = (\gamma^\bullet)^{-1}\circ \big(\pi^{\nu\ast H}_{\tilde{L}}(x)\boxtimes \id\big)\circ \gamma^\bullet = (\gamma^\bullet)^{-1}\circ \big(\pi_{\tilde{L}}^H(\nu^{-1} x)\boxtimes\id\big)\circ \gamma^\bullet,
\]
and the claim follows. The unitary map is clearly compatible with the maps induced from inclusions and restrictions of intervals, and hence is compatible with path continuations. 
\end{proof}

We also have compatibility of the action with the vacuum representation by the equivalence of $\A$-representations
\[
H_0\xrightarrow{V_{\nu^{-1}}}H_0\xrightarrow{\Gamma_\nu}T_\nu(H_0) 
\]
for all $\nu\in\Aut(\A)$. Finally, we state and prove the following compatibility with the fusion over a single interval.

\begin{proposition}\label{prop: SwapWithAction}
    Let $\varphi,\mu\in \Aut(\A)$ be automorphisms and $H^\varphi\in \Rep^\varphi(\A)$ and $K^\mu\in \Rep^\mu(\A)$ be twisted representations. Given $\tilde{I}\in \Jcal_\R$, the canonical map $H^\varphi_+(\tilde{I})\otimes K^\mu\to T_\varphi(K^\mu)\otimes H^\varphi_-(\tilde{I})$ given by $\xi\otimes\eta\mapsto \Gamma_\varphi\eta\otimes \xi$ extends to a unitary
    \[
   H^\varphi_+(\tilde{I})\boxtimes K^\mu\xrightarrow{\cong} T_\varphi(K^\mu)\boxtimes  H^\varphi_-(\tilde{I})
    \]
    compatible with the actions of $\A_\R$ and with path-continuations. Similarly, there is a canonical unitary
    \[
    H^\varphi\boxtimes K^\mu_-(\tilde{I})\xrightarrow{\cong}K^\mu_+(\tilde{I})\boxtimes T_{\mu^{-1}}(H^\varphi)
    \]
    compatible with the actions of $\A_\R$ and with path-continuations.
\end{proposition}
\begin{proof}
    We only prove the first part, the second follows analogously. Let $\xi,\xi' \in H^\varphi_+(\tilde{I}) = H_-^\varphi(\tilde{I})$ and $\eta,\eta'\in K^\mu$. Recall that, by Equation \eqref{eq: RelationZ+Z-}, it holds that $Z^-(\xi, \tilde{I}) = Z^+(\xi, \tilde{I})\circ V_{\varphi{-1}}$, and similarly for $Z^-(\xi', \tilde{I})$. Then, the inner product of $\Gamma_\varphi\eta\otimes\xi$ and $\Gamma_\varphi\eta'\otimes\xi'$ in $T_\varphi(K^\mu)\boxtimes H^\varphi_-(\tilde{I})$~is
    \begin{align*}
        \langle \Gamma_\varphi\eta\otimes\xi\,|\,\Gamma_\varphi\eta'\otimes\xi'\rangle &=\langle \pi_{\tilde{I}}^{\varphi\ast K}\big( Z^-(\xi', \tilde{I})^*Z^-(\xi, \tilde{I})\big)(\Gamma_\varphi\eta)\,|\,\Gamma_\varphi\eta'\rangle \\ &=\langle \pi_{\tilde{I}}^{\varphi\ast K}\big( V_\varphi\circ Z^+(\xi', \tilde{I})^*Z^+(\xi, \tilde{I})\circ  V_{\varphi^{-1}}\big)(\Gamma_\varphi\eta)\,|\,\Gamma_\varphi\eta'\rangle
        \\ &= \langle \pi_{\tilde{I}}^{\varphi\ast K}\big(\varphi\big(Z^+(\xi', \tilde{I})^*Z^+(\xi, \tilde{I})\big)\big)(\Gamma_\varphi\eta)\,|\,\Gamma_\varphi\eta'\rangle  
        \\ & = \langle \pi_{\tilde{I}}^{K}\big(Z^+(\xi', \tilde{I})^*Z^+(\xi, \tilde{I})\big)(\eta)\,|\,\eta'\rangle,
    \end{align*}
    which is also the inner product of $\xi\otimes\eta$ and $\xi'\otimes\eta'$ in $H^\varphi_+(\tilde{I})\boxtimes K^\mu$. Hence, the map extends to a unitary equivalence between the completions. Compatibility with path continuations follows directly from compatibility with maps induced by inclusions and restrictions of intervals, which is clear. Compatibility with the action of $\A_\R$ is also clear.
\end{proof}

\subsection{Conformal structures}

The fact that any twisted representation of a conformal net is conformal covariant will provide the crossed balance. Hence, let us recall how to produce this conformal structure.

Let $\widetilde{\Diff^+}(S^1) = \{(\widetilde{f}, f)\in \Diff(\R)\times \Diff^+(S^1)\,|\,q\circ \tilde{f} = f\circ q\}$ be the universal cover of $\Diff^+(S^1).$ We can lift the projective representation $U$ of $\Diff^+(S^1)$ on $H_0$ to a projective representation of $\widetilde{\Diff^+}(S^1)$ on $H_0$, which we continue denoting by $U$, given by $U(\widetilde{f},f) = U(f)$. Denoting the image of a unitary $V\in U(H_0)$ in $\mathcal{P}U(H_0)$ by $[V]$, we define the topological group
\[
\Diff_\A^+(S^1) = \{(\widetilde{f},f,V)\in \widetilde{\Diff^+}(S^1)\times U(H_0)\,|\,[V]= U(\widetilde{f},f)\}.
\]
The group $\Diff_\A^+(S^1)$ inherits a topology from that of $\widetilde{\Diff^+}(S^1)\times U(H_0)$, and is a central extension of $\widetilde{\Diff^+}(S^1)$ by $U(1)\cong \{V\in U(H_0)\,|\, [V] = [\id]\}$,
\[
1\to U(1)\to \Diff_\A^+(S^1)\to \widetilde{\Diff^+}(S^1)\to 1.
\]
A unitary action of $\Diff_\A^+(S^1)$ on a Hilbert space $H$ is, by definition, a continuous homomorphism $\Diff_\A^+(S^1)\to U(H)$ such that the central $U(1)$ acts in the standard way. There is a canonical unitary action of $\Diff_\A^+(S^1)$ on $H_0$ given by $(\widetilde{f},f,V)\in \Diff_\A^+(S^1)\mapsto V\in U(H_0)$. We shall continue denoting this action by $U$. 

Recall that, given an interval $I\in \Jcal$, we denote by $\Diff_I(S^1)$ the subgroup of $\Diff^+(S^1)$ of diffeomorphisms with support in $I$. We write $\widetilde{\Diff_I}(S^1)$ for the connected branch of the preimage of $\Diff_I(S^1)$ in $\widetilde{\Diff^+}(S^1)$ containing the identity, and $\Diff^+_\A(I)$ for the preimage of $\widetilde{\Diff_I}(S^1)$ in $\Diff_\A^+(S^1).$ Note that, by conformal covariance, given $(\widetilde{f},f,V)\in \Diff_\A^+(I)$, we have $U(\widetilde{f},f,V)\in \Hom_{\A(I^c)}(H_0, H_0)\cong \A(I)$. We will write $U(\widetilde{f},f,V)\in \A(I)$. These elements form the Virasoro subnet inside of $\A$.

\begin{remark}\label{rk: AutomorphismsPreserveVirasoro}
    Let $\varphi\in \Aut(\A)$. Then, $\varphi$ fixes the Virasoro subnet of $\A$. Indeed, by \cite[Sec. II.3]{gf93}, $V_\varphi$ commutes with the action of $\Mob$ on $H_0$. In addition, by \cite[Thm. 6.1.9]{weiner2007conformalcovariancerelatedproperties}, $V_\varphi$ commutes with the projective action of $\Diff^+(S^1)$ on $H_0$, meaning that $[V_\varphi\circ U(\tilde{f}, f, V)\circ V_\varphi^*] =  [U(\tilde{f}, f, V)]$ for every $(\tilde{f}, f, V)\in \Diff_\A^+(S^1)$. Therefore, for every $(\tilde{f}, f, V)\in \Diff_\A^+(S^1)$, there exists a scalar $\lambda(\tilde{f}, f, V)\in U(1)$ such that $V = \lambda(\tilde{f}, f, V)V_\varphi VV_\varphi^*$. Since both $U$ and $V_\varphi\circ U(-)\circ V_\varphi$ are homomorphisms, $\lambda: \Diff_\A^+(S^1)\to U(1)$ is a character. We claim that the character factors through the projection $\Diff_\A^+(S^1)\to \Diff^+(S^1)$. Indeed, given two elements $(\tilde{f}_1, f, V), (\tilde{f}_2, f, W)\in \Diff_\A^+(S^1)$ mapping to the same element of $\Diff^+(S^1)$, we have $[V] = U(f) = [W]$, and hence there exists $\mu\in  U(1)$ such that $W = \mu V$. Then, $\lambda(\tilde{f}_1, f, V) = VV_\varphi V^* V_\varphi^* = \mu\mu^* WV_\varphi W^* V_\varphi = \lambda(\tilde{f}_2, f, W)$. By \cite{MR267589, MR287585, MR339267}, the group $\Diff^+(S^1)$ is simple, and hence any character on it is trivial. Therefore, $\lambda: \Diff^+_\A(S^1)\to U(1)$ is trivial and $U(-) = V_\varphi\circ U(-)\circ V_\varphi^*$ as maps $\Diff^+_{\A}(S^1)\to B(H_0)$, as needed.
\end{remark}

\begin{theorem}\label{Thm: TwistedRepIsConformal}
    Let $\varphi\in \Aut(\A)$ be an automorphism and $H^\varphi\in \Rep^\varphi(\A)$ be a twisted representation. Then, $H^\varphi$ is conformal covariant in the sense that there exists a unique unitary action $U^H$ of $\Diff^+_\A(S^1)$ on $H^\varphi$ such that, for all $I\in \Jcal$ and $(\tilde{f},f,V)\in \Diff_\A^+(I)$, it holds that
    \[
    U^H(\widetilde{f},f,V) = \pi^H_I(U(\widetilde{f},f,V)).
    \]
\end{theorem}
\begin{proof}
    This is essentially \cite[Thm. 2.2]{Gui21}. By \cite{Hen19}, the canonical map
    \begin{equation}\label{eq: ColimDiffA(S1)}
    \colim_{I\in \Jcal}\Diff^+_\A(I)\to \Diff_\A^+(S^1)
    \end{equation}
    is an isomorphism of topological groups. Consider the family of continuous group homomorphisms $\Diff_\A^+(I)\xrightarrow{U} \A(I)\xrightarrow{\pi_I^H}U(H^\varphi)$. Since the automorphism $\varphi$ is the identity on the Virasoro subnet of $\A$ by Remark \ref{rk: AutomorphismsPreserveVirasoro}, the family above is compatible with the inclusion maps $J\subset I$ of intervals in $\Jcal$. Hence, it produces a canonical action of the colimit
    \[
    \Diff_\A^+(S^1)\cong \colim_{I\in \Jcal}\Diff^+_\A(I)\to U(H^\varphi).
    \]
    The representation is unique by surjectivity of \eqref{eq: ColimDiffA(S1)}.
\end{proof}

\begin{corollary}\label{cor: RepOfDiffA}
    For any $(\widetilde{f},f,V)\in \Diff^+_\A(S^1)$, it holds that $U^H(\widetilde{f},f,V)\in \bigvee\limits_{I\in \Jcal}\pi^H_I(\A(I))$, and for every $\tilde{I}\in \Jcal_\R$ and $x\in \A_\R(\tilde{I})$, it holds that
    \[
    U^H(\widetilde{f},f,V)\circ \pi_{\tilde{I}}^H(x)\circ U^H(\widetilde{f},f,V)^* = \pi_{\widetilde{f}\tilde{I}}^H\big(U(\widetilde{f},f,V)\, x \,U(\widetilde{f},f,V)^*\big). 
    \]
\end{corollary}
\begin{proof}
    By \cite[Lem. 17 ii]{Hen19}, $\Diff_\A^+(S^1)$ is algebraically generated by $\{\Diff_\A^+(I)\}_{I\in \Jcal}$. Hence, it is enough to show the first statement when $(\widetilde{f},f,V)\in \Diff_\A^+(J)$ for some $J\in \Jcal$. In this case, the claim is clear.

    For the second claim, fix $J\in \Jcal$ such that $I$ and $J$ can be covered by an interval $L$ in $\Jcal$, and assume that $(\tilde{f},f,V)\in \Diff^+_\A(J)$. Denote by $\widetilde{L}\in\Jcal_\R$ the lift of $L$ to $\R$ containing $\tilde{I}$. Then, $x\in \A_\R(\tilde{I})\subset\A_\R(\widetilde{L})$ and we have
\begin{align*}
    U^H(\widetilde{f},f,V)\circ \pi_{\tilde{I}}^H(x)\circ U^H(\widetilde{f},f,V)^*  & = \pi_{\widetilde{L}}^H\big(U(\widetilde{f},f,V)\big)\circ \pi^H_{\tilde{I}}(x)\circ \pi_{\widetilde{L}}^H\big(U(\widetilde{f},f,V)\big)^*\\ &=\pi_{\widetilde{L}}^H\big(U(\widetilde{f},f,V)\, x\, U(\widetilde{f},f,V)^*\big)
    \\ &= \pi_{\widetilde{f}\tilde{I}}^H\big(U(\widetilde{f},f,V)\, x \,U(\widetilde{f},f,V)^*\big),
\end{align*}
where the last equality comes from the fact that $U(\widetilde{f},f,V)\, x\, U(\widetilde{f},f,V)^*\in \A(fI)$ and we have taken $\widetilde{\Diff_J}(S^1)$ to be the connected lift of $\Diff^+(J)$ to $\widetilde{\Diff^+}(S^1)$ containing the identity, hence $\widetilde{f}\tilde{I}\subset \widetilde{L}$. By the additivity of $\A$, and hence of $\A_\R$, the claim holds for all $\tilde{I}\in \Jcal_\R$ and $x\in \A_\R(\tilde{I})$, for the given $(\widetilde{f},f,V)\in \Diff_\A^+(J)$. Since $\Diff_\A^+(S^1)$ is algebraically generated by $\{\Diff_\A^+(I)\}_{I\in \Jcal}$, the claim follows for all $(\tilde{f},f,V)\in \Diff_A^+(S^1)$.
\end{proof}

The Lie algebra of $\Diff^+(S^1)$ is the Lie algebra $\Vect(S^1)$ of vector fields on $S^1$, where the Lie bracket is the negative of the usual Lie bracket of vector fields. We let $\exp: \Vect(S^1)\to \Diff^+(S^1)$ be the exponential map and write $\Vect_{\C}(S^1)$ for the complexification of $\Vect(S^1)$. For each $n\in \Z$, we define the complex vector field $L_n(e^{i\theta}):=-ie^{in\theta}\frac{d}{d\theta}\in \Vect_{\C}(S^1)$. These vectors $L_n$ form the Witt algebra $\mathscr{W}$, which is a dense Lie subalgebra of $\Vect_{\C}(S^1).$ We define a $*$-structure on $\mathscr{W}$ by setting $L_n^* = L_{-n}$. Then, a vector $X\in \mathscr{W}$ is self-adjoint if and only if $iX\in \Vect(S^1)$. For a self-adjoint vector $X\in \Vect_\C(S^1)$, we write $\exp_{iX}$ for the one parameter subgroup of $\Diff^+(S^1)$ given by $t\in\R\mapsto \exp(itX)$, so that $\exp_{iL_0}$ is the subgroup of rotations of $S^1$. We also define $\widetilde{\exp}_X: \R\to \widetilde{\Diff^+}(S^1)$ for the one-parameter subgroup of $\widetilde{\Diff^+}(S^1)$ lifting $\exp_X$, and $\widetilde{\exp}(X):=\widetilde{\exp}_X(1)$.

Note that $\widetilde{\Diff^+}(S^1)$ contains the universal covering space $\widetilde{\Mob}$ of $\Mob$, which is generated by $\widetilde{\exp}(iX)$ where $X = \overline{a_{1}}L_{-1} + a_0L_0 + a_1 L_1\in\Vect_{\C}(S^1)$ with $a_0\in\R$ and $a_1\in\C$. The unitary projective representation $U^H$ of $\widetilde{\Diff^+}(S^1)$ on $H^\varphi$ restricts to a unitary projective representation of $\widetilde{\Mob}$. By \cite{MR58601}, such a projective representation lifts to a unique unitary representation of $\widetilde{\Mob}$ on $H^\varphi$. Given $X$ as above, we write $e^{iX}$ for the action of $\widetilde{\exp}(iX)\in \widetilde{\Mob}$ on $H^\varphi$.

\begin{definition}
    Let $\varphi\in \Aut(\A)$ and $H^\varphi\in \Rep^\varphi(\A)$. We denote by
    \[
    \theta_H:=e^{ -2\pi i L_0}: H^\varphi\to H^\varphi
    \]
    the unitary on the Hilbert space $H^\varphi$ given by the action of $\widetilde{\exp}(-2\pi iL_0)\in \widetilde{\Mob}$.
\end{definition}

\begin{proposition}
    The unitary $\theta_H$ provides an equivalence 
    \[ H^\varphi\xrightarrow{\cong}T_\varphi(H^\varphi)
    \]
    of twisted $\A$-representations.
\end{proposition}
\begin{proof}
    Let $\tilde{I}\in \Jcal_\R$ and $x\in \A_\R(\tilde{I})$. Then, by Corollary \ref{cor: RepOfDiffA}, we have 
    \[
    e^{-2\pi iL_0}\circ \pi^H_{\tilde{I}}(x)\circ e^{2\pi i L_0} = \pi^H_{\tilde{I} - 1}(e^{-2\pi i L_0 }\, x\, e^{2\pi i L_0}) = \pi^{\varphi*H}_{\tilde{I}}(e^{-2\pi i L_0 }\, x\, e^{2\pi i L_0}).
    \]
    Since the action of $e^{-2\pi i L_0}$ on $\A(I)$ is the identity, the claim follows.
\end{proof}

The family of unitaries $\theta$ is compatible with the action $T$ in the following way.

\begin{proposition}\label{prop: CompatibilityActionConformalStruct}
    Let $\nu\in \Aut(\A)$. The action $\theta_{T_\nu(H^\varphi)}=e^{-2\pi i L_0}$ of $\widetilde{\exp}(-2\pi iL_0)$ on $T_\nu(H^\varphi)$ is given by $T_\nu(\theta_H)$.
\end{proposition}
\begin{proof}
    Recall that $T_\nu$ acts by the identity on morphisms. Hence, we have to show that the action $U^{\nu\ast H}$ of $\Diff_\A(S^1)$ on $T_\nu(H^\varphi)$ is given by the action $U^H$ on its underlying Hilbert space $H^\varphi$. By uniqueness in Theorem \ref{Thm: TwistedRepIsConformal}, it is enough to show that, for all $I\in \Jcal$ and $(\tilde{f},f,V)\in \Diff_\A^+(I)$, it holds that
    \[
    \pi_I^{\nu\ast H}\big(U(\tilde{f},f,V)\big) = \pi_I^{H}\big(U(\tilde{f},f,V)\big).
    \]
    By definition, the left-hand side is given by $\pi_I^{H}\big(\nu^{-1}U(\tilde{f},f,V)\big) = \pi_I^{H}\big(U(\tilde{f},f,V)\big)$, where we use that $\nu^{-1}$ restricts to the identity on the Virasoro subnet of $\A$ by Remark \ref{rk: AutomorphismsPreserveVirasoro}. The claim follows.
\end{proof}

For the rest of this section, we will describe the action $e^{ -2\pi iL_0}$ on the Connes fusion $H^\varphi_+(\tilde{I})\boxtimes K^\mu$ of two twisted representations over a given interval $\tilde{I}\in \Jcal_\R$. We shall actually do it in full generality, and describe the whole conformal structure on $H^\varphi_+(\tilde{I})\boxtimes K^\mu$. The arguments are almost all word for word those in \cite[Sec. 2.4]{Gui21}. Let $(\tilde{f},f,V)\in \Diff_A^+(S^1)$ and choose a map $\lambda:[0,1]\to \Diff_A^+(S^1)$ from $1 = (\id,\id,1)$ to $(\tilde{f},f,V).$ We require that $\lambda$ descends to a path $[\lambda]:[0,1]\to \widetilde{\Diff^+}(S^1)$ which is continuous. The homotopy class of $[\lambda]$ is uniquely determined by $(\tilde{f},f,V)$. Given $\tilde{I}\in\Jcal_\R$ and $z\in \tilde{I}$, the map
\[
\begin{array}{cccc}
    \lambda_z: &[0,1]&\to &\R  \\
     & t&\mapsto & \text{proj}_{\Diff(\R)}(\lambda(t))(z)
\end{array}
\]
is a path from $\tilde{I}$ to $\tilde{f}\tilde{I}$, where $\text{proj}_{\Diff(\R)}:\Diff_A^+(S^1)\to\widetilde{\Diff^+}(S^1) \to \Diff(\R)$ is the projection.

We define the unitary map $U_0^{H\boxtimes K}(\tilde{f},f,V):H_+^\varphi(\tilde{I})\boxtimes K^\mu\xrightarrow{\cong} H_+^\varphi(\tilde{I})\boxtimes K^\mu$ as the completion~of
\begin{align}\label{eq: ConformalStructOnConnesFusion0}
U_0^{H\boxtimes K}&(\tilde{f},f,V)(\xi\otimes\eta) \\&=\nonumber (\lambda_z^\bullet)^{-1}\big(U^H(\tilde{f},f,V)\circ Z^+(\xi,\tilde{I})\circ U(\tilde{f},f,V)^{-1}(\Omega)\otimes U^K(\tilde{f},f,V)(\eta)\big)
\end{align}
for any $\xi\otimes \eta\in H^\varphi_+(\tilde{I})\otimes K^\mu$. Note that we have dropped $\tilde{I}$ from the notation of $U_0^{H\boxtimes K}$. We will continue doing so when the interval $\tilde{I}$ is clear from the context. We next argue that Equation \eqref{eq: ConformalStructOnConnesFusion0} indeed provides a unitary operator on $H_+^\varphi(\tilde{I})\boxtimes K^\mu$. First of all, we need to show that $U^H(\tilde{f},f,V)\circ Z^+(\xi,\tilde{I})\circ U(\tilde{f},f,V)^{-1}(\Omega)\in H^\varphi_+(\tilde{f}\tilde{I})$, meaning that $U^H(\tilde{f},f,V)\circ Z^+(\xi,\tilde{I})\circ U(\tilde{f},f,V)^{-1}:H_0\to H^\varphi$ intertwines the actions of $\A_\R([\tilde{f}\tilde{I}]^{c+}) = \A_\R(\tilde{f}\tilde{I}^{c+})$ on $H_0$ and $H^\varphi$. Note that $\A_\R(\tilde{f}\tilde{I}^{c+}) =  \A(f I^c) = U(\tilde{f},f,V)\,\A(I^c )\, U(\tilde{f},f,V)^{-1}$. Let $U(\tilde{f},f,V)\,x\, U(\tilde{f},f,V)^{-1}\in \A(fI^c)$. Then
\begin{align*}
    U^H(\tilde{f},f,V)\circ& Z^+(\xi,\tilde{I})\circ U(\tilde{f},f,V)^{-1}\circ \pi_{0,fI^c}\big(U(\tilde{f},f,V)\,x\, U(\tilde{f},f,V)^{-1}\big) = \\&= U^H(\tilde{f},f,V)\circ Z^+(\xi,\tilde{I})\circ \pi_{0,fI^c}(x)\circ U(\tilde{f},f,V)^{-1}\\
    &= U^H(\tilde{f},f,V)\circ \pi^H_{\tilde{I}^{c+}}(x)\circ Z^+(\xi,\tilde{I})\circ U(\tilde{f},f,V)^{-1} \\&= \pi^H_{\tilde{f}\tilde{I}^{c+}}\big(U(\tilde{f},f,V)\,x\, U(\tilde{f},f,V)^{-1}\big)\circ U^H(\tilde{f},f,V)\circ Z^+(\xi,\tilde{I})\circ U(\tilde{f},f,V)^{-1},
\end{align*}
as needed. Hence, we have that $\big(U^H(\tilde{f},f,V)\circ Z^+(\xi,\tilde{I})\circ U(\tilde{f},f,V)^{-1}\big)\otimes U^K(\tilde{f}, f, V)$ indeed defines a map $H^\varphi_+(\tilde{I})\otimes K^\mu\xrightarrow{} H^\varphi_+(\tilde{f}\tilde{I})\otimes K^\mu$. We write $U^H(\tilde{f},f,V)\,\xi\,U(\tilde{f},f,V)^{-1}:=U^H(\tilde{f},f,V)\circ Z^+(\xi,\tilde{I})\circ U(\tilde{f},f,V)^{-1}(\Omega)\in H^\varphi_+(\tilde{f}\tilde{I})$. Then, Equation \eqref{eq: ConformalStructOnConnesFusion0} reads
\begin{equation}
\label{eq: ConformalStructOnConnesFusion1}
    U_0^{H\boxtimes K}(\tilde{f},f,V)(\xi\otimes \eta) = (\lambda_z^\bullet)^{-1}\big(U^H(\tilde{f},f,V)\,\xi\,U(\tilde{f},f,V)^{-1}\otimes U^K(\tilde{f},f,V)\eta\big).
\end{equation}
Note that, a priori, we have only defined a map $U^{H\boxtimes K}_0(\tilde{f},f,V): H^\varphi_+(\tilde{I})\otimes K^\mu\to H^\varphi_+(\tilde{I})\otimes K^\mu$. We now show that it is actually an isometry. Let $\xi,\xi'\in H^\varphi_+(\tilde{I})$ and $\eta,\eta'\in K^\mu$. Then,
\begin{align*}
    \langle  U_0^{H\boxtimes K}&(\tilde{f},f,V)(\xi\otimes \eta)\,|\,  U_0^{H\boxtimes K}(\tilde{f},f,V)(\xi'\otimes \eta')\rangle \\& = \langle U^H(\tilde{f},f,V)\,\xi\,U(\tilde{f},f,V)^{-1}\otimes U^K(\tilde{f},f,V)\eta\,| \, U^H(\tilde{f},f,V)\,\xi'\,U(\tilde{f},f,V)^{-1}\otimes U^K(\tilde{f},f,V)\eta'\rangle \\
    & = \langle \pi^K_{\tilde{f}\tilde{I}}\big(U(\tilde{f},f,V)\circ Z^+(\xi',\tilde{I})^*Z^+(\xi,\tilde{I})\circ U(\tilde{f},f,V)^{-1}\big)\circ U^K(\tilde{f},f,V)(\eta)\,|\,U^K(\tilde{f},f,V)(\eta')\rangle\\ &= \langle U^K(\tilde{f},f,V)\circ \pi^K_{\tilde{I}}\big(Z^+(\xi',\tilde{I})^*Z^+(\xi,\tilde{I})\big)(\eta)\,|\,U^K(\tilde{f},f,V)(\eta')\rangle \\ & =  \langle \pi^K_{\tilde{I}}\big(Z^+(\xi',\tilde{I})^*Z^+(\xi,\tilde{I})\big)(\eta)\,|\,\eta'\rangle \\ & = \langle \xi\otimes \eta\,|\,\xi'\otimes\eta'\rangle.
\end{align*}
Since the image of $\lambda_z^\bullet U_0^{H \boxtimes K}(\tilde{f},f,V)$ is $H^\varphi_+(\tilde{f}\tilde{I})\otimes K^\mu$, which is dense in $H^\varphi_+(\tilde{f}\tilde{I})\boxtimes K^\mu$, the map $U_0^{H\boxtimes K}(\tilde{f},f,V)$ extends to a unitary $H^\varphi_+(\tilde{I})\boxtimes K^\mu\xrightarrow{\cong}H^\varphi_+(\tilde{I})\boxtimes K^\mu$.

\begin{lemma}
    Assume that Equation \eqref{eq: ConformalStructOnConnesFusion1} defines a unitary representation $U_0^{H\boxtimes K}$ of $\Diff_\A^+(S^1)$ on $H_+^\varphi(\tilde{I})\boxtimes K^\mu$. Then, $U^{H\boxtimes K}_0$ equals the canonical representation $U^{H\boxtimes K}$ of $\Diff_\A^+(S^1)$ on the twisted $\A$-representation $H_+^\varphi(\tilde{I})\boxtimes K^\mu\in \Rep^{\varphi\mu}(\A)$.
\end{lemma}
\begin{proof}
    Assuming that $U_0^{H\boxtimes K}$ is a representation, by the uniqueness in Theorem \ref{Thm: TwistedRepIsConformal}, it is enough to show that, for all $\tilde{L}\in \Jcal$ and $(\tilde{f},f,V)\in \Diff_\A^+(L)$,
    \[
    U_0^{H\boxtimes K}(\tilde{f},f,V) = \pi^{H\boxtimes K}_{\tilde{L}}\big(U(\tilde{f},f,V)\big)
    \]
    when acting on $H^\varphi_+(\tilde{I})\boxtimes K^\mu$. Note that, since the automorphism $\varphi$ restricts to the identity on the Virasoro subnet of $\A$ (by Remark \ref{rk: AutomorphismsPreserveVirasoro}), we have that $\pi^{H\boxtimes K}_{\tilde{L} + n}\big(U(\tilde{f,f,V})\big) = \pi^{H\boxtimes K}_{\tilde{L} }\big(U(\tilde{f,f,V})\big)$ for any $n\in \mathbb{Z}$. It is clear by definition of $U_0^{H\boxtimes K}$ that it commutes with the maps induced by restriction and inclusion of intervals, and hence with path continuations. Therefore, we can assume, without loss of generality, that $\tilde{L}\subset\tilde{I}^{c+}$. Then, we have that $(\tilde{f},f,V)\in \Diff^+_\A(I^c)$, which implies that the path $\lambda_z$ can be taken to be constant, and that $U^H(\tilde{f},f,V)\circ Z^+(\xi,\tilde{I})\circ U(\tilde{f},f,V)^{-1} = Z^+(\xi,\tilde{I})$ for all $\xi\in H_+^\varphi(\tilde{I})$. Hence, if $\xi\in H_+^\varphi(\tilde{I})$ and $\eta\in K^\mu$, we have
    \[
    U_0^{H\boxtimes K}(\tilde{f},f,V)(\xi\otimes\eta) = \xi\otimes U^K(\tilde{f},f,V)(\eta).
    \]
    Hence, it is only left to show that $\pi^{H\boxtimes K}_{\tilde{L}}\big(U(\tilde{f},f,V)\big)(\xi\otimes\eta) = \xi\otimes U^K(\tilde{f},f,V)(\eta)$. This follows directly from the canonical equivalence (compatible with the $\A_\R$-actions) $H^\varphi_+(\tilde{I})\boxtimes K^\mu\xrightarrow{\cong }H^\varphi_+(\tilde{I})\boxtimes K^\mu_-(\tilde{I}^{c+})$ and the definition of the action of $\A_\R(\tilde{L})$ on the codomain.
\end{proof}

It is only left to show that Equation \eqref{eq: ConformalStructOnConnesFusion1} indeed defines a unitary representation of $\Diff_\A^+(S^1)$ on $H^\varphi_+(\tilde{I})\boxtimes K^\mu$. This is the statement of the following lemma, which follows from analogous arguments to \cite[Lem. 2.20]{Gui21}. 
\begin{lemma}
    The map $U_0^{H\boxtimes K}$ defines a unitary representation of $\Diff_\A^+(S^1)$ on $H^\varphi_+(\tilde{I})\boxtimes K^\mu$. Namely, for any $\xi\in H_+^\varphi(\tilde{I})$, $\eta\in K^\mu$ and $(\tilde{f},f,V),\, (\tilde{g},g,W)\in \Diff_\A^+(S^1)$, it holds that
    \begin{equation}\label{eq: U0IsAction}
    U_0^{H\boxtimes K}(\tilde{f},f,V)\,U_0^{H\boxtimes K}(\tilde{g},g,W)(\xi\otimes \eta) = U_0^{H\boxtimes K}(\tilde{f}\tilde{g},fg,VW)(\xi\otimes\eta).
    \end{equation}
\end{lemma}

The two lemmas above imply the following result, analogous to \cite[Thm. 2.21]{Gui21}.

\begin{theorem}\label{Thm: ConformalStructureOfFusion}
    Let $\varphi,\mu\in \Aut(\A)$ be automorphisms and $H^\varphi\in \Rep^\varphi(\A)$, $K^\mu\in \Rep^\mu(\A)$ be twisted representations of $\A.$ Then, for any $\tilde{I}\in\Jcal_\R$, the unitary representation $U^{H\boxtimes K}$ of $\Diff^+_\A(S^1)$ defining the conformal structure of $H^\varphi_+(\tilde{I})\boxtimes K^\mu$ can be described as follows. Given $(\tilde{f},f,V)\in \Diff_\A^+(S^1)$, we choose a map $\lambda: [0,1]\to \Diff_\A^+(S^1)$ from 1 to $(\tilde{f},f,V)$ such that $\lambda$ descends to a continuous path $[\lambda]$ on $\widetilde{\Diff^+}(S^1)$. Pick any $z\in \tilde{I}$ and let $\lambda_z$ be the path in $\R$ given by $t\mapsto  \text{proj}_{\Diff(\R)}(\lambda(t))(z)$. Then, for any $\xi\in H^\varphi_+(\tilde{I})$ and $\eta\in K^\mu$,
    \[
    U^{H\boxtimes K}(\tilde{f},f,V)(\xi\otimes\eta) = (\lambda_z^\bullet)^{-1}\big(U^H(\tilde{f},f,V)\,\xi\,U(\tilde{f},f,V)^{-1}\otimes U^K(\tilde{f},f,V)\eta\big),
    \]
    where $U^H(\tilde{f},f,V)\,\xi\,U(\tilde{f},f,V)^{-1} = U^H(\tilde{f},f,V)\circ Z^+(\xi,\tilde{I})\circ U(\tilde{f},f,V)^{-1}(\Omega)$ and $\lambda_z^\bullet: H^\varphi_+(\tilde{I})\boxtimes K^\mu\to H^\varphi_+(\tilde{f}\tilde{I})\boxtimes K^\mu$ is the path continuation induced by $\lambda_z$.
\end{theorem}

We can now directly apply the theorem above to compute the action $e^{-2\pi i L_0}$ on the Connes fusion $H_+^\varphi(\tilde{I})\boxtimes K^\mu$ of two twisted representations over a given interval $\tilde{I}\in \Jcal_\R$.

\begin{corollary}\label{cor: e2piiL0OnConnesFusion}
    Let $z\in \tilde{I}$ be arbitrary and $\lambda_z$ be the path in $\R$ from $\tilde{I}$ to $\tilde{I}-1$ given by $t\in[0,1]\mapsto z-t$. Then, for any $\xi\in H^\varphi_+(\tilde{I})$ and $\eta\in K^\mu$, it holds that
\[
e^{-2\pi i L_0}(\xi\otimes \eta) = (\lambda_z^{\bullet})^{-1}(e^{-2\pi i L_0}\xi\otimes e^{-2\pi i L_0}\eta).
\]
\end{corollary}
\begin{proof}
    The claim follows from Theorem \ref{Thm: ConformalStructureOfFusion} and the fact that, if $(\tilde{f},f,V)\in \Diff_\A^+(S^1)$ is in the preimage of $\widetilde{\Mob}$, then $U^H(\tilde{f},f,V)\,\xi\,U(\tilde{f},f,V)^{-1} = U^H(\tilde{f},f,V)\,\xi$, since $\Omega$ is fixed by $\Mob$.
\end{proof}

\subsection{The crossed balanced $\mathrm{W}^*$-structure}

In this section, we equip the category $$\Rep^{\Aut(\A)}(\A):= \bigoplus\limits_{\varphi\in \Aut(\A)}\Rep^\varphi(\A)$$ with the structure of an $\Aut(\A)$-crossed balanced $\mathrm{W}^*$-tensor category.

Let $\widetilde{S_-^1}:=(-\nicefrac{1}{2}, 0)$ and $\widetilde{S_+^1}:=(0,\nicefrac{1}{2})$, which are both intervals in $\Jcal_\R$ mapping respectively to the lower and upper hemispheres of $S^1$ under $q:\R\to S^1$. Given automorphisms $\varphi,\mu\in \Aut(A)$, for any twisted representations $H^\varphi\in \Rep^\varphi(\A)$ and $K^\mu\in \Rep^\mu(\A)$, we define their tensor product to be $$H^\varphi\boxtimes K^\mu:=H^\varphi_+(\widetilde{S^1_-})\boxtimes K^\mu_-(\widetilde{S^1_+}).$$ We also identify $H^\varphi\boxtimes K^\mu$ with $H_+^\varphi(\widetilde{S^1_-})\boxtimes K^\mu$ and $H^\varphi\boxtimes K_-^\mu(\widetilde{S^1_+})$ using the canonical equivalences. We extend this tensor product linearly to obtain a functor
\[
\boxtimes: \Rep^{\Aut(\A)}(\A)\times \Rep^{\Aut(\A)}(\A)\to \Rep^{\Aut(\A)}(\A).
\]
Let us denote $\widetilde{S^1_-}$ and $\widetilde{S^1_+}$ by $-$ and $+$ in the Connes fusions. Given three automorphisms $\varphi,\mu,\nu\in \Aut(\A)$ and representations $H^\varphi,K^\mu,R^\nu$ of $\A$ twisted by $\varphi,\mu$ and $\nu$ respectively, we define the unitary associator $(H^\varphi\boxtimes K^\mu)\boxtimes R^\nu\xrightarrow{\cong}H^\varphi\boxtimes (K^\mu\boxtimes R^\nu)$ to be the composition
\[
(H^\varphi_+(-)\boxtimes K^\mu)\boxtimes R_-^\nu(+)\xrightarrow{\cong }  H_+^\varphi(-)\boxtimes K^\mu\boxtimes R^\nu-(+)   \xrightarrow{\cong}H^\varphi_+(-)\boxtimes (K^\mu\boxtimes R_-^\nu(+))
\]
produced in Section \ref{Sec: Associativity}, which clearly satisfies naturality. Analogously to \cite[Fig. 2.3]{Gui21}, the associator satisfies the pentagon equations.

We define the unit object to be $H_0$, and the left and right unitors evaluated on a $\varphi$-twisted representation $H^\varphi$ to be
\[
\sharp_H: H^\varphi\boxtimes H_0 = H^\varphi_+(-)\boxtimes H_0\xrightarrow{\natural_H} H^\varphi
\]
\[
\flat_H: H_0\boxtimes H^\varphi = H_0\boxtimes H^\varphi_-(+) \xrightarrow{\cong} H_+^\varphi(+)\boxtimes T_{\varphi^{-1}}(H_0)\xrightarrow{\id\boxtimes V_{\varphi^{-1}}} H_+^\varphi(+)\boxtimes H_0 \xrightarrow{\natural_H} H^\varphi,
\]
where we use the canonical unitaries from Theorem \ref{thm: PreUnitors}.
\begin{proposition}
    The left and right unitors $\sharp_H$ and $\flat_H$ satisfy the triangle identity.
\end{proposition}
\begin{proof}
   We have to show that the following diagram commutes. 
\[\begin{tikzcd}
	{H^\varphi_+(-)\boxtimes H_0\boxtimes K^\mu_-(+)} & {H^\varphi_+(-)\boxtimes (K_+^\mu(-)\boxtimes T_{\mu^{-1}}(H_0))} \\
	{H^\varphi_+(-)\boxtimes K^\mu_-(+)} & {H_+^\varphi(-)\boxtimes (K^\mu_+(-)\boxtimes H_0)}.
	\arrow["\cong", from=1-1, to=1-2]
	\arrow["{\natural_H\boxtimes \id}"', from=1-1, to=2-1]
	\arrow["{\id\boxtimes V_{\mu^{-1}}}", from=1-2, to=2-2]
	\arrow["{\id\boxtimes \natural_K}", from=2-2, to=2-1]
\end{tikzcd}\]
Fix $\tilde{I}\Subset \widetilde{S^1_-}$, $\tilde{J}\Subset\widetilde{S^1_+}$, and let $\tilde{L}\in\Jcal_\R$ such that $\tilde{L}\subset \tilde{I}^{c+}\cap \tilde{J}^{c-}$. Then, vectors of the form $\xi\otimes \pi_{0,
L}(x)\Omega\otimes \eta$, with $\xi\in H_+^\varphi(\tilde{I})$, $\eta\in K^\mu_-(\tilde{J})$ and $x\in \A(L)$, span a dense subset of $H^\varphi_+(-)\boxtimes H_0\boxtimes K^\mu_-(+)$. Pick $\xi\in H_+^\varphi(\tilde{I})$, $\eta\in K^\mu_-(\tilde{J})$ and $x\in \A(L)$. Then, the left leg of the diagram applied to $\xi\otimes \pi_{0,L}(x)\Omega\otimes \eta$ reads $\pi^H_{\tilde{I}^{c+}}(x)(\xi)\otimes \eta = \pi^H_{\tilde{L}}(x)(\xi)\otimes \eta$. Similarly, the top-right-bottom leg applied to the same vector reads
    \begin{align*}
    \xi\otimes \pi_{0,L}(x)\Omega\otimes \eta&\mapsto \xi\otimes \eta\otimes \pi_{0,L}(x)\Omega \\ &\mapsto \xi\otimes \eta\otimes V_{\mu^{-1}}\pi_{0,L}(x)\Omega = \xi\otimes\eta\otimes \pi_{0,L}(\mu^{-1}x)\Omega \\& \mapsto \xi\otimes \pi^K_{\tilde{J}^{c+}}(\mu^{-1}x)(\eta)= \xi\otimes \pi^K_{\tilde{L}}(x)(\eta).
    \end{align*}
We can easily construct path-continuations that show that both sides equal $\pi^{H\boxtimes K}_{\tilde{L}}(x)(\xi\otimes~\eta)$.
\end{proof}

The functors $T_\varphi$ for $\varphi\in\Aut(\A)$ provide a monoidal action $$T: \underline{\Aut(\A)}\to \underline{\Aut}_\otimes(\Rep^{\Aut(\A)}(\A)).$$ Given $\nu\in \Aut(\A)$, the monoidal structure of $T_\nu$ is given, using Proposition \ref{prop: ActionData}, by
\[
T_\nu(H^\varphi)\boxtimes T_\nu(K^\mu) =  T_\nu\big(H^\varphi)_+(-)\boxtimes T_\nu\big(K^\mu\big)_-(+)\cong     T_\nu\big(H^\varphi_+(-)\boxtimes K^\mu_-(+)\big) =   T_\nu(H^\varphi\boxtimes K^\mu),
\]
which is easily seen to be compatible with the associator. The monoidal structure of the functor $T$ is given, on $H^\varphi\in\Rep^\varphi(\A)$, by
\[
T_\mu\circ T_\nu(H^\varphi) = T_{\mu\nu}(H^\varphi).
\]

We define next the $\Aut(\A)$-crossed braiding on $\Rep^{\Aut(\A)}(\A)$. Let $\varrho:[0,1]\to \R$ be the path $\varrho(t) = -\nicefrac{1}{4} + \nicefrac{t}{2}$, which goes from $\widetilde{S_-^1}$ to $\widetilde{S_+^1}$. Given twisted representations $H^\varphi\in\Rep^\varphi(\A)$ and $K^\mu\in \Rep^\mu(\A)$, we define the crossed-braid operator $\B_{H,K}: H^\varphi\boxtimes K^\mu\xrightarrow{\cong} T_\varphi(K^\mu)\boxtimes H^\varphi$~by
\[
\B_{H,
K}: H^\varphi\boxtimes K^\mu = H^\varphi_+(-)\boxtimes K^\mu\xrightarrow{\varrho^\bullet} H^\varphi_+(+)\boxtimes K^\mu\cong T_\varphi(K^\mu)\boxtimes H^\varphi_-(+) = T_\varphi(K^\mu)\boxtimes H^\varphi,
\]
where we use the unitary defined in Proposition \ref{prop: SwapWithAction}. The crossed-braid operators are clearly natural in $H^\varphi$ and $K^\mu$. Hence, we obtain an $\Aut(\A)$-crossed braiding on $\Rep^{\Aut(\A)}$ if we prove the crossed-braid relations, as follows.  

\begin{theorem}\label{Thm: Aut(A)-X-Braiding}
    Fix automorphisms $\varphi,\mu,\nu\in\Aut(\A)$ and let $H^\varphi\in\Rep^\varphi(\A)$, $K^\mu\in\Rep^\mu(\A)$, $R^\nu\in\Rep^\nu(\A)$ be twisted representations of $\A$. Then, for the action $T$ of $\Aut(\A)$ on $\Rep^{\Aut(\A)}(\A)$ and the unitaries $\B$, the following diagrams commute
\[\begin{tikzcd}
	{T_\nu(H^\varphi\boxtimes K^\mu)} && {T_\nu(T_\varphi(K^\mu)\boxtimes H^\varphi)} \\
	{T_\nu(H^\varphi)\boxtimes T_\nu(K^\mu)} && {T_{\nu\varphi\nu^{-1}}\circ T_\nu(K^\mu)\boxtimes T_\nu(H^\varphi)} \\
	\\
	{H^\varphi\boxtimes K^\mu\boxtimes R^\nu} && {T_\varphi(K^\mu\boxtimes R^\nu)\boxtimes H^\varphi} \\
	{T_\varphi(K^\mu)\boxtimes H^\varphi \boxtimes R^\nu} && {T_\varphi(K^\mu)\boxtimes T_\varphi(R^\nu)\boxtimes H^\varphi} \\
	\\
	{H^\varphi\boxtimes K^\mu\boxtimes R^\nu} && {T_{\varphi\mu}(R^\nu)\boxtimes H^\varphi\boxtimes K^\mu} \\
	{H^\varphi\boxtimes T_\mu(R^\nu)\boxtimes K^\mu} && {T_\varphi\circ T_{\mu}(R^\nu)\boxtimes H^\varphi\boxtimes K^\mu}.
	\arrow["{T_\nu(\B_{H,K})}", from=1-1, to=1-3]
	\arrow["{\cong }"', from=1-1, to=2-1]
	\arrow["\cong", from=1-3, to=2-3]
	\arrow["{\B_{T_\nu(H^\varphi), T_\nu(K^\mu)}}"', from=2-1, to=2-3]
	\arrow["{\B_{H, K\boxtimes R}}", from=4-1, to=4-3]
	\arrow["{\B_{H,K}\boxtimes \id}"', from=4-1, to=5-1]
	\arrow["{\cong }", from=4-3, to=5-3]
	\arrow["{\id\boxtimes \B_{H, R}}"', from=5-1, to=5-3]
	\arrow["{\B_{H\boxtimes K, R}}", from=7-1, to=7-3]
	\arrow["{\id\boxtimes \B_{K,R}}"', from=7-1, to=8-1]
	\arrow["{=}", from=7-3, to=8-3]
	\arrow["{\B_{H, T_\mu(R^\nu)}\boxtimes\id}"', from=8-1, to=8-3]
\end{tikzcd}\]
\end{theorem}

To prove the cross-braided relations of Theorem \ref{Thm: Aut(A)-X-Braiding}, we introduce the following maps, following \cite{Gui21}. Let $H^\varphi$ be a $\varphi$-twisted representation of $\A$, and $\tilde{I}\in \Jcal_\R$ an interval. For any $K^\mu\in \Rep^\mu(\A)$ and $\xi\in H_+^\varphi(\tilde{I})$, we denote by $Z^+(\xi, \tilde{I})$ the bounded linear operator $K^\mu\to H^\varphi_+(\tilde{I})\boxtimes K^\mu$ induced by
\[
\begin{array}{cccc}
     Z^+(\xi, \tilde{I}):& K^\mu &\to &H_+^\varphi(\tilde{I})\otimes K^\mu \\
     & \eta&\mapsto &\xi\otimes \eta.
\end{array}
\]
It is clear that $Z^+(\xi, \tilde{I})$ commutes with the actions of $\A_\R(\tilde{I}^{c+})$. We will similarly write $Z^-(\xi,\tilde{I}):K^\mu\to K^\mu\boxtimes H^\varphi_-(\tilde{I})$ for the map induced by $\eta\mapsto \eta\otimes\xi$. Let $\alpha_{\tilde{I}}:[0,1]\to \R$ be a path in $\R$ from $\tilde{I}$ to $-\nicefrac{1}{4}\in \widetilde{S^1_-}$. We write $L(\xi, \tilde{I}): K^\mu\to H^\varphi\boxtimes K^\mu$ for the map
\begin{equation}\label{eq: OperatorL}
L(\xi,\tilde{I}): K^\mu\xrightarrow{Z^+(\xi, \tilde{I})} H_+^\varphi(\tilde{I})\boxtimes K^\mu\xrightarrow{\alpha_{\tilde{I}}^\bullet} H^\varphi_+(-)\boxtimes K^\mu = H^\varphi\boxtimes K^\mu.
\end{equation}
Given $\xi'\in H^\varphi_-(\tilde{I})$, we write $R(\xi', \tilde{I}): K^\mu\to K^\mu\boxtimes H^\varphi$ for the map
\begin{equation}
\label{eq: OperatorR}
R(\xi', \tilde{I}): K^\mu\xrightarrow{Z^-(\xi', \tilde{I})}K^\mu\boxtimes H_-^\varphi(\tilde{I})\xrightarrow{(\varrho*\alpha_{\tilde{I}})^\bullet} K^\mu\boxtimes H_-^\varphi(+) = K^\mu\boxtimes H^\varphi.
\end{equation}
Since path continuations intertwine the actions of $\A_\R$, we have that $L(\xi, \tilde{I})\in \Hom_{\A_\R(\tilde{I}^{c+})}(K^\mu, H^\varphi\boxtimes K^\mu)$ and $R(\xi, \tilde{I})\in \Hom_{\A_\R(\tilde{I}^{c-})}(K^\mu, K^\mu\boxtimes H^\varphi)$. In addition, the operators $L$ and $R$ are related between themselves and with the family $\mathbb{B}$ as described by the following two propositions.

\begin{proposition}\label{prop: CompatibilityLRBRaiding}
    Fix automorphisms $\varphi,\mu\in \Aut(\A)$ and let $H^\varphi\in \Rep^\varphi(\A)$ and $K^\mu\in\Rep^\mu(\A)$ be twisted representations. Then, for every $\tilde{I}\in\Jcal_\R$ and every $\xi\in H^\varphi_-(\tilde{I})$, it holds that $\Gamma_\varphi\xi\in T_\varphi(H^\varphi)_+(\tilde{I})$ and the following diagram commutes
\[\begin{tikzcd}
	{K^\mu} && {T_\varphi(H^\varphi)\boxtimes K^\mu} \\
	{K^\mu\boxtimes H^\varphi} && {T_\varphi(K^\mu)\boxtimes T_\varphi(H^\varphi)} \\
	& {T_\varphi(K^\mu\boxtimes H^\varphi),}
	\arrow["{L(\Gamma_\varphi\xi, \tilde{I})}", from=1-1, to=1-3]
	\arrow["{R(\xi, \tilde{I})}"', from=1-1, to=2-1]
	\arrow["{\mathbb{B}_{T_\varphi(H), K}}", from=1-3, to=2-3]
	\arrow["{\Gamma_\varphi}"', from=2-1, to=3-2]
	\arrow["\cong", from=2-3, to=3-2]
\end{tikzcd}\]
where the unique unlabelled morphism is the unitary in Proposition \ref{prop: ActionData}.
\end{proposition}
\begin{proof}
    Given $\xi = Z^-(\xi, \tilde{I})\Omega\in H^\varphi_-(\tilde{I})$, we have already argued that $\Gamma_\varphi\xi\in T_\varphi(H^\varphi)_+(\tilde{I})$ and, in addition $Z^+(\Gamma_\varphi\xi, \tilde{I}) = \Gamma_\varphi\circ Z^+(\xi, \tilde{I})\circ V_{\varphi^{-1}} = \Gamma_\varphi\circ Z^-(\xi, \tilde{I})$ as a morphism $H_0\to T_\varphi(H^\varphi)$. 
    
    Let $\Phi:K^\mu\boxtimes H^\varphi_-(\tilde{I})\to T_\varphi(H^\varphi)_+(\tilde{I})\boxtimes K^\mu$ be the unitary induced by the swap map $K^\mu\otimes  H^\varphi_-(\tilde{I})\to T_\varphi(H^\varphi)_+(\tilde{I})\otimes K^\mu$ given by $\eta\otimes\chi\mapsto \Gamma_\varphi\chi\otimes\eta$. Similarly, we define $\Phi': K^\mu\boxtimes H^\varphi_-(+)\to T_\varphi(H^\varphi)_+(+)\boxtimes K^\mu$ as the unitary induced by the map $K^\mu\otimes H^\varphi_-(\widetilde{S^1_+})\to T_\varphi(H^\varphi)_+(\widetilde{S^1_+})\boxtimes K^\mu$ given by $\eta\otimes\chi\mapsto \Gamma_\varphi\chi\otimes\eta$. The maps $\Phi$ and $\Phi'$ are well-defined maps of Hilbert spaces by the same arguments as the proof of Proposition \ref{prop: SwapWithAction}. Let $\alpha_{\tilde{I}}$ be a path in $\R$ from $\tilde{I}$ to $\widetilde{S^1_-}$. We have to argue the commutativity of the outer diagram~in
\[\begin{tikzcd}
	{K^\mu} && {T_\varphi(H^\varphi)_+(\tilde{I})\boxtimes K^\mu} && {T_\varphi(H^\varphi)_+(-)\boxtimes K^\mu} \\
	{K^\mu\boxtimes H^\varphi_-(\tilde{I})} &&&& {T_\varphi(H^\varphi)_+(+)\boxtimes K^\mu} \\
	\\
	{K^\mu\boxtimes H^\varphi_-(+)} && {T_\varphi(K^\mu\boxtimes H^\varphi)} && {T_\varphi(K^\mu)\boxtimes T_\varphi(H^\varphi)_-(+)}
	\arrow["{Z^+(\Gamma_\varphi\xi, \tilde{I})}", from=1-1, to=1-3]
	\arrow["{Z^-(\xi, \tilde{I})}"', from=1-1, to=2-1]
	\arrow["{\alpha_{\tilde{I}}^\bullet}", from=1-3, to=1-5]
	\arrow["{(\alpha_{\tilde{I}}\ast \varrho)^\bullet}"', from=1-3, to=2-5]
	\arrow["{\varrho^\bullet}", from=1-5, to=2-5]
	\arrow["\Phi"', from=2-1, to=1-3]
	\arrow["{(\alpha_{\tilde{I}}\ast \varrho)^\bullet}"', from=2-1, to=4-1]
	\arrow["\cong", from=2-5, to=4-5]
	\arrow["{\Phi'}", from=4-1, to=2-5]
	\arrow["{\Gamma_\varphi}"', from=4-1, to=4-3]
	\arrow["\cong"', from=4-3, to=4-5]
\end{tikzcd}\]

 By the equality $Z^+(\Gamma_\varphi\xi, \tilde{I}) =  \Gamma_\varphi\circ Z^-(\xi, \tilde{I})$, the top-left triangle of the diagram commutes. The top-right diagram commutes by compatibility of path continuations with concatenation of paths. The middle quadrilateral commutes by compatibility of $\Phi$ and $\Phi'$ with path continuations, analogously to the proof of Proposition \ref{prop: SwapWithAction}. For the commutativity of the bottom quadrilateral, let $\eta\otimes\chi\in K^\mu\boxtimes H^\varphi_-(+)$. Then, both legs applied to $\eta\otimes\chi$ read $\Gamma_\varphi\eta\otimes\Gamma_\varphi\chi$, and hence the bottom quadrilateral also commutes.
\end{proof}

The second piece of compatibility between the operators $L$ and $R$ is the following proposition.

\begin{proposition}\label{prop: LandRCommute}
    Let $H^\varphi, K^\mu,R^\nu$ be representations of $\A$ twisted by automorphisms $\varphi, \mu$ and $\nu$ respectively. Let $\tilde{I},\tilde{J}\in \Jcal_\R$ such that $\tilde{J}\subset \tilde{I}^{c+}$. Let $\xi\in H^\varphi_+(\tilde{I})$ and $\eta\in K^\mu_-(\tilde{J})$. Then, the following diagram commutes,
\begin{equation}\label{eq: CommutativityRL}\begin{tikzcd}
	{R^\nu} & {R^\nu\boxtimes K^\mu} \\
	{H^\varphi\boxtimes R^\nu} & {H^\varphi\boxtimes R^\nu\boxtimes K^\mu}.
	\arrow["{R(\eta,\tilde{J})}", from=1-1, to=1-2]
	\arrow["{L(\xi,\tilde{I})}"', from=1-1, to=2-1]
	\arrow["{L(\xi,\tilde{I})}", from=1-2, to=2-2]
	\arrow["{R(\eta,\tilde{J})}"', from=2-1, to=2-2]
\end{tikzcd}\end{equation}
\end{proposition}
\begin{proof}
    It is clear that the diagram 
\[\begin{tikzcd}
	{R^\nu} && {R^\nu\boxtimes K^\mu_-(\tilde{J})} \\
	{H^\varphi_+(\tilde{I})\boxtimes R^\nu} && {H^\varphi_+(\tilde{I})\boxtimes R^\nu\boxtimes K^\mu_-(\tilde{J})}
	\arrow["{Z^-(\eta, \tilde{J})}", from=1-1, to=1-3]
	\arrow["{Z^+(\xi, \tilde{I})}"', from=1-1, to=2-1]
	\arrow["{Z^+(\xi, \tilde{I})}", from=1-3, to=2-3]
	\arrow["{Z^-(\eta, \tilde{J})}"', from=2-1, to=2-3]
\end{tikzcd}\]
commutes, as both legs read $\chi\mapsto \xi\otimes\chi\otimes\eta$. Let $\alpha_{\tilde{I}}$ be a path in $\R$ from $\tilde{I}$ to $-\nicefrac{1}{4}\in \widetilde{S^1_-}$ and $\beta_{\tilde{J}}$ be a path in $\R$ from $\tilde{J}$ to $\nicefrac{1}{4}\in \widetilde{S^1_+}$. We can assume, by taking homotopic paths if needed, that $\gamma:=(\alpha_{\tilde{I}},\beta_{\tilde{J}})$ is a path in $\ConfR$. Note that here we use that $\tilde{J}\subset \tilde{I}^{c+}$. We have to prove the commutativity of the outer diagram of 
\[\begin{tikzcd}
	{R^\nu} && {R^\nu\boxtimes K^\mu_-(\tilde{J})} && {R^\nu\boxtimes K^\mu_-({+})} \\
	{H^\varphi_+(\tilde{I})\boxtimes R^\nu} && {H^\varphi_+(\tilde{I})\boxtimes R^\nu\boxtimes K^\mu_-(\tilde{J})} && {H^\varphi_+(\tilde{I})\boxtimes (R^\nu\boxtimes K^\mu_-({+}))} \\
	{H^\varphi_+(-)\boxtimes R^\nu} && {(H^\varphi_+(-)\boxtimes R^\nu)\boxtimes K^\mu_-(\tilde{J})} && {H^\varphi_+(-)\boxtimes R^\nu\boxtimes K^\mu_-({+})}.
	\arrow["{Z^-(\eta, \tilde{J})}", from=1-1, to=1-3]
	\arrow["{Z^+(\xi, \tilde{I})}"', from=1-1, to=2-1]
	\arrow["{\beta_{\tilde{J}}^\bullet}", from=1-3, to=1-5]
	\arrow["{Z^+(\xi, \tilde{I})}", from=1-3, to=2-3]
	\arrow["{Z^+(\xi, \tilde{I})}", from=1-5, to=2-5]
	\arrow["{Z^-(\eta, \tilde{J})}", from=2-1, to=2-3]
	\arrow["{\alpha_{\tilde{I}}^\bullet}"', from=2-1, to=3-1]
	\arrow["{\id\boxtimes \beta_{\tilde{J}}^\bullet}", from=2-3, to=2-5]
	\arrow["{\alpha_{\tilde{I}}^\bullet\boxtimes \id}", from=2-3, to=3-3]
	\arrow["{\alpha_{\tilde{I}}^\bullet}", from=2-5, to=3-5]
	\arrow["{Z^-(\eta, \tilde{J})}", from=3-1, to=3-3]
	\arrow["{\beta_{\tilde{J}}^\bullet}", from=3-3, to=3-5]
\end{tikzcd}\]
The top-left diagram commutes by the discussion above. The off-diagonal diagrams commute trivially. For the bottom-right diagram, commutativity follows from Proposition \ref{prop: BreakGammaInComponents}.
\end{proof}

Given $H^\varphi\in \Rep^\varphi(\A)$ and $K^\mu\in \Rep^\mu(\A)$, and $\tilde{I},\tilde{J}\in \Jcal_\R$ such that $\tilde{J}\subset \tilde{I}^{c+}$, we have that, for all $\xi\in H_-^\varphi(\tilde{I})$ and $\eta\in K^\mu_+(\tilde{J})$,
\begin{equation}\label{eq: Lbraiding1}
    L(\Gamma_\mu\xi,\tilde{I})\eta = \B_{K, H}\circ L(\eta,\tilde{J})\xi.
\end{equation}
Indeed, $L(\Gamma_\mu\xi, \tilde{I})\eta = L(\Gamma_\mu\xi, \tilde{I})R(\eta, \tilde{J})\Omega = R(\eta, \tilde{J})L(\Gamma_\mu\xi, \tilde{I})\Omega = \Gamma_\mu^{-1}\Gamma_\mu R(\eta, \tilde{J})\Gamma_\mu\xi = \Gamma_\mu^{-1}\circ \B_{ T_\mu(K), T_\mu(H)}\circ L(\Gamma_\mu\eta, \tilde{J})\Gamma_\mu\xi = \B_{ K, H}\circ L(\eta, \tilde{J})\xi$, where we have dropped the compatibility of $T$ with the Connes fusion for readability. We have also used that $L(\Gamma_\mu\eta, \tilde{J})\Gamma_\mu\xi = \Gamma_\mu L(\eta, \tilde{J})\xi$ by Proposition \ref{prop: ActionData} and that the diagram
\[\begin{tikzcd}
	{K^\mu\boxtimes H^\varphi} && {T_\mu(H^\varphi)\boxtimes K^\mu} \\
	{T_\mu(K^\mu\boxtimes H^\varphi)} && {T_\mu(T_\mu(H^\varphi)\boxtimes K^\mu)} \\
	{T_\mu(K^\mu)\boxtimes T_\mu(H^\varphi)} && {T_\mu\circ T_\mu(H^\varphi)\boxtimes T_\mu(K^\mu)}
	\arrow["{\mathbb{B}_{K, H}}", from=1-1, to=1-3]
	\arrow["{\Gamma_\mu}"', from=1-1, to=2-1]
	\arrow["{\Gamma_\mu}", from=1-3, to=2-3]
	\arrow["{\cong }"', from=2-1, to=3-1]
	\arrow["{\cong }", from=2-3, to=3-3]
	\arrow["{\mathbb{B}_{T_\mu(K), T_\mu(H)}}"', from=3-1, to=3-3]
\end{tikzcd}\]
trivially commutes. Hence, if $\tilde{I}\subset\tilde{J}^{c+}$, and $\xi\in H^\varphi_+(\tilde{I})$, $\eta\in K^\mu_-(\tilde{J})$, we have
\begin{equation}\label{eq: Lbraiding2}
L(\xi,\tilde{I})\eta = \B_{H,K}^{-1}\circ L(\Gamma_\varphi\eta,\tilde{J})\xi.
\end{equation}

It is also straightforward to see that, if $H^\varphi,\hat{H}^\varphi\in\Rep^\varphi(\A)$ and $K^\mu, \hat{K}^\mu\in \Rep^\mu(\A)$ are twisted representations and $F\in\Hom_{\Rep^\varphi(\A)}(H^\varphi, \hat{H}^\varphi)$ and $G\in\Hom_{\Rep^\mu(\A)}(K^\mu,\hat{K}^\mu)$ are morphisms between them, it holds for all $\tilde{I}\in\Jcal_\R$, $\xi\in H^\varphi_+(\tilde{I})$ and $\eta\in K^\mu$ that
\[
(F\boxtimes G)L(\xi,\tilde{I})\eta = L(F\xi,\tilde{I})G\eta.
\]

The following two propositions follow from the same arguments as \cite[Prop. 3.6]{Gui21} and \cite[Prop. 3.7]{Gui21}.
\begin{proposition}\label{prop: JoinLs}
    Let $\tilde{I},\tilde{J},\tilde{O}\in \Jcal_\R$ such that $\tilde{I},\tilde{J}\subset \tilde{O}$ and $\xi\in H_+^\varphi(\tilde{I})$, $\eta\in K^\mu_+(\tilde{J})$. Then, $L(\xi,\tilde{I})\eta\in (H^\varphi\boxtimes K^\mu)_+(\tilde{O})$ and
    \[
    L(\xi,\tilde{I})L(\eta,\tilde{J}) = L(L(\xi,\tilde{I})\eta,\tilde{O})
    \]
    when acting on any $R^\nu\in \Rep^\nu(\A).$
\end{proposition}
\begin{proposition} \label{prop: SwapLsGetBraiding}
    Let $\tilde{I},\tilde{J}\in \Jcal_\R$ with $\tilde{J}\subset\tilde{I}^{c+}$. Suppose that there exists an interval $\tilde{O}\in \Jcal_\R$ with $\tilde{I},\tilde{J}\subset \tilde{O}$. Then, for any twisted representations $H^\varphi$, $K^{\mu}$ and $R^\nu$, and vectors $\xi\in H_-^\varphi(\tilde{I})$, $\eta\in K^\mu_+(\tilde{J})$ and $\chi\in R^\nu$, we have
    \[
    L(\Gamma_\mu\xi,\tilde{I})L(\eta,\tilde{J})\chi = ( \B_{K,H}\boxtimes\id_{R})\circ L( \eta,\tilde{J})L(\xi,\tilde{I})\chi.
    \]
\end{proposition}
\begin{proof}
    We compute
    \begin{align*}
        L(\Gamma_\mu\xi,\tilde{I})L(\eta, \tilde{J})\chi & = L(L(\Gamma_\mu\xi, \tilde{I})\eta, \tilde{O})\chi\\& = L(\B_{K, H}\circ L(\eta, \tilde{J})L(\xi, \tilde{I}), \tilde{O})\chi \\ & = (\B_{K, H}\boxtimes \id_R)\circ L(\eta, \tilde{J})L(\xi, \tilde{I})\chi,
    \end{align*}
    where we have used Equation \eqref{eq: Lbraiding1}.
\end{proof}

We are now ready to prove the crossed-braided relations in Theorem \ref{Thm: Aut(A)-X-Braiding}.

\begin{proof}[Proof of Theorem \ref{Thm: Aut(A)-X-Braiding}]
    Let $\tilde{I},\tilde{J}\in \Jcal_\R$ such that $\tilde{J}\subset\tilde{I}^{c-}$. 
    
    We argue the commutativity of the first diagram. Let $\xi\in H^\varphi_+(\tilde{I})$ and $\eta\in K^\mu_-(\tilde{J})$. Then, $\Gamma_\nu L(\xi,\tilde{I})\eta\in T_\nu(H^\varphi\boxtimes K^\mu)$ , and vectors of this type span a dense subset of $T_\nu(H^\varphi\boxtimes K^\mu)$. Then, the top-right leg of the first diagram, applied to $\Gamma_\nu L(\xi,\tilde{I})\eta$, reads, using Equation \eqref{eq: Lbraiding2} $$\Gamma_\nu L(\xi,\tilde{I})\eta\xmapsto{T_\nu(\B_{H,K})} T_\nu(\mathbb{B}_{H,K})\Gamma_\nu L(\xi,\tilde{I})\eta  = \Gamma_\nu\B_{H, K}L(\xi, \tilde{I})\eta= \Gamma_\nu L(\Gamma_\varphi\eta, \tilde{J})\xi\mapsto L(\Gamma_{\nu\varphi}\eta, \tilde{J})\Gamma_\nu\xi.$$ Similarly, the left-bottom leg applied to $\Gamma_\nu L(\xi, \tilde{I})$ reads $$\Gamma_\nu L(\xi,\tilde{I})\eta\mapsto L(\Gamma_\nu\xi,\tilde{I})\Gamma_\nu\eta\xmapsto{\B_{T_\nu(H),T_\nu(K)}} \B_{T_\nu(H), T_\nu(K)}L(\Gamma_\nu\xi,\tilde{I})\Gamma_\nu\eta = L(\Gamma_{\nu\varphi\nu^{-1}}\Gamma_\nu\eta,\tilde{J})\Gamma_\nu\xi.$$ Hence, the first diagram commutes. 
    
    To argue the commutativity of the other two diagrams, choose another interval $\tilde{L}\in \Jcal_\R$ such that $\tilde{L}\subset \tilde{J}^{c-}$. Assume in addition, and without loss of generality, that the union $\tilde{I}\cup\tilde{J}\cup\tilde{L}$ can be covered by an interval in $\Jcal_\R$. Let $\tilde{O}\in \Jcal_\R$ be an interval such that $\tilde{I}\cup\tilde{J}\subset \tilde{O}$ and $\tilde{L}\subset \tilde{O}^{c-}$. 

    We shall prove first the commutativity of the third diagram. Let $\xi\in H^\varphi_+(\tilde{I}),\ \eta\in K^\mu_+(\tilde{J})$ and $\chi\in R_-^\nu(\tilde{L})$. The action of the left-bottom side of the third diagram on $L(\xi,\tilde{I})L(\eta,\tilde{J})\chi$ is
    \[\hspace{-.2cm}L(\xi,\tilde{I})L(\eta,\tilde{J})\chi\xmapsto{\id\boxtimes\B_{K,R}}L(\xi,\tilde{I})\B_{K,R}L(\eta,\tilde{J})\chi = L(\xi,\tilde{I})L(\Gamma_\mu\chi,\tilde{L})\eta\xmapsto{\B_{H,T_\mu(R)}\boxtimes\id} L(\Gamma_\varphi\Gamma_\mu\chi,\tilde{L})L(\xi,\tilde{I})\eta,
    \]
    where we use Equation \eqref{eq: Lbraiding2} and Proposition \ref{prop: SwapLsGetBraiding}. Now, the top-right leg of the third diagram acts on $L(\xi,\tilde{I})L(\eta,\tilde{J})\chi$ by
    \[
    L(\xi,\tilde{I})L(\eta,\tilde{J})\chi = L\big(L(\xi,\tilde{I})\eta,\tilde{O}\big)\chi \xmapsto{\B_{H\boxtimes K,R}} \B_{H \boxtimes K, R}L\big(L(\xi,\tilde{I})\eta,\tilde{O}\big)\chi = L(\Gamma_{\varphi\mu}\chi, \tilde{L})L(\xi,\tilde{I})\eta,
    \]
    using Equation \eqref{eq: Lbraiding1}. We have proved the commutativity of the third diagram.

    For the second diagram, we shall prove the commutativity of the following square, which is equivalent to the crossed-braided relation,
\[\begin{tikzcd}
	{T_{\varphi}(R^\nu)\boxtimes T_{\varphi}(K^\mu)\boxtimes H^\varphi} && {T_\varphi(R^\nu\boxtimes K^\mu)\boxtimes H^{\varphi}} \\
	{T_\varphi(R^\nu)\boxtimes H^\varphi\boxtimes K^\mu} && {H^\varphi\boxtimes R^\nu\boxtimes K^\mu}.
	\arrow["{\cong }", from=1-1, to=1-3]
	\arrow["{\id\boxtimes \B_{H,K}^{-1}}"', from=1-1, to=2-1]
	\arrow["{\B_{H, R\boxtimes K}^{-1}}", from=1-3, to=2-3]
	\arrow["{\B_{H,R}^{-1}\boxtimes\id}"', from=2-1, to=2-3]
\end{tikzcd}\]
Let $\xi\in H^\varphi_+(\tilde{I})$, $\eta\in K^\mu_-(\tilde{J})$ and $\chi\in R^\nu_+(\tilde{L})$. The action of the left-bottom leg on $L(\Gamma_\varphi\chi,\tilde{L})L(\Gamma_\varphi\eta,\tilde{J})\xi$~is
\[\hspace{-.2cm}
L(\Gamma_\varphi\chi,\tilde{L})L(\Gamma_\varphi\eta,\tilde{J})\xi\xmapsto{\id\boxtimes\B_{H,K}^{-1}} L(\Gamma_\varphi\chi,\tilde{L})\B_{H,K}^{-1}L(\Gamma_\varphi\eta,\tilde{J})\xi  = L(\Gamma_\varphi\chi,\tilde{L})L(\xi,\tilde{I})\eta \xmapsto{\B_{H,R}^{-1}\boxtimes\id }L(\xi,\tilde{I})L(\chi,\tilde{L})\eta,
\]
where we use Equation \eqref{eq: Lbraiding2} and Proposition \ref{prop: SwapLsGetBraiding}. The top-right leg applied to $L(\Gamma_\varphi\chi,\tilde{L})L(\Gamma_\varphi\eta,\tilde{J})\xi$ reads
\[
L(\Gamma_\varphi\chi,\tilde{L})L(\Gamma_\varphi\eta,\tilde{J})\xi \mapsto L\big(\Gamma_\varphi L(\chi,\tilde{L})\eta,\tilde{O}\big)\xi\xmapsto{\B_{H,R\boxtimes K}^{-1}} L(\xi,\tilde{I})L(\chi,\tilde{L})\eta,
\] using Equation \eqref{eq: Lbraiding2}. Hence, the commutativity of the diagram is proven.
\end{proof}

\begin{remark}
    Restricting to the identity component of $\Rep^{\Aut(\A)}(\A)$, we obtain a braided $\mathrm{W}^*$-tensor structure on $\Rep(\A)$. We want to note that this braided tensor structure is the opposite to the one defined in \cite[Sec. 2]{Gui21}. In the context of $\Rep(\A)$, it is a choice to take our braided tensor structure or that of Gui. In the crossed braided setting, this choice disappears, and hence we obtain the tensor and the braiding on $\Rep(\A)$ that extends to an $\Aut(\A)$-crossed braided tensor structure on $\Rep^{\Aut(\A)}(\A).$
\end{remark}

We conclude this section by providing the $\Aut(\A)$-crossed braided tensor category $\Rep^{\Aut(\A)}(\A)$ with an $\Aut(\A)$-crossed balance. On a twisted representation $H^\varphi\in \Rep^\varphi(\A)$, we define the balance to be 
\[
\theta_H = e^{-2\pi i L_0}:H^\varphi\to T_\varphi(H^\varphi).
\]

\begin{theorem}\label{Thm: RepAutAIsCrossedBalanced}
    Let $\A$ be a conformal net. The $\mathrm{W}^*$-category $\Rep^{\Aut(\A)}(\A)$ of twisted representations of $\A$, equipped with the Connes fusion of twisted representations, the action $T$, the braiding $\B$ and the unitaries $\theta$ becomes an $\Aut(\A)$-crossed balanced $\mathrm{W}^*$-tensor category.
\end{theorem}
\begin{proof}
    We have argued that Connes fusion provides a tensor structure compatible with the action $T$. Theorem \ref{Thm: Aut(A)-X-Braiding} proves that $\B$ defines an $\Aut(\A)$-crossed braiding. It is only left to argue that $\theta$ provides a crossed balance. The family $\theta$ is natural with respect to morphisms of twisted $\A$-representations by Theorem \ref{Thm: TwistedRepIsConformal}.

    The first condition in Definition \ref{def: CrossedCategoricalTwist} is exactly Proposition \ref{prop: CompatibilityActionConformalStruct}. Let us show the second condition. Let $\varphi,\mu\in \Aut(\A)$ and $H^\varphi\in \Rep^\varphi(\A)$, $K^\mu\in \Rep^\mu(\A).$ We have to argue the commutativity of
\[\begin{tikzcd}
	{H^\varphi\boxtimes K^\mu} && {T_{\varphi\mu}(H^\varphi\boxtimes K^\mu)} \\
	&& {T_{\varphi\mu}(H^\varphi)\boxtimes T_{\varphi\mu}(K^\mu)} \\
	&& {T_{\varphi\mu\varphi\mu^{-1}\varphi^{-1}}T_{\varphi\mu\varphi^{-1}}(H^\varphi)\boxtimes T_{\varphi\mu\varphi^{-1}}T_\varphi(K^\mu)} \\
	{T_\varphi(K^\mu)\boxtimes H^\varphi} && {T_{\varphi\mu\varphi^{-1}}(H^\varphi)\boxtimes T_\varphi(K^\mu)}
	\arrow["{\theta_{H\boxtimes K}}", from=1-1, to=1-3]
	\arrow["{\B_{H,K}}"', from=1-1, to=4-1]
	\arrow["\cong", from=1-3, to=2-3]
	\arrow["{=}"', from=3-3, to=2-3]
	\arrow["{\B_{T_\varphi(K^\mu), H}}"', from=4-1, to=4-3]
	\arrow["{\theta_{T_{\varphi\mu\varphi^{-1}}(H^\varphi)}\boxtimes \theta_{T_\varphi(K^\mu)}}"', from=4-3, to=3-3]
\end{tikzcd}\]
Let $\varrho:[0,1]\to \R$ be the path $\varrho(t) = -\nicefrac{1}{4}+\nicefrac{t}{2}$, and $\varrho+\nicefrac{1}{2}:[0,1]\to \R$ be given by $(\varrho+\nicefrac{1}{2})(t) = \nicefrac{1}{4}+\nicefrac{t}{2}$, which is a path from $\widetilde{S^1_+}$ to $\widetilde{S^1_-}+1$. Note that $(\varrho, \varrho+\nicefrac{1}{2})$ gives a path in $\ConfR$, and we denote by $\tau:=(\varrho+\nicefrac{1}{2})*\varrho$ the concatenation, which reads $\tau(t) = -\nicefrac{1}{4}+t$ and is a path from $\widetilde{S^1_-}$ to $\widetilde{S^1_-}+1$. We will first compute the composition $\B_{T_\varphi(K^\mu), H}\circ \B_{H, K}$, which is by definition the top-right leg of the diagram
\[\begin{tikzcd}
	{H^\varphi_+(-)\boxtimes K^\mu} & {H^\varphi_+(+)\boxtimes K^\mu} & {T_\varphi(K^\mu)\boxtimes H^\varphi_-(+)} \\
	{H^\varphi_+(\widetilde{S^1_-}+1)\boxtimes K^\mu} & {} & {T_\varphi(K^\mu)_+(-)\boxtimes H^\varphi_-(+)} \\
	&& {T_\varphi(K^\mu)_+(-)\boxtimes H^\varphi} \\
	{T_\varphi(K^\mu)\boxtimes  H^\varphi_-(\widetilde{S^1_-}+1)} && {T_\varphi(K^\mu)_+(+)\boxtimes H^\varphi} \\
	{T_\varphi(K^\mu)_+(+)\boxtimes H^\varphi_-(\widetilde{S^1_-}+1)} && {T_{\varphi\mu\varphi{-1}}(H^\varphi)_+(-)\boxtimes T_\varphi(K^\mu)}.
	\arrow["{\varrho^\bullet}", from=1-1, to=1-2]
	\arrow["{\tau^\bullet}"', from=1-1, to=2-1]
	\arrow["\cong", from=1-2, to=1-3]
	\arrow["{(\varrho+\nicefrac{1}{2})^\bullet}", from=1-2, to=2-1]
	\arrow["{\cong }", from=1-3, to=2-3]
	\arrow["{(\varrho+\nicefrac{1}{2})^\bullet}"', from=1-3, to=4-1]
	\arrow["{\cong }"', from=2-1, to=4-1]
	\arrow["{\cong }", from=2-3, to=3-3]
	\arrow["{(\varrho, \varrho+\nicefrac{1}{2})^\bullet}"', from=2-3, to=5-1]
	\arrow["{\varrho^\bullet}", from=3-3, to=4-3]
	\arrow["{\cong }"', from=4-1, to=5-1]
	\arrow["{\cong }", from=4-3, to=5-3]
	\arrow["{\cong }"', from=5-1, to=4-3]
	\arrow[from=5-1, to=5-3]
\end{tikzcd}\]
We claim that the outer diagram commutes. Let us comment on the commutativity of the inner diagrams from top to bottom. The top-left triangle commutes by the compatibility of path continuations with concatenations. The next quadrilateral commutes by Proposition \ref{prop: SwapWithAction}. The next two quadrilaterals commute by Proposition \ref{prop: SinglePathContinuationWRTDoublePathContinuation}. We define the bottom horizontal arrow so that the bottom-right triangle commutes. Hence, the whole diagram commutes and the composition $\B_{T_\varphi(K^\mu), H}\circ \B_{H, K}$ is equal to 
\[
H^\varphi_+(-)\boxtimes K^\mu\xrightarrow{\tau^\bullet} H^\varphi(\widetilde{S^1_
-}+1)\boxtimes K^\mu\xrightarrow{\cong } T_{\varphi\mu\varphi^{-1}}(H^\varphi)_+(-)\boxtimes T_\varphi(K^\mu),
\]
where the second map is induced by the map $H^\varphi_+(\widetilde{S^1_-}+1)\otimes K^\mu\to T_{\varphi\mu\varphi^{-1}}(H^\varphi)_+(-)\otimes T_\varphi(K^\mu)$ given by $\xi\otimes\eta\mapsto \Gamma_{\varphi\mu\varphi^{-1}}\xi\otimes\Gamma_{\varphi}\eta$, for $\xi\in H^\varphi_+(\widetilde{S^1_-}+1) $ and $\eta\in K^\mu $. Hence, we have reduced the problem to showing the commutativity of the outer diagram in
\[\begin{tikzcd}
	{H^\varphi_+(-)\boxtimes K^\mu} && {H^\varphi_+(-)\boxtimes K^\mu} \\
	& {H^\varphi_-(\widetilde{S_-^1} - 1)\boxtimes K^\mu} \\
	\\
	&& {T_{\varphi\mu\varphi^{-1}}(H^\varphi)_+(\widetilde{S^1_-}-1)\boxtimes T_\varphi(K^\mu)} \\
	{H^\varphi_+(\widetilde{S^1_-}+1)\boxtimes K^\mu} && {T_{\varphi\mu\varphi^{-1}}(H^\varphi)_+(-)\boxtimes T_\varphi(K^\mu)}
	\arrow["{e^{-2\pi i L_0}}", from=1-1, to=1-3]
	\arrow["{e^{-2\pi i L_0}\boxtimes e^{-2\pi i L_0}}", from=1-1, to=2-2]
	\arrow["{\tau^\bullet}"', from=1-1, to=5-1]
	\arrow["{\cong }", from=1-3, to=4-3]
	\arrow["{(\tau - 1)^\bullet}"{pos=0.3}, from=2-2, to=1-3]
	\arrow["{e^{-2\pi i L_0}\boxtimes e^{-2\pi i L_0}}"{pos=0.3}, from=5-1, to=1-3]
	\arrow["\cong"', from=5-1, to=5-3]
	\arrow["{e^{-2\pi i L_0}\boxtimes e^{-2\pi i L_0}}"', from=5-3, to=4-3]
\end{tikzcd}\]
where both unlabelled equivalences are induced by the map $\xi\otimes\eta\mapsto \Gamma_{\varphi\mu\varphi^{-1}}\xi\otimes \Gamma_\varphi\eta$ between the relevant domains and targets. The top triangle in the diagram above commutes by Corollary \ref{cor: e2piiL0OnConnesFusion} and the fact that the path continuation $(\lambda^\bullet)^{-1}: H^\varphi_+(\widetilde{S^1_-}-1)\boxtimes K^\mu\to H^\varphi_+(-)\boxtimes K^\mu$ in the statement of the Corollary can be taken to be the path continuation $(\tau-1)^\bullet$, for $\tau-1:[0,1]\to \R$ the path $(\tau-1)(t) = -\nicefrac{5}{4}+t$. The left quadrilateral commutes by the fact that the action of $\Diff_A^+(S^1)$ commutes with the maps induced by inclusions and restrictions of intervals, and hence with path-continuations. The right quadrilateral commutes trivially. Hence, the claim follows.
\end{proof}

If we have a discrete group $G$ acting faithfully on $\A$ by an injective group homomorphism $\Phi:G\to \Aut(\A)$, we can pull back the $\Aut(\A)-$crossed balanced structure on $\Rep^{\Aut(\A)}(\A)$ along $\Phi$ to obtain a $G$-crossed balanced structure on $\Rep^G(\A)$.

\begin{theorem}\label{thm: BigCorollary}
    Let $G$ be a discrete group acting faithfully on a conformal net $\A$ by a group homomorphism $\Phi: G\to \Aut(\A)$. Then, the $G$-crossed braided $\mathrm{W}^*$-tensor category
    \[
    \Rep^G(\A):=\bigoplus_{g\in G}\Rep^{\Phi(g)}(\A)
    \]
    admits a structure of a $G$-crossed balanced $\mathrm{W}^*$-tensor category.
\end{theorem}

Applying Theorem \ref{thm: BigCorollary} to the case $G = \{e\}$ we obtain a balanced $\mathrm{W}^*$-tensor structure on~$\Rep(\A)$. An alternative, more accessible proof for the following corollary has appeared in the note \cite{marinsalvador2026balancedstructurecategoryrepresentations}.

\begin{corollary}
    Let $\A$ be a conformal net. The braided $\mathrm{W}^*$-tensor category $\Rep(\A)$ of representations of $\A$ admits a canonical structure of a balanced $\mathrm{W}^*$-tensor category.
\end{corollary}

\section{Relation to Müger's crossed braided tensor category of localized endomorphisms}\label{Sec: Muger}

In this section, we show that the $\Aut(\A)$-crossed braided category $\Rep^{\Aut(\A)}(\A)$ is equivalent to the category of $\Aut(\A)$-localized endomorphisms in some $I_0\in\Jcal_\mathrm{p}$, denoted $\Aut(\A)-\Loc_{I_0}(\A)$, as defined in \cite{muger05}. Let us recall the definition of $\Aut(\A)-\Loc_{I_0}(\A)$ first. Let $\Jcal_{\mathrm{p}}:=\{I\in \Jcal\ |\ \mathrm{p}\notin \mathrm{cl}(I)\}$ be the collection of intervals of $S^1$ not containing $\mathrm{p} = 1\in S^1$ in their closure. We write 
\[
\A_\infty := \bigcup\limits_{I\in \Jcal_\mathrm{p}}\A(I)\subset B(H_0)
\]
for the $
*$-subalgebra of $B(H_0)$ given by the union of all the $\A(I)$, for $I\in \Jcal_\mathrm{p}$. Note that $\A_\infty$ is an algebraic inductive limit, and that no closure is involved.

\begin{definition}
Let $\End\,\A_\infty$ be the $\C$-linear strict tensor category whose objects are unital $*$-algebra homomorphisms from $\A_\infty$ to itself, with    
\begin{align*}
    \Hom_{\End\,\A_\infty}(\rho,\sigma) & = \{s\in\A_\infty\ |\ s\rho(x) = \sigma(x)s\ \ \forall x\in \A_\infty\}\\
    t\circ s &= ts,\text{      for $s\in \Hom_{\End\,\A_\infty}(\rho,\sigma)$ and $t\in\Hom_{\End\,\A_\infty}(\sigma,\eta)$}\\
    \rho\otimes\sigma &= \rho\circ\sigma\\
    s\otimes t & = s\rho(t) = \rho'(t)s,\text{      for $s\in \Hom_{\End\,\A_\infty}(\rho,\rho')$ and $t\in\Hom_{\End\,\A_\infty}(\sigma,\sigma')$}
\end{align*}
where $\rho,\rho',\sigma,\sigma'\in \End\,\A_\infty$. The unit is given by $\id_{\A_\infty}$.
\end{definition}

 Fix an interval $I_0\in \Jcal_\mathrm{p}$ for the rest of the section, and let $\varphi\in \Aut(\A)$ be an automorphism. We say an object $\rho\in \End\,\A_\infty$ is $\varphi$\emph{-localized} in $I_0$ if
\[
\begin{array}{clc}
    \rho(x) &= x \hspace{1cm} & \forall J\in \Jcal_\mathrm{p}\text{ such that $J\subset(\mathrm{p},\partial_-I_0)$ and all $x\in \A(J)$},\\
    \rho(x) &= \varphi(x) \hspace{1cm} & \forall J\in \Jcal_\mathrm{p}\text{ such that $J\subset(\partial_+I_0,\mathrm{p})$ and all $x\in \A(J)$}.
\end{array}
\]
Here, $\partial_- I_0$, $\partial_+I_0$, $(\mathrm{p},\partial_-I_0)$ and $(\partial_+I_0,\mathrm{p})$ are taken with the usual orientation of $S^1$, meaning in particular that $\partial_+I_0$ is counter-clockwise to $\partial_-I_0$ along $I_0$. By diffeomorphism covariance of $\A$, any $\varphi$-localized endomorphism $\rho\in\End\,\A_\infty$ is \emph{transportable}, meaning that, for every $J\in\Jcal_\mathrm{p}$, there exists another endomorphism $\rho'\in\End\,\A_\infty$ which is $\varphi$-localized in $J$ and satisfies $\rho\cong \rho'$. We write $\varphi-\Loc_{I_0}(\A)$ for the subcategory of $\End\,\A_\infty$ of endomorphisms which are $\varphi$-localized in $I_0$, and whose morphisms are morphisms in $\End\,\A_\infty$ which are in $\A(I_0)$. We say that an endomorphism is $\Aut(\A)$-localized in $I_0$ if it is $\varphi$-localized in $I_0$ for some $\varphi\in \Aut(\A)$.

\begin{definition}\label{def: CategoryOfAut(A)LocEndomorphisms}
        We define the $\mathrm{W}^*$-category $\Aut(\A)-\Loc_{I_0}(\A):= \bigoplus\limits_{\varphi\in \Aut(\A)}\varphi-\Loc_{I_0}(\A)$, whose objects are countable orthogonal direct sums of $\Aut(\A)$-localized-in-$I_0$ endomorphisms of $\A_\infty$.
\end{definition}

  If $G$ is a discrete group acting on $\A$, we say that an endomorphism $\rho\in \End\,\A_\infty$ is $G$-localized in $I_0$ if it is $g$-localized in $I_0$ for some $g\in G$, and we write $G-\Loc_{I_0}(\A)\subset \Aut(\A)-\Loc_{I_0}(\A)$ for the full subcategory on the direct sums of $G$-localized endomorphisms in $I_0$. 

  We let $\Aut(\A)$ act on $\Aut(\A)-\Loc_{I_0}(\A)$ by $\gamma_\varphi(\rho) = \varphi\circ\rho\circ\varphi^{-1}$ and $\gamma_\varphi(s) = \varphi(s)$ for $\varphi\in \Aut(\A)$, $\rho,\sigma\in\Aut(\A)-\Loc_{I_0}(\A)$ and $s\in\Hom(\rho,\sigma)$. The crossed braiding is defined as follows. Fix $\varphi,\mu\in \Aut(\A)$ and let $\rho,\sigma \in \Aut(\A)-\Loc_{I_0}(\A)$ be endomorphisms with $\rho$ $\varphi$-localized in $I\in \Jcal_\mathrm{p}$ and $\sigma$ $\mu$-localized in $J\in \Jcal_\mathrm{p}$. Let $I', J'\in \Jcal_\mathrm{p}$ be intervals such that $(\partial_- I', \partial_+ J')$ is covered by $I_0$. Note that this means that the arc that travels clockwise from $J'$ to $I'$ does not meet $\mathrm{p}$. By transportability, there exist $\rho'$ and $\sigma'$ localized in $I'$ and $J'$ respectively, and unitaries $u\in \Hom(\rho,\rho')$ and $v\in\Hom(\sigma,\sigma')$. By \cite[Lem. 2.14]{muger05}, we have $\rho'\otimes\sigma' = \gamma_\varphi(\sigma')\otimes\rho'$, and hence we can define the braiding of $\rho$ and $\sigma$ as the composite
  \[
  c_{\rho,\sigma}: \rho\otimes\sigma \xrightarrow{u\otimes v} \rho'\otimes \sigma'
 = \gamma_\varphi(\sigma')\otimes\rho'\xrightarrow{\gamma_\varphi(v^*)\otimes u^*} \gamma_\varphi(\sigma)\otimes\rho. \]
  As an element of $\A_\infty$, we have $c_{\rho,\sigma} = \gamma_\varphi(\sigma)(u^*)\gamma_\varphi(v^*)u\rho(v)$. It can be shown that $c_{\rho,\sigma}$ is independent of the choices of $I',J',\rho',\sigma'$ and that it defines an $\Aut(\A)$-crossed braiding on $\Aut(A)-\Loc_{I_0}(\A)$, see \cite[Prop. 2.17]{muger05}.

  Proposition 2.17 in \cite{muger05} provides the following result.

  \begin{theorem}\label{Thm: MugersCatIsGCrossed}
      The $\mathrm{W}^*$-tensor category $\Aut(\A)-\Loc_{I_0}(\A)$ with the action $\gamma$ of $\Aut(\A)$ and the crossed braiding $c$ is an $\Aut(\A)$-crossed braided tensor category. If $G$ is a discrete group acting faithfully on $\A$, then $\gamma$ and $c$ induce on $G-\Loc_{I_0}(\A)$ the structure of a $G$-crossed braided $\mathrm{W}^*$-tensor category.
  \end{theorem}

We will prove that the $\Aut(\A)$-crossed braided tensor category $\Aut(\A)-\Loc_{I_0}(\A)$ is equivalent to $\Rep^{\Aut(\A)}(\A)$. We define a fully faithful essentially surjective functor $\mathfrak{E}: \Aut(\A)-\Loc_{I_0}(\A)\to \Rep^{\Aut(\A)}(\A)$ as follows. We let $\mathfrak{E}(0)$ be the zero representation of $\A$. Given $\rho$ a $\varphi$-localized endomorphism in $I_0$, we define $\mathfrak{E}(\rho) = (H^{\varphi}_{\rho}, \pi^\rho)\in \Rep^\varphi(\A)$ by letting $H^{\varphi}_{\rho}:= H_0$ and, for all $\tilde{I}\in \Jcal_\R$ with $I\in \Jcal_{\mathrm{p}}$ and $x\in \A_\R(\tilde{I})=\A(I)$,
\[
\pi_{\tilde{I}}^\rho(x):=\pi_{0,I}(\rho \circ \varphi^{\epsilon(\tilde{I})} x).
\]
Since $\rho$ is $\varphi$-localized in $I_0$, the formula above extends to a $*$-action of all $\A_\R(\tilde{I})$, even if $\mathrm{p}\in \mathrm{cl}(I)$. We will abuse the notation and continue writing $\pi_{\tilde{I}}^\rho(x):=\pi_{0,I}(\rho \circ \varphi^{\epsilon(\tilde{I})} x)$ even if $\mathrm{p}\in \mathrm{cl}(I)$. It is clear that
\[
\pi_{\tilde{I}+1}^\rho = \pi_{\tilde{I}}^\rho\circ \varphi,
\]
hence producing a $\varphi$-twisted representation of $\A$. Given $s\in\Hom_{\Aut(\A)-\Loc_{I_0}(\A)}(\rho,\sigma)\subset \A(I_0)$ a morphism between $\varphi$-localized endomorphisms in $I_0$, we can regard it as a morphism between representations by letting it act on $H_0$ via $\pi_{0,I_0}$. Hence, we let $\mathfrak{E}(s) = \pi_{0,I_0}(s)\in \Hom_{\Rep^\varphi(\A)}(H^\varphi_\rho, H^\varphi_\sigma)$. By Haag duality and the fact that $\rho,\sigma$ are $\varphi$-localized in $I_0$, any element of $\Hom_{\Rep^\varphi(\A)}(H^\varphi_\rho, H^\varphi_\sigma)$ arises this way. We obtain an identification $\Hom(\rho,\sigma)\cong\Hom_{\Rep^\varphi(\A)}(H^\varphi_\rho, H^\varphi_\sigma)$ and the functor $\mathfrak{E}: \Aut(\A)-\Loc_{I_0}(\A)\to\Rep^{\Aut(\A)}(\A)$ is a fully faithful $\mathrm{W}^*$-functor of $\Aut(\A)$-graded $\mathrm{W}^*$-categories. 

\begin{proposition}\label{prop: mathfrakEEquivalenceofCats}
    The functor $\mathfrak{E}: \Aut(\A)-\Loc_{I_0}(\A)\to\Rep^{\Aut(\A)}(\A)$ is an equivalence of $\mathrm{W}^*$-categories.
\end{proposition}
\begin{proof}
    It is only left to show essential surjectivity. To do so, let $(H^\varphi,\pi^H)\in\Rep^\varphi(\A)$ and recall that we write $\hat{I_0}\in \Jcal_\R$ for the lift of $I_0$ to $\R_{> 0}$ such that $\mathrm{cl}(\hat{I_0})\subset \R_{>0}$ is as close to zero as possible. Then, $(H^\varphi,\pi^H_{\hat{I_0}^{c-}})$ is a representation of the type $\mathrm{III}_1$ factor $\A_\R(\hat{I_0}^{c-}) = \A(I_0^c)$. Hence, there exists a unitary $u: H_0\xrightarrow{\cong }H^\varphi$ intertwining the actions $\pi_{0,I_0^c}$ and $\pi_{\hat{I_0}^{c-}}^H$. We define, for every $I\in \Jcal_{\mathrm{p}}$,
    \[
    \rho_{I}(-) = u^*\pi^H_{\hat{I}}(-)u.
    \]
    By Haag duality, it holds that $\rho_{I_0}(\A(I_0))\subset \A(I_0)$. If $J\in\Jcal_{\mathrm{p}}$ is an interval such that $J\subset (\mathrm{p}, \partial_-I_0)$, then $\hat{J}\subset \hat{I_0}^{c-}$ and hence $\rho_{J} = \id_{\A(J)}$. If $J\subset (\partial_+I_0, \mathrm{p})$, then $\hat{J}-1\subset \hat{I_0}^{c-}$ and $\rho_J = \varphi$. Hence, $\rho$ is a $\varphi$-localized-in-$I_0$ endomorphism of $\A_\infty$, and $u$ provides an equivalence between $\mathfrak{E}(\rho)$ and $(H^\varphi,\pi^H)$.
\end{proof}

We next upgrade the functor $\mathfrak{E}$ to a monoidal equivalence. Let $\tilde{I_0}$ be a lift of $I_0$ to $\R$. Let $\varphi,\mu\in\Aut(\A)$ and $\rho$ and $\sigma$ be $\varphi$- and $\mu$-localized objects of $\Aut(\A)-\Loc_{I_0}(\A)$, respectively. Then, $\mathfrak{E}(\rho)\boxtimes \mathfrak{E}(\sigma) = H^\varphi_{\rho}\boxtimes H^\mu_\sigma$ and $\mathfrak{E}(\rho\circ\sigma) = H^{\varphi\mu}_{\rho\sigma}$ are not equal on the nose. However, there is a unitary isomorphism of twisted $\A$-representations given as follows. Note that the subspace of $H^\varphi_{\rho}\boxtimes H^\mu_\sigma$ given by vectors of the form $L(\pi_0(x)\Omega,\tilde{I_0})\pi_0(y)\Omega$, for $x,y\in \A(I_0)$, is dense in $H^\varphi_{\rho}\boxtimes H^\mu_\sigma$. We define
\[
\begin{array}{cccc}
    \Psi_{\rho,\sigma}: & H^\varphi_{\rho}\boxtimes H^\mu_\sigma &\to & H^{\varphi\mu}_{\rho\sigma}  \\
     & L(\pi_0(x)\Omega,\tilde{I_0})\pi_0(y)\Omega&\mapsto &\pi_0(\rho(y)x)\Omega
\end{array}
\]
which is a unitary compatible with the actions of $\A$, see \cite[Sec. V.B.$\delta$]{Connes}.

\begin{proposition}
    The pair $(\mathfrak{E},\Psi): \Aut(\A)-\Loc_{I_0}(\A)\to \Rep^{\Aut(\A)}(\A)$ is an equivalence of $\mathrm{W}^*$-tensor categories.
\end{proposition}
\begin{proof}
Given a third automorphism $\nu\in \Aut(\A)$ and $\tau$ a $\nu$-localized in $I_0$ object of $\Aut(\A)-\Loc_{I_0}(\A)$, the associator
\[
\big( H^\varphi_\rho\boxtimes H^\mu_\sigma\big)\boxtimes H^\nu_\tau \cong H^\varphi_\rho\boxtimes\big(H^\mu_\sigma\boxtimes H^\nu_\tau\big)
\]
reads $L\big(L(\pi_0(x)\Omega,\tilde{I_0})\pi_0(y)\Omega,\tilde{I_0}\big)\pi_0(z)\Omega\mapsto L(\pi_0(x)\Omega,\tilde{I_0})L(\pi_0(y)\Omega,\tilde{I_0})\pi_0(z)\Omega$ for all $x,y,z\in \A(I_0)$, using Proposition \ref{prop: JoinLs}. Hence, compatibility of $\Psi$ with the associator is equivalent to the equality of the outcomes of the following two compositions of unitaries, for all $x,y,z\in \A(I_0)$. On the one side, we have
\[
L\big(L(\pi_0(x)\Omega,\tilde{I_0})\pi_0(y)\Omega,\tilde{I_0}\big)\pi_0(z)\Omega\xmapsto{\Psi_{\rho,\sigma}} L\big(\pi_0(\rho (y)x)\Omega,\tilde{I_0}\big)\pi_0(z)\Omega\xmapsto{\Psi_{\rho\sigma,\tau}}\pi_0(\rho\circ\sigma(z)\rho(y)x)\Omega,
\]
and on the other side,
\[
L(\pi_0(x)\Omega,\tilde{I_0})L(\pi_0(y)\Omega,\tilde{I_0})\pi_0(z)\Omega\xmapsto{\Psi_{\sigma,\tau}}L(\pi_0(x)\Omega,\tilde{I_0})\pi_0(\sigma(z)y)\Omega\xmapsto{\Psi_{\rho,\sigma\tau}}\pi_0(\rho(\sigma(z)y)x)\Omega.
\]
The compositions indeed agree. To argue naturality of $\Psi$, let $\rho_1,\rho_2\in \Aut(\A)-\Loc_{I_0}(\A)$ be $\varphi$-localized, $\sigma_1,\sigma_2\in \Aut(\A)-\Loc_{I_0}(\A)$ be $\mu$-localized and $s\in \Hom(\rho_1,\rho_2)$, $t\in \Hom(\sigma_1,\sigma_2).$ Then, we need, for all $x,y\in \A(I_0),$ that the following two compositions agree:
\[
L(\pi_0(x)\Omega,\tilde{I_0})\pi_0(y)\Omega\xmapsto{\pi_0(s)\boxtimes\pi_0(t)}L(\pi_0(sx)\Omega,\tilde{I_0})\pi_0(ty)\Omega\xmapsto{\Psi_{\rho_2,\sigma_2}}\pi_0(\rho_2(ty)sx)\Omega
\]
and
\[
L(\pi_0(x)\Omega,\tilde{I_0})\pi_0(y)\Omega\xmapsto{\Psi_{\rho_1,\sigma_1}}\pi_0(\rho_1(y)x)\Omega\xmapsto{\pi_0(\rho_2(t)s)}\pi_0(\rho_2(t)s\rho_1(y)x)\Omega = \pi_0(\rho_2(t)\rho_2(y)sx)\Omega.
\]
Since both compositions indeed agree, we have proved that $(\mathfrak{E},\Psi)$ is a monoidal equivalence.    
\end{proof}

Let us argue the compatibility of $\mathfrak{E}$ with the action of $\Aut(\A)$. Let $\rho\in \Aut(\A)-\Loc_{I_0}(\A)$ be a $\varphi$-localized endomorphism of $\A_\infty$, and $\nu\in \Aut(\A)$ be an automorphism. Then, $T_\nu(H^\varphi_\rho)$ is the twisted $\A$-representation on $H_0$ given by, for $\tilde{I}\in \Jcal_\R$ and $x\in\A_\R(\tilde{I})$,
\[
\pi^{\nu\ast H_\rho}_{\tilde{I}}(x) = \pi_{0,I}(\rho\circ \varphi^{\epsilon(\tilde{I})}\circ \nu^{-1}x).
\]
On the other hand, $H^{\nu\varphi\nu^{-1}}_{\gamma_\nu(\rho)}:=\mathfrak{E}(\gamma_\nu(\rho))$ is the $\A$-representation on $H_0$ given by, for $\tilde{I}\in \Jcal_\R$ and $x\in\A_\R(\tilde{I})$,
\[
\pi^{\gamma_\nu(\rho)}_{\tilde{I}}(x) = \pi_{0,I}\big(\nu\circ\rho\circ\nu^{-1}\circ (\nu\circ \varphi\circ \nu^{-1})^{\epsilon(\tilde{I})}(x)\big) = \pi_{0,I}(\nu\circ \rho\circ \varphi^{\epsilon(\tilde{I})}\circ\nu^{-1}x).
\]
Hence, the unitary $V_{\nu}: H_0\to H_0$ provides an equivalence $T_\nu(\mathfrak{E}(\rho))\cong \mathfrak{E}(\gamma_\nu(\rho))$ which is clearly natural in $\rho$. Compatibility with $\Psi$ amounts to showing that the following two compositions agree, for all $x,y\in\A(I_0)$. On the one hand, we have
\[
L(\pi_0(x)\Omega,\tilde{I_0})\pi_0(y)\Omega\xmapsto{T_\nu(\Psi_{\rho,\sigma})}\pi_{0}(\rho(y)x)\Omega\xmapsto{V_\nu} V_\nu\circ \pi_{0}(\rho(y)x)\Omega = \pi_0\big(\nu(\rho(y)x)\big)\Omega,
\]
and on the other hand
\begin{align*}
L(\pi_0(x)\Omega,\tilde{I_0})\pi_0(y)\Omega\xmapsto{V_\nu\boxtimes V_\nu} L(V_\nu\circ \pi_0(x)&\Omega,\tilde{I_0})V_\nu\circ \pi_0(y)\Omega = L(\pi_0(\nu x)\Omega,\tilde{I_0})\pi_0(\nu y)\Omega\\ &\xmapsto{\Psi_{\gamma_\nu(\rho),\gamma_\nu(\sigma)}} \pi_0\big(\nu\circ\rho\circ\nu^{-1}\circ \nu(y) \nu(x)\big)\Omega.
\end{align*}
Since both outcomes agree, we obtain that $(\mathfrak{E},\Psi,V)$ is an equivalence of $\mathrm{W}^*$-tensor categories with an $\Aut(\A)$-action. It is only left to argue that this equivalence is compatible with the $\Aut(\A)$-crossed braiding. 

\begin{proposition}\label{prop: mathfrakEcompatiblebraiding}
    Fix $\varphi,\mu\in \Aut(\A)$, and let $\rho\in \Aut(\A)-\Loc_{I_0}(\A)$ be a $\varphi$-localized endomorphism and $\sigma\in \Aut(\A)-\Loc_{I_0}(\A)$ be a $\mu$-localized endomorphism. Then, the following diagram commutes,
\[\begin{tikzcd}
	{H^\varphi_\rho\boxtimes H^\mu_\sigma} &&&& {H^{\varphi\mu}_{\rho\sigma}} \\
	{T_\varphi(H^\mu_\sigma)\boxtimes H^\varphi_\rho} &&&& {H_{\gamma_\varphi(\sigma)\rho}^{\varphi\mu\varphi^{-1}\varphi}} \\
	&& {H^{\varphi\mu\varphi^{-1}}_{\gamma_\varphi(\sigma)}\boxtimes H^\varphi_\rho}
	\arrow["{\Psi_{\rho, \sigma}}", from=1-1, to=1-5]
	\arrow["{\mathbb{B}_{H^\varphi_\rho, H^\mu_\sigma}}"', from=1-1, to=2-1]
	\arrow["{\mathfrak{E}(c_{\rho, \sigma})}", from=1-5, to=2-5]
	\arrow["{V_\varphi\boxtimes \id}"', from=2-1, to=3-3]
	\arrow["{\Psi_{\gamma_\varphi(\sigma), \rho}}"', from=3-3, to=2-5]
\end{tikzcd}\]
\end{proposition}
\begin{proof}
    Let $I_1,I_2\in \Jcal$ be disjoint intervals contained in $I_0$ and such that $(\partial_-I_1, \partial_+I_2)$ is covered by $I_0$. By transportability, $\rho$ and $\sigma$ are unitarily isomorphic in $\Aut(\A)-\Loc_{I_0}(\A)$ to endomorphisms $\rho'$ and $\sigma'$ localized in $I_1$ and $I_2$ respectively. Since $(\mathfrak{E},\Psi)$ is monoidal, it suffices to show the commutativity of the diagram above when $\rho$ is localized in $I_1$ and $\sigma$ is localized in $I_2$. Then, $c_{\rho,\sigma} = \id$. 
    
    Let $\tilde{I_1},\tilde{I_2}$ be the lifts of $I_1,I_2$ to $\R$ lying inside of $\tilde{I_0}.$ We let $\tilde{J_1},\tilde{J_2}\in \Jcal_\R$ subintervals of $\tilde{I_0}$ such that $\tilde{J_2}\subset \tilde{I_1}^{c+}$ and $\tilde{J_1}\subset \tilde{J_2}^{c+}\cap \tilde{I_2}^{c-}$. We consider the dense subset of $H^\varphi_\rho\boxtimes H^\mu_\sigma$ spanned by vectors of the form $L(\pi_0(x)\Omega,\tilde{I_0})\pi_0(y)\Omega$ for $x\in\A_\R(\tilde{J_1})$ and $y\in\A_\R(\tilde{J_2}).$ Applied to such a vector $L(\pi_0(x)\Omega,\tilde{I_0})\pi_0(y)\Omega$, the top-right leg of the diagram reads
    \[
L(\pi_0(x)\Omega,\tilde{I_0})\pi_0(y)\Omega\mapsto \pi_0(\rho(y)x)\Omega = \pi_0(\varphi(y)x)\Omega.
    \]
    Under the bottom-left leg of the diagram, the same vector gets mapped to
    \begin{align*}
L(\pi_0(x)\Omega,\tilde{I_0})\pi_0(y)\Omega&\xmapsto{\mathbb{B}_{H^\varphi_\rho, H^\mu_\sigma}} \mathbb{B}_{H^\varphi_\rho, H^\mu_\sigma} L(\pi_0(x)\Omega,\tilde{I_0})\pi_0(y)\Omega = L(\Gamma_\varphi\pi_0(y)\Omega,\tilde{I_0})\pi_0(x)\Omega\\ &\xmapsto{V_\varphi\boxtimes \id} L(V_\varphi\pi_0(y)\Omega,\tilde{I_0})\pi_0(x)\Omega  = L(\pi_0(\varphi y)\Omega,\tilde{I_0})\pi_0(x)\Omega\\ &\xmapsto{\Psi_{\gamma_\varphi(\sigma), \rho}} \pi_0(\varphi\sigma\varphi^{-1}(x)\varphi(y))\Omega = \pi_0(x\varphi(y))\Omega = \pi_0(\varphi(y)x)\Omega.
    \end{align*}
The claim follows.
\end{proof}

We have argued the following result.

\begin{theorem}\label{Thm: RepIsMugerForG}
    Let $I_0\in\Jcal_\mathrm{p}$ be an interval. Then, the triple $(\mathfrak{E}, \Psi, V)$ provides an equivalence of $\Aut(\A)$-crossed braided $\mathrm{W}^*$-tensor categories between $\Aut(\A)-\Loc_{I_0}(\A)$ and $\Rep^{\Aut(\A)}(\A)$. If $G$ is a discrete group acting faithfully on $\A$, the triple $(\mathfrak{E}, \Psi, V)$ induces an equivalence of $G$-crossed braided $\mathrm{W}^*$-tensor categories between $G-\Loc_{I_0}(\A)$ and $\Rep^G(\A)$.
\end{theorem}
\begin{proof}
   The first claim follows from Propositions \ref{prop: mathfrakEEquivalenceofCats}-\ref{prop: mathfrakEcompatiblebraiding}. The second claim follows analogously restricting all the relevant functors along the homomorphism $G\to \Aut(\A)$ describing the action of $G$ on $\A$.
\end{proof}

In the case when the conformal net $\A$ is rational, the category $\Loc_{I_0}(\A):= \{e\}-\Loc_{I_0}(\A)$ is a unitary fusion category. In that case, every object $\rho\in \Loc_{I_0}(\A)$ comes with a dual localized endomorphism $\bar{\rho}\in \Loc_{I_0}(\A)$ and evaluation and coevaluation morphisms
\[
\text{ev}_{\rho}: \bar{\rho}\otimes \rho \to \id_{\A_\infty}\hspace{2cm} \text{coev}_{\rho}: \id_{\A_\infty}\to \rho\otimes \bar{\rho}
\]
satisfying the snake identities. By unitarity, $\Loc_{I_0}(\A)$ also comes equipped with a preferred pivotal structure, that is a family of compatible isomorphisms $\pi_\rho: \rho\to \bar{\bar{\rho}}$ indexed by $\rho\in \Loc_{I_0}(\A)$. Then, one can define a balance $\theta'$ on $\Loc_{I_0}(\A)$ by the usual picture of a kink, meaning
\[
\theta'_\rho: \rho\cong \rho\otimes\id_{\A_\infty}\xrightarrow{\id\otimes\text{coev}_{\rho}}\rho\otimes\rho\otimes\bar{\rho}\xrightarrow{c_{\rho,\rho}\otimes \id} \rho\otimes\rho\otimes\bar{\rho} \xrightarrow{\id\otimes\pi_{\rho}\otimes \id}\rho\otimes\bar{\bar{\rho}}\otimes\bar{\rho} \xrightarrow{\id\otimes \text{ev}_{\bar{\rho}}} \rho\otimes \id_{\A_\infty}\cong \rho.
\]
It is therefore natural to ask if, in this case, the equivalence in Theorem \ref{Thm: RepIsMugerForG} is compatible with the balances on each side. This is essentially the conformal Spin and Statistics Theorem in \cite{MR1410566}. There, the braiding used is the reverse of the braiding we have defined in Section  \ref{Sec: Muger}, as is clear from the way endomorphisms are localized in the definition of the statistics operator \cite[p. 8]{MR1410566}. Hence, transporting their result to our setting requires us to also take the reverse balance, that is, its inverse.

\begin{theorem}
    Let $\A$ be a rational conformal net, and denote by $\theta'$ the canonical balance on the unitary fusion category $\Loc_{I_0}(\A)$. Then, the equivalence of braided $\mathrm{W}^*$-tensor categories $\mathfrak{E}: \Loc_{I_0}(\A)\cong \Rep(\A)$ in Theorem \ref{Thm: RepIsMugerForG} extends to an equivalence of balanced $\mathrm{W}^*$-tensor categories.
\end{theorem}
\begin{proof}
    By the same arguments as \cite[Thm. 3.13]{MR1410566} for our conventions, the balance $\theta'$ on an endomorphism $\rho\in \Loc_{I_0}(\A)$ is given by the action of $e^{-2\pi i L_0}$ on the representation $\mathfrak{E}(\rho)$ of $\A$ induced by $\rho$. The claim follows.
\end{proof}

\appendix
\section{Bi-involutive structures}

In this Appendix, we construct an involutive $\mathrm{W}^*$-structure on the $\mathrm{W}^*$-tensor category $\Rep^{\Aut(\A)}(\A)$, following \cite{bicommutantfromnets}. We note that, in \cite{henriques2024completewcategories}, what we refer to as involutive $\mathrm{W}^*$-tensor categories are called bi-involutive $\mathrm{W}^*$-tensor categories.

Informally, an involutive structure on a $\mathrm{W}^*$-tensor category $\mathcal{C}$ consists of an anti-linear, anti-tensor functor $\overline{\,\cdot\,}:\mathcal{C}\to \mathcal{C}$ that squares to the identity. Given a discrete group $G$, in Definition \ref{def: G-X-braidedInvolutive} we upgrade the notion of a $G$-crossed braiding on a $\mathrm{W}^*$-tensor category $\Ccal$ from Definition \ref{def: G-X-braidedWTensorCat} to the situation where $\mathcal{C}$ is further equipped with an involutive structure. This provides the notion of a $G$-crossed braided involutive $\mathrm{W}^*$-tensor category. We also introduce, in this setting, the notion of a $G$-crossed balance, which is required to be compatible with the involutive structure in the sense that a certain ribbon-like condition is satisfied, see Definition \ref{def: Involutive-X-Balance} and Remark \ref{rk: Ribbon}. A $G$-crossed balanced involutive $\mathrm{W}^*$-tensor category is a $G$-crossed braided involutive $\mathrm{W}^*$-tensor category with a $G$-crossed balance.

The involution functor on $\Rep^{\Aut(\A)}(\A)$ is introduced in Definition \ref{def: involutionRepA}, and the necessary structure data are presented in Lemmas \ref{lemm: InvolutiveDataUnit} and \ref{lemm: InvolutiveDataSquare}, and in Proposition \ref{prop: 2CellInvolution}. In Theorem \ref{thm: Involutive-AutX-X-Balanced} we show that this data indeed defines a $G$-crossed balanced involutive $\mathrm{W}^*$-structure on $\Rep^{\Aut(\A)}(\A)$. The proof of the ribbon-like condition between the balance and the involution on $\Rep^{\Aut(\A)}(\A)$ comes from a good understanding of the conformal structure of the image of a representation $H^\varphi\in \Rep^\varphi(\A)$ under $\overline{\,\cdot\,}$ in terms of the conformal structure of $H^\varphi$. This is done in Proposition \ref{prop: ConformalActionOnConjugate}. 

Let us first recall the notion of a bi-involutive tensor category. Let $\mathcal{C}$ be a tensor category with unit $\1\in \mathcal{C}$. An \emph{involutive tensor structure} on $\mathcal{C}$ is an anti-linear functor $\overline{\,\cdot\,}:\mathcal{C}\to \mathcal{C}$ which is anti-tensor in the sense that it comes equipped with an isomorphism $r: \1\to \overline{\1}$ and natural isomorphisms
\[
\nu_{X,Y}: \overline{X}\otimes \overline{Y}\to \overline{Y\otimes X}
\]
satisfying the obvious coherence between the associator of $\mathcal{C}$ and its image under $\overline{\,\cdot\,}$ \cite{MR2861112}. We further require the data of a trivialization of the square of $\overline{\,\cdot\,}$, meaning a natural family of isomorphisms $\varphi_X: X\to \overline{\overline{X}}$ such that $\overline{\varphi_X} = \varphi_{\overline{X}}$, $\varphi_{\1} = \overline{r}\circ r$ and $\varphi_{X\otimes Y} = \overline{\nu_{X\otimes Y}}\circ \nu_{\overline{X},\overline{Y}}\circ (\varphi_X\otimes\varphi_Y)$. We call $\overline{X}$ the conjugate of $X$.

\begin{definition}(\cite[Def. 7.1]{henriques2024completewcategories})\label{def: involutiveW*Cat}
    An \emph{involutive $\mathrm{W}^*$-tensor category} is a $\mathrm{W}^*$-tensor category equipped with an involutive structure such that the functor $\overline{\,\cdot\,}$ is a functor of $\mathrm{W}^*$-categories and the structure morphisms $\nu, r, \varphi$ are unitary.
\end{definition}

An \emph{involutive $\mathrm{W}^*$-tensor functor} between involutive $\mathrm{W}^*$-tensor categories $\mathcal{C}$ and $\mathcal{D}$ is a tensor functor $(F,\Psi)$ of $\mathrm{W}^*$-tensor categories with a natural family of unitary isomorphisms $\chi_X:\overline{F(X)} \to   F(\overline{X})$ such that the following diagrams commute
\[\begin{tikzcd}
	{F(X)} && {F(\overline{\overline{X}})} && {\1_{\mathcal{D}}} & {\overline{\1_{\mathcal{D}}}} & {\overline{F(\1_{\mathcal{}C}})} \\
	{\overline{\overline{F(X)}}} && {\overline{F(\overline{X})}} && {F(\1_{\Ccal})} && {F(\overline{\1_{\Ccal}})}
	\arrow["{F(\varphi_X)}", from=1-1, to=1-3]
	\arrow["{\varphi_{F(X)}}"', from=1-1, to=2-1]
	\arrow["r", from=1-5, to=1-6]
	\arrow["\cong"', from=1-5, to=2-5]
	\arrow["\cong", from=1-6, to=1-7]
	\arrow["{\chi_{1_{\Ccal}}}", from=1-7, to=2-7]
	\arrow["{\overline{\chi_X}}"', from=2-1, to=2-3]
	\arrow["{\chi_{\overline{X}}}"', from=2-3, to=1-3]
	\arrow["{F(r)}"', from=2-5, to=2-7]
\end{tikzcd}\]
\[\begin{tikzcd}
	{\overline{F(X)}\otimes \overline{F(Y)}} && {F(\overline{X})\otimes F(\overline{Y})} \\
	{\overline{F(Y)\otimes F(X)}} && {F(\overline{X}\otimes\overline{Y})} \\
	{\overline{F(Y\otimes X)}} && {F(\overline{Y\otimes X})}
	\arrow["{\chi_{X}\otimes \chi_Y}", from=1-1, to=1-3]
	\arrow["{\nu_{F(X), F(Y)}}"', from=1-1, to=2-1]
	\arrow["{\Psi_{\overline{X}, \overline{Y}}}", from=1-3, to=2-3]
	\arrow["{\overline{\Psi_{Y, X}}}"', from=2-1, to=3-1]
	\arrow["{F(\nu_{X,Y})}", from=2-3, to=3-3]
	\arrow["{\chi_{Y\otimes X}}"', from=3-1, to=3-3]
\end{tikzcd}\]
for all $X,Y\in \Ccal$. A natural transformation between two such involutive $\mathrm{W}^*$-tensor functors $F$ and $G$ consists of a natural transformation $\eta:F\to G$ such that the following diagram commutes for all $X\in \mathcal{C}$
\begin{equation}\label{eq: InvolutiveNatTrans}\begin{tikzcd}
	{\overline{F(X)}} && {\overline{G(X)}} \\
	{F(\overline{X})} && {G(\overline{X})}.
	\arrow["{\overline{\eta_X}}", from=1-1, to=1-3]
	\arrow["{\chi^F_{X}}"', from=1-1, to=2-1]
	\arrow["{\chi^G_{X}}", from=1-3, to=2-3]
	\arrow["{\eta_{\overline{X}}}", from=2-1, to=2-3]
\end{tikzcd}\end{equation}
Let us denote by $\underline{\text{Aut}}_\otimes^{\overline{\,\cdot\,}}(\mathcal{C})$ the tensor category of involutive $\mathrm{W}^*$-tensor automorphisms of $\mathcal{C}$ and unitary involutive $\mathrm{W}^*$-natural transformations. Given a discrete group $G$, we can upgrade the notion of a $G$-crossed braided $\mathrm{W}^*$-tensor category to that of a $G$-crossed braided involutive $\mathrm{W}^*$-tensor category

\begin{definition}\label{def: G-X-braidedInvolutive}
    Let $G$ be a discrete group and $\mathcal{C}$ be an involutive $\mathrm{W}^*$-tensor category. A \emph{$G$-crossed braided structure} on $\mathcal{C}$ consists of a $G$-crossed braided structure on the underlying $\mathrm{W}^*$-category $\mathcal{C}$ with the extra data of a collection of unitary isomorphisms
    \[
    \chi_g(X): \overline{T_g(X)}\to T_g(\overline{X})
    \]
    for every $g\in G$ and $X\in \mathcal{C}$, natural in $X$. We require that
\begin{enumerate}
    \item the functor $\overline{\,\cdot\,}:\mathcal{C}\to \mathcal{C}$ restricts to equivalences $\mathcal{C}_g\to\mathcal{C}_{{g}^{-1}}$ for all $g\in G$,
    \item for every $g$, the family $\chi_g$ endows $T_g$ with the structure of an involutive $\mathrm{W}^*$-tensor functor,
    \item for every $g\in G$, $X\in \mathcal{C}_g$, and $Y\in \mathcal{C}$, the following diagram commutes,
\begin{equation}\label{eq: CompatibilityBraidingInvolution}\begin{tikzcd}
	{\overline{X\otimes Y}} && {\overline{T_g(Y)\otimes X}} & {\overline{X}\otimes \overline{T_g(Y)}} \\
	{\overline{Y}\otimes \overline{X}} & {T_e(\overline{Y})\otimes \overline{X}} & {T_{g^{-1}}\circ T_g(\overline{Y})\otimes \overline{X}} & {\overline{X}\otimes T_g(\overline{Y})}
	\arrow["{\overline{\beta_{X, Y}}}", from=1-1, to=1-3]
	\arrow["{\nu_{X, Y}}"', from=1-1, to=2-1]
	\arrow["{\nu_{T_g(Y), X}}", from=1-3, to=1-4]
	\arrow["{\id\otimes \chi_g(Y)}", from=1-4, to=2-4]
	\arrow["\cong"', from=2-1, to=2-2]
	\arrow["\cong"', from=2-2, to=2-3]
	\arrow["{\beta_{\overline{X}, T_g(\overline{Y})}}", from=2-4, to=2-3]
\end{tikzcd}\end{equation}
    
    \item for every $g,h\in G$ and $X\in \mathcal{C}$, the following diagram commutes,
\[\begin{tikzcd}
	{\overline{T_gT_h(X)}} && {\overline{T_{gh}(X)}} \\
	{T_g(\overline{T_h(X)})} && {T_{gh}(\overline{X})} \\
	& {T_gT_h(\overline{X})}.
	\arrow["{\overline{n_{g,h}(X)}}", from=1-1, to=1-3]
	\arrow["{\chi_g(T_h(X))}"', from=1-1, to=2-1]
	\arrow["{\chi_{gh}(X)}", from=1-3, to=2-3]
	\arrow["{T_g(\chi_h(X))}"', from=2-1, to=3-2]
	\arrow["{n_{g,h}(\overline{X})}"', from=3-2, to=2-3]
\end{tikzcd}\]
\end{enumerate}
We say that an involutive $\mathrm{W}^*$-tensor category with a $G$-crossed braided structure is a \emph{$G$-crossed braided involutive $\mathrm{W}^*$-tensor category}.
\end{definition}

Given a $G$-crossed braided involutive $\mathrm{W}^*$-tensor category, we can define the notion of a $G$-crossed balance compatible with the involutive structure.

\begin{definition}\label{def: Involutive-X-Balance}
    Let $G$ be a discrete group and $\mathcal{C}$ a $G$-crossed braided involutive $\mathrm{W}^*$-tensor category in the sense of Definition \ref{def: G-X-braidedInvolutive}. A \emph{$G$-crossed balance} on $\mathcal{C}$ consists of a $G$-crossed balance $\theta$ on the underlying $G$-crossed braided $\mathrm{W}^*$-tensor category of $\mathcal{C}$ such that for all $g\in G$ and $X\in\mathcal{C}_g$, the following diagram commutes
\begin{equation}\begin{tikzcd}\label{eq: RibbonCondition}
	{\overline{T_g(X)}} && {\overline{X}} \\
	&& {T_e(\overline{X})} \\
	{T_g(\overline{X})} && {T_{g^{-1}}\circ T_g(\overline{X})}.
	\arrow["{\overline{\theta_X^*}}", from=1-1, to=1-3]
	\arrow["{\chi_g(X)}"', from=1-1, to=3-1]
	\arrow["\cong"', from=2-3, to=1-3]
	\arrow["{\theta_{T_g(\overline{X})}}"', from=3-1, to=3-3]
	\arrow["{n_{g^{-1}, g}(\overline{X})}"', from=3-3, to=2-3]
\end{tikzcd}\end{equation}
    A $G$-crossed braided involutive $\mathrm{W}^*$-tensor category equipped with a $G$-crossed balance is a \emph{$G$-crossed balanced involutive $\mathrm{W}^*$-tensor category}.
\end{definition}

\begin{remark}\label{rk: Ribbon}
    Let $\mathcal{C}$ be a rigid semisimple monoidal category with a $G$-crossed braided structure on its underlying monoidal category. For an object $X\in \mathcal{C}$, we denote by $X^\vee\in\Ccal$ its right dual, and write $\text{ev}_X: X\otimes X^\vee\to\1$ and $\text{coev}_{X}: \1\to X^\vee\otimes X$ for its evaluation and coevaluation maps. Given a morphism $f: X\to Y$ we write $f^t:=(\id_{X^\vee}\otimes \text{ev}_Y)\circ (\id_{X^\vee}\otimes f\otimes\id_{Y^\vee})\circ (\text{coev}_{X}\otimes \id_{Y^\vee})\in \Hom_{\Ccal}(Y^\vee, X^\vee)$. A $G$-crossed ribbon $\theta$ on $\mathcal{C}$ consists of a $G$-crossed balance which satisfies the ribbon condition
\begin{equation}\begin{tikzcd}\label{eq: RigidRibbon}
	{T_g(X)^\vee} && {X^\vee} \\
	&& {T_e(X^\vee)} \\
	{T_g(X^\vee)} && {T_{g^{-1}}\circ T_g(X^\vee)}
	\arrow["{(\theta_X)^t}", from=1-1, to=1-3]
	\arrow["\cong"', from=1-1, to=3-1]
	\arrow["\cong"', from=2-3, to=1-3]
	\arrow["{\theta_{T_g(X^\vee)}}"', from=3-1, to=3-3]
	\arrow["{n_{g,g^{-1}}(X^\vee)}"', from=3-3, to=2-3]
\end{tikzcd}\end{equation}
for all $g\in G$ and $X\in \mathcal{C}_g$, see \cite[Def. 2.3 and Eq. (2.4.a)]{MR2674592}. If $\mathcal{C}$ is further a $\mathrm{W}^*$-tensor category, $\mathcal{C}$ has a canonical structure of an involutive $\mathrm{W}^*$-tensor category, with $\overline{X} = X^\vee$, see \cite[Sec. 5.3]{MR4581741}. In this context, given a morphism $f:X\to Y$, we have $f^t = \overline{f^*}$. Hence, the condition \eqref{eq: RibbonCondition} is the generalization of the crossed-ribbon condition \eqref{eq: RigidRibbon} from the rigid semisimple setting to the context of involutive $\mathrm{W}^*$-tensor categories.
\end{remark}

Given two $G$-crossed braided involutive $\mathrm{W}^*$-tensor categories, a functor between them is a functor $(F,\Psi, \Phi_g)$ between the underlying $G$-crossed braided $\mathrm{W}^*$-tensor categories equipped with a natural isomorphism $\chi$ making $(F,\Psi, \chi)$ an involutive $\mathrm{W}^*$-tensor functor such that the diagram
\[\begin{tikzcd}
	{\overline{T_g(F(X))}} && {\overline{F(T_g(X))}} \\
	{T_g(\overline{F(X)})} && {F(\overline{T_g(X)})} \\
	{T_g(F(\overline{X}))} && {F(T_g(\overline{X}))}
	\arrow["{\overline{\Phi_g(X)}}", from=1-1, to=1-3]
	\arrow["{\chi_g^\mathcal{D}(F(X))}"', from=1-1, to=2-1]
	\arrow["{\chi^{\phantom{D}}_{T_g(X)}}", from=1-3, to=2-3]
	\arrow["{T_g(\chi^{\phantom{D}}_X)}"', from=2-1, to=3-1]
	\arrow["{F(\chi_g^\mathcal{C}(X))}", from=2-3, to=3-3]
	\arrow["{\Phi_g(\overline{X})}"', from=3-1, to=3-3]
\end{tikzcd}\]
commutes for all $g\in G$ and $X\in \mathcal{C}$.

Given involutive $G$-crossed balances on both categories, a functor of $G$-crossed braided involutive $\mathrm{W}^*$-tensor categories is a functor of $G$-crossed balanced involutive $\mathrm{W}^*$-tensor categories if it is a functor between the underlying $G$-crossed balanced $\mathrm{W}^*$-tensor categories.

Recall that we fix the point $\mathrm{p}=1\in S^1$, and that we define $\widetilde{S^1_+} := (0,\nicefrac{1}{2})\in \Jcal_\R$, and $S^1_+:=q(\widetilde{S^1_+})$, the upper semi-circle. The von Neumann algebra $\A(S^1_+)$ acts on $H_0$ by definition, and the vacuum vector $\Omega\in H_0$ is cyclic for the action of $\A(S^1_+)$, by the Reeh-Schlieder Theorem. One can therefore define an anti-linear unbounded operator $S: \A(S^1_+)\Omega\to \A(S^1_+)\Omega$ such that, for every $x\in \A(S^1_+)$,
\[
S\pi_{0, S_+^1}(x)(\Omega) = \pi_{0, S_+^1}(x^*)(\Omega), 
\]
following Tomita-Takesaki theory. The operator $S$ is preclosed and we continue denoting its closure by $S$. Let $S = \mathrm{J}\circ \Delta^{1/2}$ be the polar decomposition of $S$. Then, the anti-unitary map $\mathrm{J}: H_0\to H_0$ is called the modular conjugation and satisfies $\mathrm{J}\Omega = \Omega$ and $\mathrm{J}^2 = \id_{H_0}$. By the Bisognano-Wichmann Theorem \cite{bgl93}, it also holds that, for every $I\in\Jcal$ and $x\in \A(I)$,
\[
\mathrm{J}\circ \pi_{0, I}(x)\circ \mathrm{J}\in \Hom_{\A(\overline{I})^c}(H_0,H_0)\cong \A(\overline{I}),
\]
where $\overline{I} := \{\overline{z}\,|\, z\in I\}\in \Jcal$ is the image of $I$ under complex conjugation in $S^1$. Hence, we obtain a von Neumann algebra isomorphism
\[
\begin{array}{cccc}
     j: &\A(I) &\to & \A(\overline{I})^{\text{op}} \\
     & x&\mapsto & j(x),
\end{array}
\]
where $j(x)$ is the unique element in $\A(\overline{I})$ such that $\pi_{0, \overline{I}}(j(x)) = \mathrm{J}\circ \pi_{0, I}(x^*)\circ \mathrm{J}$. We think of $j$ as the action on $\A$ of complex conjugation on $S^1$. It is clear that $j$ is functorial with respect to inclusion of intervals. The modular conjugation $\mathrm{J}$ is also compatible with automorphisms of $\A$. Indeed, let $\varphi\in \Aut(\A)$ be an automorphism implemented by a unitary $V_\varphi\in U(H_0)$. Since $V_\varphi\A(S^1_+)V_\varphi^* = \A(S^1_+)$ and $V_\varphi(\Omega) = \Omega$, then $V_\varphi$ commutes with the modular conjugation $\mathrm{J}$, see for example \cite[Thm. 3.2]{MR407615}. Hence, it follows that $j$ and $\varphi$ commute.

Let $\varphi\in \Aut(\A)$ and $(H^\varphi, \pi^H)\in \Rep^\varphi(\A)$. We define a new representation $\overline{(H^\varphi, \pi^H)}:=(\overline{H^\varphi}, \overline{\pi^H})\in \Rep^{\varphi^{-1}}(\A)$. Here, $\overline{H^\varphi}$ denotes the complex conjugate Hilbert space of $H^\varphi$ and $\overline{\pi^H}$ is defined as follows. Let $\tilde{I}\in \Jcal_{\R}$ and $x\in \A_\R(\tilde{I})$. Then, we write
\begin{equation}\label{eq: defInvolutionRepGA}
\overline{\pi^H}_{\tilde{I}}(x)(\overline{\xi}) = \overline{\pi^H_{-\tilde{I}}(j(x)^*)(\xi)}.
\end{equation}
for all $\overline{\xi}\in \overline{H^\varphi}$.
\begin{lemma}\label{lemm: InvolutionCompatibleGrading}
    The formula \eqref{eq: defInvolutionRepGA} defines a $\varphi^{-1}$-twisted representation of $\A$ on $\overline{H^\varphi}$.
\end{lemma}
\begin{proof}
    Fix $\tilde{I}\in \Jcal_\R$. It is clear that $\overline{\pi^H}_{\tilde{I}}: \A_\R(\tilde{I})\to B(\overline{H})$ is a $*$-action. Let $\tilde{J}\in \Jcal_\R$ be another interval such that $\tilde{J}\subset \tilde{I}$. Then, given $x\in \A_{\R}(\tilde{J})$, it holds that
    \[
    \overline{\pi^H}_{\tilde{J}}(x)  = \pi^H_{-\tilde{J}}(j(x)^*) = \pi^H_{-\tilde{I}}(j(x)^*) = \overline{\pi^H}_{\tilde{I}}(x).
    \]
    In order to show the representation is $\varphi^{-1}$-twisted, we compute
    \[
    \overline{\pi^H}_{\tilde{I}+1}(x) = \pi^H_{-\tilde{I}-1}(j(x)^*) = \pi^H_{-\tilde{I}}(\varphi^{-1}(j(x)^*)) = \pi^H_{-\tilde{I}}(j(\varphi^{-1}x)^*) = \overline{\pi^H}_{\tilde{I}}(\varphi^{-1}x),
    \]
   and the claim follows.
\end{proof}

It is clear that, if $K^\varphi\in \Rep^\varphi(\A)$ is a twisted representation and $F: H^\varphi\to K^\varphi$ is a morphism of $\varphi$-twisted representations, the complex conjugate $\overline{F}: \overline{H^\varphi}\to \overline{K^\varphi}$ intertwines the $\varphi^{-1}$-twisted actions $\overline{\pi^H}$ and $\overline{\pi^K}$ of $\A$.

\begin{definition}\label{def: involutionRepA}
    We define the anti-linear functor $\overline{\,\cdot\,}: \Rep^{\Aut(\A)}(\A)\to \Rep^{\Aut(\A)}(\A)$ by sending a twisted representation $(H^\varphi, \pi^H)\in \Rep^\varphi(\A)$ to $(\overline{H^\varphi}, \overline{\pi^H})$, and a morphism of twisted representations $F: H^\varphi\to K^\varphi$ to the complex conjugate morphism $\overline{F}: \overline{H^\varphi}\to \overline{K^\varphi}$.
\end{definition}

Let us provide next the structure morphisms that upgrade $\overline{\,\cdot\,}$ to an involutive $\mathrm{W}^*$-structure on $\Rep^{\Aut(\A)}(\A)$. By definition of $j$, the following lemma holds.

\begin{lemma}\label{lemm: InvolutiveDataUnit}
    The modular conjugation $\mathrm{J}$ provides an isomorphism of $\A$-representations
    \[
    \begin{array}{cccc}
    i:&(H_0, \pi_0)&\xrightarrow{\cong} &\overline{(H_0,\pi_0)}\\&
    \xi &\mapsto & \overline{\mathrm{J}\xi}.
    \end{array}
    \]
\end{lemma}

The following result is also clear from the definition of $\overline{\,\cdot\,}$.

\begin{lemma}\label{lemm: InvolutiveDataSquare}
    Let $(H^\varphi, \pi^H)\in \Rep^\varphi(\A)$ be a $\varphi$-twisted representation. Then, the unitary $H^\varphi\to \overline{\overline{H^\varphi}}$ given by $\xi\mapsto \overline{\overline{\xi}}$ provides a unitary isomorphism of twisted $\A$-representations
    \[
   \phi_{H}: (H^\varphi,\pi^H)\xrightarrow{\cong}\overline{\overline{(H^\varphi, \pi^H)}}.
    \]
\end{lemma}

In order to provide the anti-tensor structure for $\overline{\,\cdot\,}$, we need the following computations. Let $\tilde{I}\in \Jcal_\R$ and $\xi\in H_-^\varphi(\tilde{I})$. Then consider the unitary
\begin{equation}\label{eq: Zbar}
Z^+(\overline{\xi}, -\tilde{I}):H_0\xrightarrow{i} \overline{H_0}\xrightarrow{\overline{Z^-(\xi, \tilde{I})}} \overline{H^\varphi}.
\end{equation}
Note that $-(-\tilde{I})^{c+} = -\big((-\partial_+\tilde{I}, -\partial_{-}\tilde{I})^{c+}\big) = -(-\partial_{-}\tilde{I}, -\partial_+\tilde{I}+1) = (\partial_{+}\tilde{I}-1, \partial_-\tilde{I}) = \tilde{I}^{c-}$. Hence, given $x\in \A_\R\big((-\tilde{I})^{c+}\big)$ and $\eta\in H_0$, it holds that
\begin{align*}
    \overline{\pi^H_{(-\tilde{I})^{c+}}}(x)\circ \overline{Z^-(\xi, \tilde{I})}(\overline{\mathrm{J}\eta}) & = \overline{\pi^H_{ -(-\tilde{I})^{c+}}(j(x)^*)({Z^-(\xi,\tilde{I})(\mathrm{J}\eta)})}\\&= \overline{\pi^H_{\tilde{I}^{c-}}(j(x)^*)\circ Z^-(\xi, \tilde{I})\circ \mathrm{J}(\eta)}\\&=\overline{Z^-(\xi, \tilde{I})\circ \pi_{0}(j(x)^*)\circ\mathrm{J}(\eta)}\\ &= \overline{Z^-(\xi, \tilde{I})\circ \mathrm{J}\circ \pi_0(x)(\eta)}.
\end{align*}
Hence, it follows that $Z^+(\overline{\xi}, -\tilde{I})$ is $\A_{\R}\big((-\tilde{I})^{c+}\big)$-equivariant. In addition, we have $Z^+(\overline{\xi}, -\tilde{I})(\Omega) = \overline{\xi}$, and we obtain an injective map $ H^\varphi_-(\tilde{I})\hookrightarrow \overline{H^\varphi}_+(-\tilde{I})$ given by $\xi\mapsto \overline{\xi}$. We can analogously produce an injective map the other way, and hence we find a canonical isomorphism
\[
H^\varphi_-(\tilde{I})\cong  \overline{H^\varphi}_+(-\tilde{I}).
\]
We similarly obtain an isomorphism $H^\varphi_+(\tilde{I})\cong  \overline{H^\varphi}_-(-\tilde{I})$.

\begin{proposition}\label{prop: 2CellInvolution}
    Let $\varphi,\mu\in \Aut(\A)$ be automorphisms and $H^\varphi\in \Rep^\varphi(\A)$ and $K^\mu\in\Rep^\mu(\A)$ be twisted representations. Given $\tilde{I}\in \Jcal_\R$, the map $\overline{H^\varphi}_+(\tilde{I})\otimes \overline{K^\mu}\to \overline{K^\mu\otimes H^\varphi_-(-\tilde{I})}$ given by $\overline{\xi}\otimes\overline{\eta}\mapsto \overline{\eta\otimes\xi}$ defines a unitary
    \[
    \nu_{H,K}^{\tilde{I}}: \overline{H^\varphi}_+(\tilde{I})\boxtimes \overline{K^\mu}\to \overline{K^\mu\boxtimes H^\varphi_-(-\tilde{I})}.
    \]
     In addition, the map above is a unitary equivalence of $(\varphi^{-1}\circ\mu^{-1})$-twisted representations of $\A$. Analogously, there is an equivalence of $(\varphi^{-1}\circ\mu^{-1})$-twisted representations of $\A$
    \[
   ^{\text{op}}\nu_{H,K}^{\tilde{I}}: \overline{H^\varphi}\boxtimes \overline{K^\mu}_-(\tilde{I})\to \overline{K^\mu_+(-\tilde{I})\boxtimes H^\varphi}.
    \]
\end{proposition}
\begin{proof}
We only discuss the first equivalence, the other one follows from analogous arguments. We shall first show that the map $\overline{H^\varphi}_+(\tilde{I})\otimes \overline{K^\mu}\to \overline{K^\mu\otimes H^\varphi_-(-\tilde{I})}$ given by $\overline{\xi_1}\otimes\overline{\eta_1}\mapsto \overline{\eta_1\otimes\xi_1}$ is an isometry. Let $\xi,\xi'\in H_-^\varphi(-\tilde{I})$ and $\eta,\eta'\in K^\mu$. Then, $\overline{\xi}, \overline{\xi'}\in \overline{H^\varphi}_+(\tilde{I})$ and we can compute
\begin{align*}
    \langle \overline{\xi}\otimes\overline{\eta}, \overline{\xi'}\otimes\overline{\eta'}\rangle & = \langle \overline{\pi^K}_{\tilde{I}}(Z^+(\overline{\xi'}, \tilde{I})^*Z^+(\overline{\xi}, \tilde{I}))(\overline{\eta}),\overline{\eta'}\rangle\\& = \langle \overline{\pi^K_{-\tilde{I}}\big(Z^-({\xi'}, -\tilde{I})^*Z^-({\xi}, -\tilde{I})\big)(\eta)},\overline{\eta'}\rangle \\ & = \langle \eta',\pi_{-\tilde{I}}^K\big({Z^-({\xi'}, -\tilde{I})}^*{Z^-({\xi}, -\tilde{I})}\big)(\eta)\rangle \\ & = \langle \pi_{-\tilde{I}}^K\big({Z^-({\xi}, -\tilde{I})}^*{Z^-({\xi'}, -\tilde{I})}\big)(\eta'), \eta\rangle \\ & =\langle \eta'\otimes \xi',\eta\otimes\xi\rangle\\ & = \langle \overline{\eta\otimes\xi}, \overline{\eta'\otimes\xi'}\rangle, 
\end{align*}
as needed. Therefore, since the map $\overline{H^\varphi}_+(\tilde{I})\otimes \overline{K^\mu}\to \overline{K^\mu\otimes H^\varphi_-(-\tilde{I})}$ given by $\overline{\xi}\otimes\overline{\eta}\mapsto \overline{\eta\otimes\xi}$ maps a dense subset of $\overline{H^\varphi}_+(\tilde{I})\boxtimes \overline{K^\mu}$ to a dense subset of $\overline{K^\mu\boxtimes H^\varphi_-(-\tilde{I})}$, it produces a unitary $\overline{H^\varphi}_+(\tilde{I})\boxtimes \overline{K^\mu}\to \overline{K^\mu\boxtimes H^\varphi_-(-\tilde{I})}$. We next argue the compatibility with the $\A$-action. Let $\tilde{L}\in \Jcal_\R$ and $x\in \A_\R(\tilde{L})$. Let $\alpha$ be a path in $\R$ from $\tilde{I}$ to $\tilde{L}$. Let us assume that Lemma \ref{Lemm: TechnicalLemmaInvolutionTensorator} below holds and denote by $\pi^{\overline{H}\boxtimes \overline{K}}$ the action of $\A$ on the domain of $\nu_{H, K}^{\tilde{I}}$. Then, we can compute
\begin{align*}
   \nu_{H, K}^{\tilde{I}}\circ \pi^{\overline{H}\boxtimes \overline{K}}_{\tilde{L}}(x) & = \nu_{H, K}^{\tilde{I}}\circ(\alpha^{\bullet})^{-1}\circ (\overline{\pi^{H}}_{\tilde{L}}(x)\boxtimes \id)\circ \alpha^\bullet\\& = \overline{((-\alpha)^{\bullet})^{-1}}\circ \nu_{H, K}^{\tilde{I}}\circ (\overline{\pi^{H}}_{\tilde{L}}(x)\boxtimes \id)\circ \alpha^\bullet \\ & =  \overline{((-\alpha)^{\bullet})^{-1}}\circ\overline{(\id\boxtimes\pi_{-\tilde{L}}^{H}(j(x)^*))}\circ \nu_{H, K}^{\tilde{I}}\circ \alpha^\bullet \\ & = \overline{((-\alpha)^{\bullet})^{-1}}\circ\overline{(\id\boxtimes\pi_{-\tilde{L}}^{H}(j(x)^*))}\circ \overline{(-\alpha)^\bullet}\circ \nu_{H, K}^{\tilde{I}}\\ & = \overline{\pi^{K\boxtimes H}}_{\tilde{L}}(x)\circ \nu_{H, K}^{\tilde{I}}, 
\end{align*}
where we have used the first statement in Lemma \ref{Lemm: TechnicalLemmaInvolutionTensorator} twice and the second statement~once.
\end{proof}
\begin{lemma}\label{Lemm: TechnicalLemmaInvolutionTensorator}
    Let $\tilde{L}\in \Jcal_\R$. Then, the following diagrams commute,
    \begin{enumerate}
        \item for every $\tilde{I}\in\Jcal_\R$ and every path $\alpha$ in $\R$ from $\tilde{I}$ to $\tilde{L}$,
\[\begin{tikzcd}
	{\overline{H^\varphi}_+(\tilde{I})\boxtimes \overline{K^\mu}} && {\overline{K^\mu\boxtimes H^\varphi_-(-\tilde{I})}} \\
	{\overline{H^\varphi}_+(\tilde{L})\boxtimes \overline{K^\mu}} && {\overline{K^\mu\boxtimes H^\varphi_-(-\tilde{L})}};
	\arrow["{\nu^{\tilde{I}}_{H, K}}", from=1-1, to=1-3]
	\arrow["{\alpha^\bullet}"', from=1-1, to=2-1]
	\arrow["{\overline{(-\alpha)^\bullet}}", from=1-3, to=2-3]
	\arrow["{\nu^{\tilde{L}}_{H, K}}"', from=2-1, to=2-3]
\end{tikzcd}\]
\item for all $x\in \A_\R(\tilde{L})$,
\[\begin{tikzcd}
	{\overline{H^\varphi}_+(\tilde{L})\boxtimes \overline{K^\mu}} && {\overline{K^\mu\boxtimes H^\varphi_-(-\tilde{L})}} \\
	{\overline{H^\varphi}_+(\tilde{L})\boxtimes \overline{K^\mu}} && {\overline{K^\mu\boxtimes H^\varphi_-(-\tilde{L})}}.
	\arrow["{\nu^{\tilde{L}}_{H, K}}", from=1-1, to=1-3]
	\arrow["{\overline{\pi^H}_{\tilde{L}}(x)\boxtimes \id}"', from=1-1, to=2-1]
	\arrow["{\overline{\id\boxtimes \pi_{-\tilde{L}}^H(j(x)^*)}}", from=1-3, to=2-3]
	\arrow["{\nu^{\tilde{L}}_{H, K}}"', from=2-1, to=2-3]
\end{tikzcd}\]

    \end{enumerate}
\end{lemma}
\begin{proof}
    The commutativity of the first diagram is clear from the definition of path continuations. To argue the commutativity of the second diagram, let $\overline{\xi}\otimes\overline{\eta}\in {\overline{H^\varphi}_+(\tilde{L})\otimes \overline{K^\mu}}$. Then, we have
    \begin{align*}
        \nu_{H, K}^{\tilde{L}}\circ (\overline{\pi^H}_{\tilde{L}}(x)\boxtimes \id)(\overline{\xi}\otimes\overline{\eta}) & = \nu_{H, K}^{\tilde{L}}\big(\overline{\pi^H}_{\tilde{L}}(x)(\xi)\otimes \overline{\eta}\big)\\ &= \nu_{H, K}^{\tilde{L}}\big(\overline{\pi^H_{-\tilde{L}}(j(x)^*)(\xi)}\otimes\overline{\eta}\big)\\ & = \overline{\eta\otimes \pi^H_{-\tilde{L}}(j(x)^*)(\xi)}\\ & = \overline{\id\boxtimes \pi^H_{-\tilde{L}}(j(x)^*)}\circ \nu_{H, K}^{\tilde{L}}(\overline{\xi}\otimes\overline{\eta}),
    \end{align*}
    as needed.
\end{proof}

We next argue the compatibility of the involution with the action of $\Aut(\A)$. 

\begin{lemma}
    Fix $\varphi,\mu\in \Aut(\A)$ automorphisms and let $(H^\varphi,\pi^H)\in \Rep^\varphi(\A)$ be a twisted representation. Then, we have an equality of $(\mu\circ\varphi^{-1}\circ \mu^{-1})$-twisted $\A$-representations
    \[
T_\mu\big(\overline{(H^\varphi,\pi^H)}\big) = \overline{T_\mu((H^\varphi,\pi^H))}.
    \]
\end{lemma}
\begin{proof}
    The underlying Hilbert space for both representations is $\overline{H^\varphi}.$ Let $\tilde{I}\in \Jcal_\R$ and $x\in \A_\R(\tilde{I})$. Then, $x$ acts on $ T_\mu\big(\overline{(H^\varphi,\pi^H)}\big)$ by
    \[
    \overline{\pi^H}_{\tilde{I}}(\mu^{-1}x) = \pi^H_{-\tilde{I}}(j(\mu^{-1}x)^*) = \pi^H_{-\tilde{I}}(\mu^{-1}\circ j(x)^*),
    \]
    and it acts on $\overline{T_\mu((H^\varphi,\pi^H))}$ by
    \[
    \overline{\pi^{\mu\ast H}}_{\tilde{I}}(x) = \pi^{\mu\ast H}_{-\tilde{I}}(j(x)^*) = \pi_{-\tilde{I}}^H(\mu^{-1}\circ j(x)^*).
    \]
    Hence, the claim follows.
\end{proof}

The following computations will be needed to prove the different required compatibilities between the structure morphisms we have introduced for the involution on $\Rep^{\Aut(\A)}(\A)$ and the $\Aut(\A)$-crossed braiding.

\begin{proposition}\label{prop: NecessaryEqualitiesForInvolutiveFunctor}
    Let $\varphi,\mu,\nu\in \Aut(\A)$ be automorphisms of $\A$. Let $H^\varphi\in \Rep^\varphi(\A)$ and $K^\mu\in \Rep^\mu(\A)$ be twisted representations of $\A$. Then, we have the following equalities of morphisms:
\begin{enumerate}
    \item $T_\nu(\phi_{H}) = \phi_{T_\nu(H^\varphi)}$ as morphisms from $T_\nu(H^\varphi)$ to $T_\nu(\overline{H^\varphi})= \overline{T_\nu(H^\varphi)}$,
    \item $\mathrm{J}\circ V_\nu = V_\nu\circ T_\nu(\mathrm{J})$ as morphisms from $T_\nu(H_0)$ to $H_0$,
    \item the two legs of the following diagram
\[\begin{tikzcd}
	{\overline{T_\nu(H^\varphi)}_+(\tilde{I})\boxtimes \overline{T_\nu(K^\mu)}} &&& {\overline{T_\nu(K^\mu)\boxtimes T_\nu(H^\varphi)_-(-\tilde{I})}} \\
	{T_\nu(\overline{H^\varphi})_+(\tilde{I})\boxtimes T_\nu(\overline{K^\mu})} &&& {\overline{T_\nu(K^\mu\boxtimes H^\varphi_-(-\tilde{I}))}} \\
	{T_\nu(\overline{H^\varphi}_+(\tilde{I})\boxtimes \overline{K^\mu})} &&& {T_\nu(\overline{K^\mu\boxtimes H^\varphi_-(-\tilde{I})})}
	\arrow["{\nu^{\tilde{I}}_{T_\nu(H^\varphi), T_\nu(K^\mu)}}", from=1-1, to=1-4]
	\arrow["{=}"', from=1-1, to=2-1]
	\arrow["\cong", from=1-4, to=2-4]
	\arrow["\cong"', from=2-1, to=3-1]
	\arrow["{=}", from=2-4, to=3-4]
	\arrow["{T_\nu(\nu_{H, K}^{\tilde{I}})}"', from=3-1, to=3-4]
\end{tikzcd}\]
where the two vertical non-equality arrows are the morphisms in Proposition \ref{prop: ActionData}.
\end{enumerate}
\end{proposition}
\begin{proof}
    The first equality is obvious given that $T_\nu$ acts trivially on morphisms. The second follows from the compatibility of $\mathrm{J}$ with conformal net automorphisms and the fact that $T_\nu$ is trivial on morphisms. For the last statement, let $\overline{\xi}\otimes\overline{\eta}\in {\overline{T_\nu(H^\varphi)}_+(\tilde{I})\otimes \overline{T_\nu(K^\mu)}}$. Then, both legs applied to $\overline{\xi}\otimes\overline{\eta}$ produce $\overline{\eta\otimes\xi}$, and the claim follows.
\end{proof}

We finally discuss the compatibility of the involution with the conformal structure of twisted representations of $\A$. This is necessary to argue the compatibility of the crossed balance and the involutive structures, in the sense of Definition \ref{def: Involutive-X-Balance}. Given $(\tilde{f}, f,V)\in \Diff_{\A}^+(S^1)$, we write $\overline{(\tilde{f}, f,V)}\in \Diff_{\A}^+(S^1)$ for the element
\[
\overline{(\tilde{f}, f,V)}:=\big( x\mapsto -\tilde{f}(-x), z\mapsto \overline{f(\overline{z})}, \mathrm{J}\circ V\circ \mathrm{J}\big).
\]
It is clear that $z\mapsto \overline{f(\overline{z})}$ is orientation preserving and that $x\mapsto -\tilde{f}(-x)$ covers it. Recall that given a unitary $V\in U(H_0)$, we write $[V]\in \mathcal{P}U(H_0)$ for its class in $\mathcal{P}U(H_0)$. Then, $[\mathrm{J}\circ V\circ \mathrm{J}] = \mathrm{J}\circ [V]\circ \mathrm{J} = \mathrm{J}\circ U(f)\circ \mathrm{J} = U(z\mapsto \overline{f(\overline{z})})$ by the Bisognano-Wichmann Theorem \cite{bgl93} and the uniqueness of the projective $\Diff^+(S^1)$-action extending the action of $\Mob$ on $H_0$ \cite[Thm. 6.1.9]{weiner2007conformalcovariancerelatedproperties}.  Hence $\overline{(\tilde{f}, f,V)}$ is a well-defined element of $\Diff_{\A}^+(S^1)$. 
\begin{proposition}\label{prop: ConformalActionOnConjugate}
    Fix an automorphism $\varphi\in \Aut(\A)$ and let $H^\varphi\in \Rep^\varphi(\A)$ be a $\varphi$-twisted representation of $\A$. Let us denote by $U^H$ the unitary action of $\Diff^+_{\A}(S^1)$ on $H^\varphi$. Then, the action $U^{\overline{H}}$ of $\Diff^+_{\A}(S^1)$ on $\overline{H^\varphi}$ is given by
    \[
    U^{\overline{H}}(\tilde{f}, f, V) = \overline{U^H(\overline{(\tilde{f}, {f}, V)})}
    \]
    for all $(\tilde{f}, f, V)\in \Diff_{\A}^+(S^1)$.
\end{proposition}
\begin{proof}
    Let $I\in\Jcal$ and assume that $(\tilde{f}, f,V)\in \Diff_{\A}^+(I)$. Then, by definition
    \[
    U^{\overline{H}}(\tilde{f}, f,V) = \overline{\pi^H}_I(U(\tilde{f}, f, V))=  \overline{\pi_{\overline{I}}^H(j(U(\tilde{f}, f,V))^*)}.
    \]
    But $\pi_{0,\overline{I}}(j(U(\tilde{f}, f,V))^*) = \mathrm{J}\circ U(\tilde{f}, f,V)\circ \mathrm{J} = U(\overline{(\tilde{f}, f, V)})$. By uniqueness of the action in Theorem \ref{Thm: TwistedRepIsConformal} and the fact that $U^{\overline{H}}$ defines an action of $\Diff_\A^+(S^1)$ on $\overline{H^\varphi}$, the claim follows.
\end{proof}

We can now endow $\Rep^{\Aut(\A)}(\A)$ with the structure of an $\Aut(\A)$-crossed balanced involutive $\mathrm{W}^*$-tensor category. Let $\overline{\,\cdot\,}: \Rep^{\Aut(\A)}(\A)\to \Rep^{\Aut(\A)}(\A)$ be the involution in Definition \ref{def: involutionRepA}. Then, we have a unitary isomorphism
\[
\nu_{H, K}: \overline{H^\varphi}\boxtimes \overline{K^\mu} = \overline{H^\varphi}_+(\widetilde{S^1_-})\boxtimes \overline{K^\mu}\xrightarrow{\nu_{H, K}^{\widetilde{S^1_-}}}\overline{K^\mu\boxtimes H_-^\varphi(-\widetilde{S_-^1})} = \overline{K^\mu\boxtimes H_-^\varphi(\widetilde{S_+^1})} = \overline{K^\mu\boxtimes H^\varphi}.
\]

\begin{lemma}\label{lemm : IndeedGetInvolutive}
The functor
\[
\overline{\,\cdot\, }: \Rep^{\Aut(\A)}(\A)\to \Rep^{\Aut(\A)}(\A), 
\]
in Definition \ref{def: involutionRepA}, with the isomorphisms $i: H_0\cong\overline{H_0}$ and $\phi_{H}: H^\varphi\cong\overline{\overline{H^\varphi}},$ and~the family of isomorphisms $\nu$, endow $\Rep^{\Aut(A)}(\A)$ with the structure of an involutive $\mathrm{W}^*$-tensor~category. 
\end{lemma}
\begin{proof}
    The compatibility between the associator $\alpha$ and the unitaries $\nu$ is straightforward to check, as is the fact that $\overline{\phi_{H}} = \phi_{\overline{H}}$. Since $\mathrm{J^2} = \id_{H_0}$, it also holds that $\overline{i}\circ i = \phi_{H_0}$. Finally, we need to show the commutativity of 
\[\begin{tikzcd}
	{H^\varphi\boxtimes K^\mu} && {\overline{\overline{H^\varphi\boxtimes K^\mu}}} \\
	{\overline{\overline{H^\varphi}}\boxtimes \overline{\overline{K^\mu}}} && {\overline{\overline{K^\mu}\boxtimes \overline{H^\varphi}}}.
	\arrow["{\phi_{H\boxtimes K}}", from=1-1, to=1-3]
	\arrow["{\phi_{H}\boxtimes \phi_{K}}"', from=1-1, to=2-1]
	\arrow["{\nu_{\overline{H}, \overline{K}}}"', from=2-1, to=2-3]
	\arrow["{\overline{\nu_{K, H}}}"', from=2-3, to=1-3]
\end{tikzcd}\]
Let $\xi\in H^\varphi_+(\widetilde{S^1_-})$ and $\eta\in K^\mu$. Then, the long leg of the diagram above applied to $\xi\otimes\eta$ reads
\[
\xi\otimes\eta\mapsto \overline{\overline{\xi}}\otimes\overline{\overline{\eta}}\mapsto \overline{\overline{\eta}\otimes\overline{\xi}}\mapsto \overline{\overline{\xi\otimes\eta}},
\]
which is also the image of $\xi\otimes\eta$ under the top arrow, as needed.
\end{proof}

We next argue that the $\Aut(\A)$-crossed balanced structure of $\Rep^{\Aut(\A)}(\A)$ in Theorem \ref{Thm: RepAutAIsCrossedBalanced} is compatible with the involutive $\mathrm{W}^*$-structure in the sense that they provide an $\Aut(\A)$-crossed balanced involutive $\mathrm{W}^*$-tensor structure.

\begin{theorem}\label{thm: Involutive-AutX-X-Balanced}
    Let $\A$ be a conformal net. The $\Aut(\A)$-crossed balanced $\mathrm{W}^*$-tensor category $\Rep^{\Aut(\A)}(\A)$ is naturally an $\Aut(\A)$-crossed balanced involutive $\mathrm{W}^*$-tensor category.
\end{theorem}
\begin{proof}
    We use the involutive structure on $\Rep^{\Aut(\A)}$ described in Lemma \ref{lemm : IndeedGetInvolutive}. By Lemma \ref{lemm: InvolutionCompatibleGrading}, the involution functor is compatible with the grading by $\Aut(\A)$, and by Proposition \ref{prop: NecessaryEqualitiesForInvolutiveFunctor}, the functor $T_\nu$ is involutive for the structure data
    \[
    T_\nu(\overline{\,\cdot\,}) = \overline{T_\nu(\cdot)}
    \]
    for all $\nu\in \Aut(\A)$. Since all the structure data appearing in the fourth axiom in Definition \ref{def: G-X-braidedInvolutive} are identities in our case, that equality also follows. Le us show the third axiom in Definition \ref{def: G-X-braidedInvolutive}. Fix $\varphi,\mu\in \Aut(\A)$ automorphisms and twisted representations $H^\varphi\in\Rep^\varphi(\A)$ and $K^\mu\in \Rep^\mu(\A)$. In Diagram \eqref{eq: CompatibilityBraidingInvolution}, take $g = \varphi$, $X = H^\varphi$ and $Y = K^\mu$. Then, to a vector $\overline{\xi\otimes\eta}\in \overline{H^\varphi\boxtimes K^\mu}$, the top-right leg of the diagram reads
    \[
    \overline{\xi\otimes \eta}\mapsto \overline{\Gamma_g\eta\otimes\xi}\mapsto \overline{\xi}\otimes \overline{\Gamma_g\eta}\mapsto \overline{\xi}\otimes \Gamma_{g}\overline{\eta}\mapsto \Gamma_{g^{-1}}\Gamma_{g}\overline{\eta}\otimes\overline{\xi}.
    \]
    The left-bottom leg applied to the same vector reads
    \[
    \overline{\xi\otimes\eta}\mapsto \overline{\eta}\otimes \overline{\xi} = \Gamma_e\overline{\eta}\otimes \overline{\xi}\mapsto \Gamma_{g^{-1}}\Gamma_g\overline{\eta}\otimes\overline{\xi},
    \]
as needed.
    
    It is only left to show the ribbon condition \eqref{eq: RibbonCondition}. Fix $\varphi\in\Aut(\A)$ and let $H^\varphi\in \Rep^\varphi(\A)$ be a twisted representation. Since the vertical arrows in \eqref{eq: RibbonCondition} are identities in our setting, it is enough to show that $\theta_{T_\varphi(\overline{H^\varphi})} = \overline{\theta^*_H}$ as unitaries $\overline{H^\varphi}\to \overline{H^\varphi}$. By the fact that $\theta$ is an $\Aut(\A)$-crossed balance, $\theta_{T_\varphi(\overline{H^\varphi})} = T_\varphi(\theta_{\overline{H}})$, see Definition \ref{def: CrossedCategoricalTwist} $\mathrm{(i)}$. Since $T_\varphi$ acts as the identity on morphisms, $\theta_{T_\varphi(\overline{H^\varphi})} = \theta_{\overline{H}}$, and we only have to show that $\overline{\theta_{H}^*} = \theta_{\overline{H}}$. Now, by definition, $\theta_{\overline{H}}=e^{-2\pi i L_0}:\overline{H^\varphi}\to \overline{H^\varphi}$, which, by Proposition \ref{prop: ConformalActionOnConjugate} is exactly $\overline{e^{2\pi i L_0}}=\overline{\theta_H^*}$, as needed. 
\end{proof}

\begin{corollary}
    Let $\A$ be a conformal net being acted on faithfully by a discrete group $G$. Then, the $G$-crossed balanced $\mathrm{W}^*$-tensor category $\Rep^{G}(\A)$ is naturally a $G$-crossed balanced involutive $\mathrm{W}^*$-tensor category.  
\end{corollary}

\begin{corollary}
    Let $\A$ be a conformal net. Then, the balanced $\mathrm{W}^*$-tensor category $\Rep(\A)$ is naturally a balanced involutive $\mathrm{W}^*$-tensor category. In particular, given a representation $H\in \Rep(\A)$, it holds that $\theta_{\overline{H}} = \overline{\theta_{H}^*}$, where $\theta$ denotes the balance. 
\end{corollary}

\bibliographystyle{halpha-abbrv}
\bibliography{Gcrossed}
\end{document}